\documentclass[11pt]{amsart}
\usepackage{xfrac}
\usepackage{ytableau}
\usepackage{young}
\usepackage{youngtab}
\usepackage{lmodern}
\usepackage{enumitem}
\usepackage{graphicx} 
\usepackage[T1]{fontenc}
\usepackage{amsthm}
\usepackage{tikz-cd}
\usepackage[maxbibnames=99,
backend=biber,
style=alphabetic,
doi=false,
url=false,giveninits=true
]{biblatex}
\usepackage{fullpage}
\usepackage{subfig}

\usepackage{mathrsfs}
\usepackage{bbm,amsmath,amssymb,amsfonts,graphicx,color,subfig,mathtools, verbatim}
   \usepackage[bookmarksopen=false,pdftex=true,breaklinks=true,%
      plainpages=false,%
      hyperindex=true,pdfstartview=FitH,colorlinks=true,%
      pdfpagelabels=true,colorlinks=true,linkcolor=blue,%
      citecolor=red,urlcolor=blue,hypertexnames=false%
      ]%
   {hyperref}

\newcommand{\balpha}{\boldsymbol{\alpha}}
   
\renewcommand{\vartheta}{\theta}

\newcommand{\B}{\mathcal{B}}

\newcommand{\KK}{\mathcal{K}}
\newcommand{\LL}{\mathcal{L}}
\newcommand{\N}{\mathbb{N}}
\newcommand{\Z}{\mathbb{Z}}
\renewcommand{\P}{\mathcal{P}}

\newcommand{\R}{\mathbb{R}}

\newcommand{\Q}{\mathcal{Q}}
\renewcommand{\S}{\mathcal{S}}
\newcommand{\p}{\mathfrak{p}}
\renewcommand{\leq}{\leqslant}

\newcommand{\la}{\lambda}
\renewcommand{\a}{\alpha}

\usepackage{faktor}

\usepackage{MnSymbol} 
\numberwithin{equation}{section}
\newtheorem{theorem}{Theorem}[section]
\newtheorem{lemma}[theorem]{Lemma}
\newtheorem{remark}[theorem]{Remark}

\newtheorem{corollary}[theorem]{Corollary}
\newtheorem{proposition}[theorem]{Proposition}
\theoremstyle{definition}
\newtheorem{definition}[theorem]{Definition}
\newtheorem{example}[theorem]{Example}
\newtheorem{conjecture}[theorem]{Conjecture}
\newtheorem{question}[theorem]{Question}
  \newtheoremstyle{TheoremNum}
        {\topsep}{\topsep}              
        {\itshape}                      
        {}                              
        {\bfseries}                     
        {.}                             
        { }                             
        {\thmname{#1}\thmnote{ \bfseries #3}}
    \theoremstyle{TheoremNum}
    \newtheorem{thmn}{Theorem}
  
    \newtheorem{propn}{Proposition}

\DeclareMathOperator{\spe}{sp}
\DeclareMathOperator{\spn}{span}
\DeclareMathOperator{\Tr}{Tr}

\DeclareMathOperator{\trop}{trop}
\DeclareMathOperator{\inti}{int}
\DeclareMathOperator{\clo}{cl}

\DeclareMathOperator{\tcone}{tcone}
\DeclareMathOperator{\conv}{conv}
\DeclareMathOperator{\cone}{cone}

\DeclareMathOperator{\M}{\mathcal{M}}
\DeclareMathOperator{\rank}{rk}

\DeclareMathOperator{\dhu}{dh}

\def\PSD{\operatorname{PSD}}
\def\sym{\operatorname{sym}}
\def\m{\mathfrak{m}}
\def\v{\mathbf{v}}

\DeclareFontFamily{U}{mathx}{\hyphenchar\font45}
\DeclareFontShape{U}{mathx}{m}{n}{
      <5> <6> <7> <8> <9> <10>
      <10.95> <12> <14.4> <17.28> <20.74> <24.88>
      mathx10
      }{}
\DeclareSymbolFont{mathx}{U}{mathx}{m}{n}
\DeclareFontSubstitution{U}{mathx}{m}{n}
\DeclareMathAccent{\widecheck}{0}{mathx}{"71}

\ytableausetup{centertableaux}

\usepackage{listings}
\lstdefinelanguage{Sage}[]{Python}
{morekeywords={False,sage,True},sensitive=true}
\definecolor{dblackcolor}{rgb}{0.0,0.0,0.0}
\definecolor{dbluecolor}{rgb}{0.01,0.02,0.7}
\definecolor{dgreencolor}{rgb}{0.2,0.4,0.0}
\definecolor{dgraycolor}{rgb}{0.30,0.3,0.30}

\begin{document}
\title[]
{ 
Symmetric nonnegative functions, the tropical Vandermonde cell and superdominance of power sums
}

\author{Jose Acevedo}
\address{Escuela de Matem\'aticas, Universidad Industrial de Santander, Bucaramanga, Colombia}
\email{jgacehab@correo.uis.edu.co}

\author{Grigoriy Blekherman}
\address{School of Mathematics, Georgia Institute of Technology, 686 Cherry Street Atlanta, GA 30332, USA}
\email{greg@math.gatech.edu}

\author{Sebastian Debus}
\address{Fakultät für Mathematik, Technische Universität Chemnitz, 09107 Chemnitz, Germany}
\email{sebastian.debus@math.tu-chemnitz.de}

\author{Cordian Riener}
\address{Department of Mathematics and Statistics, UiT - the Arctic University of Norway, 9037 Troms\o, Norway}
\email{cordian.riener@uit.no}


\begin{abstract}
We study nonnegative and sums of squares symmetric (and even symmetric) functions of fixed degree. We can think of these as limit cones of symmetric nonnegative polynomials and symmetric sums of squares of fixed degree as the number of variables goes to infinity. We compare these cones, including finding explicit examples of nonnegative polynomials which are not sums of squares for any sufficiently large number of variables, and compute the tropicalizations of their dual cones in the even symmetric case. We find that the tropicalization of the dual cones is naturally understood in terms of the overlooked \emph{superdominance order} on partitions. The power sum symmetric functions obey this same partial order (analogously to how term-normalized power sums obey the dominance order \cite{cuttler2011inequalities}). 
\end{abstract}
\thanks{The first and second author were partially supported by NSF Grant DMS-1901950. The third and fourth author were partially supported by European Union’s Horizon 2020 research and innovation programme under the
Marie Skłodowska-Curie grant agreement 813211 (POEMA). The third author was also partially supported by DFG grant 314838170,
GRK 2297 MathCoRe,  and the fourth author by the Tromsø Research foundation grant agreement
17matteCR}
\maketitle
\markboth{J.~Acevedo, G.~Blekherman, S.~Debus, and C.~Riener}{Symmetric forms}

\section{Introduction}

The relationship between nonnegative polynomials and sums of squares plays an important role in real algebraic geometry, and symmetric polynomials form a distinguished family of examples. The study of inequalities in real symmetric polynomials goes back to Newton \cite{hardy1952inequalities,newton1732arithmetica}. We are interested in inequalities in symmetric polynomials which hold regardless of the number of variables; we can view such inequalities as inequalities between \emph{symmetric functions} (see e.g. \cite[Chapter I.2]{macdonald1998symmetric}). We can also regard symmetric functions as natural limits of symmetric polynomials as the number of variables increases. Therefore symmetric functions can be seen as functions on the {\em image at infinity} of the so-called {\em Vandermonde map}, which has been studied extensively by Arnold, Givental, Kostov, and Ursell in finitely many variables \cite{arnol1986hyperbolic,givental1987moments,kostov1989geometric,kostov1999hyperbolicity,kostov2004very,ursell1959inequalities} and at infinity \cite{kostov2004very,kostov2007stably,acevedo2023wonderful}. 
\smallskip

Inequalities in symmetric functions can be naturally expressed in the power sum basis. Examining such inequalities leads us to the \emph{superdominance order} on partitions, and we find extensions of inequalities in $\ell^p$-norms for integer exponents and Lyapunov's inequality (\cite[Sec.~ 2.9 and 2.10]{hardy1952inequalities}). We find explicit uniform examples of power sum inequalities valid for any number of variables, which do not have a sum of squares certificate, e.g. the quartic $ 4p_1^4-5p_2p_1^2-\frac{139}{20}p_3p_1+4p_2^2+4p_4$ and the decic $ \frac{1}{18}p_2^5+3 p_8p_2 + 6 p_6p_4-3p_6p_2^2$. Moreover, we prove that for symmetric functions, the cones of symmetric (and even symmetric) nonnegative forms are not semialgebraic for any degree $2d\ge4$, (respectively $2d\ge6$), while the cones of symmetric, and even symmetric sums of squares are semialgebraic, hence proving their difference (Theorem~\ref{thm:Limitcones not equal}).
We then study the tropicalizations of the dual cones of the limit cones, and discover hidden combinatorial structure, also naturally expressed in terms of the superdominance order.
\smallskip

Recently \textit{tropicalization} has been applied beyond classical algebraic geometry in extremal combinatorics \cite{blekherman2022tropicalization,blekherman2022path} and real algebraic geometry \cite{blekherman2022moments,acevedo2024power}. Using tropicalization we study the dual cones to nonnegative and sums of squares homogeneous even symmetric functions. We show that the tropicalization of the \emph{Vandermonde cell} (the limit of the image of $\R^n$ under the even power sum map) is a rational polyhedral cone, and give an explicit list of facet-defining inequalities (Theorem~\ref{thm:trop(Nd)}). This allows us to compute the tropicalization of the dual cone to even symmetric nonnegative functions (Proposition~\ref{prop:trop dual to psd}) via \emph{tropical convexity} \cite{develin2004tropical} similar to \cite{blekherman2022moments, acevedo2024power}. Furthermore, the inequalities defining the tropical Vandermonde cell can be described in terms of the superdominance order, a partial order on partitions which is also related to power sum inequalities (Theorem~\ref{thm: power sums order}), analogous to how \emph{term-normalized power sums} obey the dominance order (see \cite{cuttler2011inequalities}). We give a combinatorial description of the tropicalization of the dual cone to even symmetric nonnegative functions (Theorem~\ref{thm:tropmom}). The superdominance order also helps us compute the tropicalization of the dual cone to symmetric sums of squares via a lemma on tropical spectrahedra (Lemma~\ref{le:trop of sos}).  We also use symmetry-adapted bases constructed from higher Specht polynomials (as in \cite{blekherman2021symmetric}) and \emph{partial symmetry reduction} (a technique introduced in \cite{acevedo2024power}). We conjecture that the tropicalization of the dual cone to symmetric sums of squares has a combinatorial characterization in terms of the superdominance order (Conjecture~\ref{conj:tropsos}).
We now introduce background and notation, and discuss our results in detail.

\subsection{Background and notation}
Let $\S_n$ denote the \emph{symmetric group} on $n$ letters, i.e. the group of permutations of $[n]:=\{1,2,\dots,n\}$. Symmetric polynomials in $\R[X_1,\dots,X_n]$ are invariant polynomials under the action of $\S_n$ that permutes the variables.
\smallskip

Let $\B_n$ denote the \emph{hyperoctahedral group} on $n$ letters, i.e. the group of signed permutations of $[n]$. We call a polynomial in $\R[X_1,\dots,X_n]$ \emph{even symmetric} if it is invariant under the action of $\B_n$ that permutes the variables (multiplying by the sign of the permutation). Equivalently, $n$-variate even symmetric polynomials are symmetric polynomials in the ring $\R[X_1^2,\dots,X_n^2]$.
\smallskip

Let $H_{n,d}:=\R[X_1,\dots,X_n]_d$ be the $\R$-vector space of $n$-variate homogeneous polynomials (forms) of degree $d$ in the variables $X_1,\dots,X_n$. With the above actions, each vector space $H_{n,d}$, for fixed $n$ and any $d$, is both an $\S_n$-module and a $\B_n$-module. Denote by $H_{n,d}^\S$ the subspace of $\S_n$-invariant forms in $H_{n,d}$, and by $H_{n,d}^\B$ the subspace of $\B_n$-invariant forms in $H_{n,d}$.
\smallskip

It is well-known that the dimension of $H_{n,d}^\S$ (resp. $H_{n,2d}^\B$) as an $\R$-vector space is the number of partitions of $d$ with at most $n$ parts. A \emph{partition} $\la$ of a positive integer $d$ is a decreasing sequence of positive integers that add up to $d$, i.e. $\la=(\la_1,\dots,\la_k)$ where $\la_1\ge\la_2\ge\dots\ge\la_k$, $\la_1+\la_2+\dots+\la_k=d$, and $\la_i\in\N_{>0}$ for $i=1,\dots,k$. The numbers $\la_i$ are called the \emph{parts} of $\la$, and the number of parts of $\la$ is called its \emph{length}, which we denote by $\ell(\la)$. The sum of the parts of a partition $\la$, denoted $|\la|$, is called its \emph{size}. We call a partition \emph{even} if all its parts are even numbers. For example, if $\la=(4,2,2)$ then $\la$ is even, $\la_1=4,\la_2=\la_3=2$, $\ell(\la)=3$ and $|\la|=8$. We condense the number of repeated parts in an exponent, for example, $(4,2^3,1^2):=(4,2,2,2,1,1)$. We consider the empty set $\emptyset$ to be the unique partition of $0$, so that $\ell(\emptyset)=|\emptyset|=0$. 
\smallskip

We denote by $\Lambda_d$ the set of partitions of $d$ and by $\Lambda_{2d}^\varepsilon$ the set of even partitions of $2d$. We will often use the notation
\[ \la\vdash d \text{ if } \la\in\Lambda_d, \text{ and }\la\vDash 2d \text{ if } \la\in\Lambda_{2d}^{\varepsilon}.\]

The vector spaces $H_{n,d}^\S$ and $H_{n,2d}^\B$ have many well-known bases, some of them being the \emph{power sum polynomials}, the \emph{monomial symmetric polynomials}, and the \emph{elementary symmetric polynomials}. For a positive integer $r$, $p_r(X_1,\dots,X_n):=X_1^r+\dots+X_n^r$ is the $r$-th power sum polynomial, and, if $r\le n$, $e_r(X_1,\dots,X_n):=\sum_{1\le i_1<\dots<i_r\le n}X_{i_1}\cdots X_{i_r}$ is the $r$-th elementary symmetric polynomial. We will write $p_r$ and $e_r$ when the number of variables is clear from the context. 
For a partition $\la$ we define
\begin{align*}
    p_\la =\prod_{i=1}^{\ell(\la)}p_{\la_i}, \quad    e_\la =\prod_{i=1}^{\ell(\la)}e_{\la_i}, \quad
    m_\la =\sum X^\la,
\end{align*}
the power sum, elementary, and monomial symmetric polynomials respectively, and where we sum over all monomials whose exponents are precisely $\la_1,\dots,\la_{\ell(\la)}$ in $m_\la$. For $n\ge d$ the sets $\{p_\la:\la\vdash d\}$, $\{e_\la:\la\vdash d\}$, $\{m_\la:\la\vdash d\}$ are bases for $H_{n,d}^\S$, and the sets $\{p_\la:\la\vDash 2d\}$, $\{e_\la:\la\vDash 2d\}$, $\{m_\la:\la\vDash 2d\}$ are bases for $H_{n,2d}^\B$. Hence, for $n\ge d$ the dimension of $H_{n,d}^\S$ (and the dimension of $H_{n,2d}^\B$) is the number of partitions of $d$, which we denote by $\pi(d)$. 
The transition matrices between all three bases do not depend on the number of variables for all $n \geq d$  \cite[Sec.~I.6.]{macdonald1998symmetric}. 
\smallskip

A \emph{symmetric function} is a formal power series in countably many variables that is invariant under the permutation of the variables and of bounded degree. We consider the \emph{power sum functions} $\p_k := \sum_{i \in \N} X_i^k$, \emph{elementary symmetric functions} $\mathfrak{e}_k = \sum_{I \subset \N, |I| = k}\prod_{i \in I}X_i$ and \emph{monomial symmetric functions} $\mathfrak{m}_\la = \sum X^\la$. An \emph{even symmetric function} is a symmetric function in which the exponent vector of every monomial contains only even integers.
\smallskip

We will usually write $\R^{\pi(d)}$, indicating that the coordinates are labeled with the partitions of $d$ ordered according to the \emph{reverse lexicographic order}. The reverse lexicographic order $>_{revlex}$ is a total order on the set of partitions. For any partition $\la$ let $\la_{(1)}\le\la_{(2)}\le\dots\le\la_{(\ell(\la))}$ denote the parts of $\la$ in increasing order, i.e., $\la_{(i)}=\la_{\ell(\la)+1-i}$ for $i=1,\dots,\ell(\la)$. For distinct partitions $\la$ and $\mu$ we say that $\la>_{revlex}\mu$ if for the smallest $i$ such that $\la_{(i)}\ne\mu_{(i)}$ we have $\la_{(i)}<\mu_{(i)}$. For example, $(1^4)>(2,1^2)>(3,1)>(2^2)>(4)$ where $>$ is $>_{revlex}$. We list coordinates of vectors $c\in\R^{\pi(d)}$ in this order, from largest to smallest, e.g. for $c\in\R^{\pi(4)}$: $c=(c_1,c_2,c_3,c_4,c_5)=(c_{(1^4)},c_{(2,1^2)},c_{(3,1)},c_{(2^2)},c_{(4)})$.
\medskip

\begin{definition}\label{def:Vandermonde map}
The \emph{Vandermonde map} of degree $d$ in $n$ variables is the function 
\begin{align*}
\nu_{n,d}:\R^n&\longrightarrow\R^d\\
x&\longmapsto(x_1+\dots+x_n,x_1^2+\dots+x_n^2,\dots,x_1^d+\dots+x_n^d).
\end{align*}
We write $\mathcal{M}_{n,d} = \nu_{n,d}(\R^n)$ for the image of the Vandermonde map. In our work, the Vandermonde map appears naturally in the study of nonnegative symmetric forms. For fixed degree $d$ and increasing number of variables the images of these maps form an ascending chain, i.e., we have $\mathcal{M}_{n,d}  \subseteq  
 \mathcal{M}_{n+1,d}$ for all $n\ge1$. We are particularly interested in the image of the \textit{Vandermonde map at infinity}, i.e. in the closure of the union of all the images $$\mathcal{M}_d:=\clo \left( \bigcup_{n=1}^\infty\mathcal{M}_{n,d} \right)~.$$ We also consider the \textit{even Vandermonde map} 
\begin{align*}
\nu_{n,d}^\varepsilon:\R^n&\longrightarrow\R^d\\
x&\longmapsto(x_1^2+\dots+x_n^2,x_1^4+\dots+x_n^4,\dots,x_1^{2d}+\dots+x_n^{2d}).
\end{align*} 
its image $\mathcal{N}_{n,d} = \nu_{n,d}^\varepsilon (\R^n)$,
and its \textit{image at infinity} $\mathcal{N}_d:=\clo \left( \bigcup_{n=1}^\infty\mathcal{N}_{n,d} \right)$.  Note that $\mathcal{N}_{n,d}$ equals the image of $\nu_{n,d}$ restricted to the nonnegative orthant. We call the image of the even Vandermonde map \emph{the Vandermonde cell}. 
\end{definition}

When working in finitely many variables we use the power sum basis of $H^S_{n,d}$. With the restriction $p_\la(X_1,\ldots,X_{n+1}) \mapsto p_\la(X_1,\ldots,X_n,0)$ we identify the vector spaces $H_{n,d}^{\S}$ for all $n \geq d$. Analogously, we identify the vector spaces $H_{n,2d}^{\B}$. This defines an inverse system whose limit can be identified with the homogeneous component of degree $d$ (resp. $2d$) of the ring of (resp. even) symmetric functions. \smallskip

Denoting by $\Sigma_{n,2d}^\S$ and $\P_{n,2d}^\S$ the cones of sums of squares and nonnegative forms in $H_{n,2d}^{\S}$, we have the inclusions $\Sigma_{n+1,2d}^\S \subset \Sigma_{n,2d}^\S$ and $\P_{n+1,2d}^\S \subset \P_{n,2d}^\S$.
The cones $$\S\Sigma_{2d} = \bigcap_{n \geq 2d} \Sigma_{n,2d}^\S \qquad\text{and}\qquad \S\P_{2d} = \bigcap_{n \geq 2d} \P_{n,2d}^\S$$ then contain any homogeneous symmetric function of degree $2d$ that is a sum of squares/nonnegative in any number of variables. 
Analogously, we define $\B\Sigma_{2d}$ and $\B\P_{2d}$ for the subspace of even symmetric forms.
\smallskip

\subsection{Main results in detail}

In Section \ref{sec:superdominance} we introduce the \emph{superdominance order} $\succeq$ on the set $\Lambda_d$ (see Definition \ref{def:superdominance}) and observe that the power sum polynomials are ordered in the same way as partitions with respect to $\succeq$. \medskip

\begin{thmn}[\ref{thm: power sums order}]
Let $\la,\mu\vdash d$. Then $p_\la\ge p_\mu$ on $\R_{\ge0}^n$ for all $n$ if and only if $\la\succeq\mu$. 
\end{thmn} \smallskip

Furthermore, if $\la$ and $\mu$ are even partitions then $\la\succeq\mu$ is equivalent to $p_\la-p_\mu$ being a sum of squares for all number of variables (Proposition \ref{prop:+impliesSOS}). By exploiting combinatorial properties of the superdominance order we prove one direction of Theorem \ref{thm: power sums order} in Section \ref{sec:superdominance}, and use properties of tropicalization to finally prove the theorem in Section \ref{sec:tropicalization}. In Section \ref{sec:trop/superdominance/power sums/ limit cones} we show that the superdominance order is tightly related to the tropicalizations of the dual cones to $\B\Sigma_{2d}$ and $\B\P_{2d}$.

In Section \ref{sec:limit cones} we study the general convex geometry of the cones $\S\Sigma_{2d}$, $\S\P_{2d}$, $\B\Sigma_{2d}$, $\B\P_{2d}$.
Our first observation in Theorem \ref{thm:Limitconesare full dimensional} is that these cones are full-dimensional pointed closed convex cones and that these are the Kuratowski limits of the cones in finitely many variables. We study the symmetric quartics in detail and find a form in $\S\P_4 \setminus \S\Sigma_4 $. 
To the best of our knowledge, this is the first example of a symmetric nonnegative form that is not a sum of squares \emph{regardless of the number of variables} ($n\ge4$), where the coefficients are independent of the number of variables when expressed in the power sum basis. \smallskip

\begin{thmn}[\ref{thm:limit quartic}] 
The symmetric quartic $ 4p_1^4-5p_2p_1^2-\frac{139}{20}p_3p_1+4p_2^2+4p_4$ is nonnegative for all number of variables. Moreover, these quartic forms are not sums of squares for any number of variables $n\ge4$.
\end{thmn} \smallskip

We are able to construct this family by finding a \emph{test set for nonnegativity} (and a test set for being a sum of squares), similar to \cite{choi1987even}, for symmetric quartics.
The test set for nonnegativity is derived from \cite{kostov2007stably}, while the test set for being a sum of squares is obtained from examining the extremal rays of the dual cone to $\S\Sigma_4$. \medskip

\begin{thmn}[\ref{thm:testset}] 
Let $f(\p_1,\p_2,\p_3,\p_4)$ be a homogeneous even symmetric function of degree $4$. Then, $f$ is nonnegative in all number of variables if and only if $f$ is nonnegative on the discrete set of parallel lines $$\{(x,1,n^{-1/2},n^{-1})\,\,|\,\,x\in\R,\, n\in\N_{>0}\}.$$ Moreover, $f$ is a sum of squares for all number of variables if and only if\/ $f$ is nonnegative on $$\{(x,1,u,u^2)\,\,|\,\,x\in\R,\, 0\le u\le1\}.$$
\end{thmn} \medskip

In Section \ref{sec:trop/superdominance/power sums/ limit cones} we find, with the help of tropicalization, an example of an even symmetric nonnegative decic that is not a sum of squares for any sufficiently large number of variables.
\medskip

\begin{propn}[\ref{prop:deg 10 nonnegative but not sos}]
The even symmetric form $  \frac{1}{18}p_{(2^5)}+3 p_{(8,2)} + 6 p_{(6,4)}-3p_{(6,2^2)}$
is nonnegative for all number of variables but not a sum of squares for any sufficiently large number of variables $n \geq N$ for some $N$.
\end{propn}
\medskip

We also show that the cones of sums of squares and nonnegative symmetric functions are distinct for all even degrees $\geq 6$. \medskip

\begin{thmn}[\ref{thm:Limitcones not equal}]
For all $d \geq 2$ we have $\S\Sigma_{2d}\subsetneq \S\P_{2d}$ and for all $d \geq 3$ we have $\B\Sigma_{2d} \subsetneq \B\P_{2d}$. Moreover, the sets $\S\P_{2d}$ and $\B\P_{2d}$ are non-semialgebraic in all these cases, and the sets $\S\Sigma_{2d}$, $\B\Sigma_{2d}$ are semialgebraic. 
\end{thmn} \medskip

The \emph{Vandermonde cell} $\mathcal{N}_d = \bigcup_{n \in \N}(p_1,\ldots,p_d)(\R_{\geq 0}^n)$ is closed under addition and under coordinatewise (Hadamard) multiplication (see Lemma \ref{lem:newton}). This, together with the two families of classical binomial inequalities between power sums (Remark \ref{rem:classical ieqs}), and the aid of a \emph{tropical Farkas lemma} (Lemma \ref{le:max-closure}), allows us to compute the tropicalization of the Vandermonde cell $\mathcal{N}_d$, which is a rational polyhedron.\medskip  

\begin{thmn}[\ref{thm:trop(Nd)}]
The set $\trop(\mathcal{N}_d)$ is a rational polyhedron with the facet-defining inequalities: 
\begin{align*}
    \trop \left( \mathcal{N}_d \right) = \left\{ \left. y\in\R^d : \begin{cases}
    y_k + y_{k+2}\geq 2y_{k+1}\; ,\quad k=1,\ldots,d-2 \\
    dy_{d-1}\geq (d-1)y_d\; .
    \end{cases} \right\}  \right\}.
\end{align*}
\end{thmn} \medskip

 We then use the $H$-representation of $\trop(\mathcal{N}_d)$ to immediately prove that power sums obey the superdominance order (Theorem \ref{thm: power sums order}). \medskip
 
We turn to the tropicalization of the dual cones $\B\Sigma_{2d}^*$ and $\B\P_{2d}^*$ in Section \ref{sec:trop/superdominance/power sums/ limit cones}. Using the superdominance order we define a decreasing chain of rational polyhedra $T^{(k)}_{2d}$ (see Definition \ref{def:sequence T2dk}) and prove that the sequence stabilizes and it does so at $\trop(\B\P_{2d}^*)$ (Theorem \ref{thm:tropmom}). With the aid of Lemma \ref{le:trop of sos} (bounds the tropicalization of a spectrahedral cone), using properties of symmetric functions and the superdominance order, we find that $\trop(\B\Sigma_{2d}^*)$ is a polyhedral cone sitting in between $T_{2d}^{(1)}$ and $T_{2d}^{(2)}$. We finally conclude the strict inclusion between the tropicalizations of the dual cones. \medskip

\begin{thmn}[\ref{thm:tropsdiffer}]
    $\trop(\B\Sigma_{2d}^*)\supsetneq\trop(\B\P_{2d}^*)$ for all $d\ge5$.   
\end{thmn} \medskip

Although we have $\B\P_{6}^* \subsetneq \B\Sigma_{6}^*$ and $\B\P_{8}^* \subsetneq \B\Sigma_{8}^*$ we find that their tropicalizations coincide.  \smallskip

In Appendix \ref{sec:appendix} we briefly explain symmetry and partial symmetry reduction techniques for representing symmetric sums of squares.  
Using partial symmetry reduction we find that the cones $\S\Sigma_{2d}^*$ and $\B\Sigma_{2d}^*$ are spectrahedra, and construct an explicit spectrahedral representation for $\S\Sigma_4^*$ and each $\B\Sigma_{2d}^*$ for $d=3,4,5$.

\section{Superdominance Order}\label{sec:superdominance}

In this section, we introduce a partial order on partitions of the same size, which we call the \emph{superdominance order}. We will see that the inequalities on power sums $p_\la\ge p_\mu$ that hold for all number of variables on the nonnegative orthant, where $\la$ and $\mu$ are partitions of the same size, are precisely those for which $\la$ superdominates $\mu$ (Theorem \ref{thm: power sums order}). We show that $\la$ superdominates $\mu$ implies $p_\la \geq p_\mu$, delaying the proof of the reverse implication to Section \ref{sec:tropicalization}.
\medskip

\begin{definition}\label{def:superdominance}
Let $\la,\mu\vdash d$. We say that $\la$ \emph{superdominates} $\mu$, denoted $\la\succeq\mu$, if
\begin{align*}
\sum_{i=1}^j\la_{(i)}\le\sum_{i=1}^j\mu_{(i)}
\end{align*}
for each $j=1,\dots,\min\{\ell(\la),\ell(\mu)\}$. We write $\la\succ\mu$ if $\la\succeq\mu$ and $\la\ne\mu$.
\end{definition}

Observe that $\succeq$ defines a partial order on $\Lambda_d$. For $d\le5$ the order $\succeq$ is a total order on $\Lambda_d$, e.g. for $d=5$,
\begin{align*}
(1^5)\succ(2,1^3)\succ(3,1^2)\succ(2^2,1)\succ(4,1)\succ(3,2)\succ(5)
\end{align*}
but $\succeq$ is no longer a total order on $\Lambda_d$ for all $d\ge6$.

\begin{remark}
Definition \ref{def:superdominance} is motivated by the following definition in the literature \cite[A.2.~Definition]{marshall1979inequalities} for real vectors. For $x,y\in\R^n$ it is said that $x$ is \emph{weakly supermajorized} by $y$, denoted $x\prec^\mathbf{w} y$, if $$\sum_{i=1}^k x_{(i)}\ge\sum_{i=1}^k y_{(i)}$$ for all $k=1,\dots,n$, where $(z_{(1)},\dots,z_{(n)})$ is obtained by permuting the entries of $z\in\R^n$ such that $z_{(1)}\le\dots\le z_{(n)}$.
\end{remark}

Using properties of the superdominance order and Theorem \ref{thm:trop(Nd)} we show that power sums of a fixed degree are ordered in precisely the same way as partitions with respect to the superdominance order. 

\begin{theorem}\label{thm: power sums order}
Let $\la,\mu\vdash d$. Then $p_\la\ge p_\mu$ on $\R_{\ge0}^n$ for all $n$ if and only if $\la\succeq\mu$. 
\end{theorem}

An analogous result was obtained by Cuttler, Greene, and Skandera \cite[Theorem 4.2]{cuttler2011inequalities} for \emph{term-normalized power sums} and the \emph{dominance order}. \medskip

For a partition $\la$ the $\la$\emph{-th} \emph{term-normalized power sum} is $P_\la:=\frac{p_\la}{n^{\ell(\la)}}$. The dominance order $\unrhd$ on $\Lambda_d$, is defined as follows: $\la\unrhd\mu$ if
\begin{align*}
\sum_{i=1}^{j}\la_i\ge\sum_{i=1}^j\mu_i
\end{align*}
for each $j=1,\dots,\min(\ell(\la),\ell(\mu))$.

\begin{remark}\label{rem:equiv of orders}
    Observe that the dominance order is equivalent to the superdominance order on partitions of the same length, i.e., if $\la,\mu\vdash d$ satisfy $\ell(\la)=\ell(\mu)$ then $\la\succeq\mu$ if and only if $\la\unrhd\mu$.
\end{remark}

\begin{remark}
In analogy to Theorem \ref{thm: power sums order}, the binomial inequalities in elementary symmetric polynomials $e_k = \sum_{1 \leq i_1 < \ldots < i_k \leq n}X_{i_1}\cdots X_{i_k}$ are characterized by the reverse-dominance order. For partitions $\lambda,\mu \in \Lambda_d$ we have $e_\la \geq e_\mu$ on $\R_{\geq 0}^n$ for all $n \in \N$ if and only if $\lambda \unlhd \mu$. 
This characterization is the same as for term-normalized elementary symmetric polynomials $E_\lambda:=\frac{e_\lambda}{e_\lambda(1,\ldots,1)}$ (see \cite[Theorem~3.2]{cuttler2011inequalities}). One observes that $e_{(k-1,l+1)}(\underline{X}^2)-e_{(k,l)}(\underline{X}^2)$ is a sum of squares and thus nonnegative, where $\underline{X}^2 := (X_1^2,\ldots,X_n^2)$. The converse direction can be proved exactly as in \cite{cuttler2011inequalities} for $E_\la \geq E_\mu$.
\end{remark}

\subsection{Properties of the superdominance order}
It is well-known that the lexicographic order \emph{refines} the dominance order. We show that the reverse lexicographic order refines the superdominance order, among other elementary properties. 

\begin{proposition}\label{prop:weakorderprpts}
If $\la,\mu\vdash d$ and $\la\succ\mu$, then $\ell(\la)\ge\ell(\mu)$ and $\la>_{revlex}\mu$.
\end{proposition}
\begin{proof}
Since both are partitions of $d$ then one cannot finish adding the parts of $\la$ first, so $\ell(\la)\ge\ell(\mu)$.

If $\mu>_{revlex}\la$ then the smallest index $j$ such that $\la_{(j)}\ne\mu_{(j)}$ also satisfies $\mu_{(j)}<\la_{(j)}$, and then $\sum_{i=1}^j\mu_{(i)}<\sum_{i=1}^j\la_{(i)}$, so $\la\not\succeq\mu$.
\end{proof}

The converse of Proposition \ref{prop:weakorderprpts} is not true: for example $\la=(3^4,1)$ and $\mu=(7,2^3)$ are incomparable in the superdominance order.

For $\la,\mu\vdash d$ denote $\la\gtrdot\mu$ if $\la$ \emph{covers} $\mu$ in the superdominance order, i.e., $\la\succ\mu$ and there is no $\nu\vdash d$ such that $\la\succ\nu\succ\mu$.

\begin{definition}
    If $\ell(\la)>1$ denote by $\la^*$ the partition obtained from $\la$ by merging $\la_1$ and $\la_2$ into a single part $\la_1+\la_2$.
\end{definition}

For example $(3^2)^*=(6)$ and $(4,3,1)^*=(7,1)$.

\begin{proposition}\label{Prop:coverprpts}
    Let $\la,\mu\vdash d$ with $\ell(\la)>\ell(\mu)$. Then 
    \begin{enumerate}
        \item[(i)] $\la\succ\mu\Rightarrow\la^*\succeq\mu$,
        \item[(ii)] $\la\gtrdot\mu\Rightarrow\mu=\la^*$,
        \item[(iii)] $\la\gtrdot\la^*\iff \la_1-\la_2\le1$.
    \end{enumerate}    
\end{proposition}
\begin{proof}
    (i) Let $\ell(\la)=l$. We have
    \begin{align*}
      \sum_{i=1}^k\la_{(i)}^*=\sum_{i=1}^k\la_{(i)}\le\sum_{i=1}^k\mu_{(i)}
    \end{align*}
    for each $k=1,\dots,l-2$. Since $\ell(\la^*)=l-1$ and $\ell(\mu)\le l-1$ then $\sum_{i=1}^{l-1}\la_{(i)}^*=d=\sum_{i=1}^{l-1}\mu_{(i)}$, therefore $\la^*\succeq\mu$.
    \smallskip

    (ii) By (i) $\la\gtrdot\mu$ implies $\la^*\succeq\mu$. But $\la\succ\la^*$, so $\la\succ\la^*\succeq\mu$, and since $\la\gtrdot\mu$ we must have $\mu=\la^*$.
    \smallskip

    (iii) If $\la_1-\la_2>1$ the partition $\nu$ with $\nu_1=\la_1-1$ and $\nu_2=\la_2+1$ and the rest of its parts coinciding with $\la$ lies strictly between $\la$ and $\la^*$. If $\la$ does not cover $\la^*$ then we have $\la\succ\nu\succ\la^*$ for some $\nu\vdash d$, which, due to (i), is only possible if $\ell(\nu)=\ell(\la)$. But then $\la_i=\nu_i$ for $i=3,\dots,\ell(\la)$, and so $(\la_1,\la_2)\rhd(\nu_1,\nu_2)$ and $\la_1+\la_2=\nu_1+\nu_2$ so $\la_1-\la_2>1$.
\end{proof}

\begin{proposition}\label{prop:covering}
    Let $\la,\mu\vdash d$. Then $\la\gtrdot\mu$ if and only if one of the following holds:
    \begin{enumerate}
        \item[(i)] $\ell(\la)=\ell(\mu)$ and $\la$ covers $\mu$ in the dominance order.
        \item[(ii)] $\la_1-\la_2\le1$ and $\mu=\la^*$.
    \end{enumerate}
\end{proposition}
\begin{proof}
    By Remark \ref{rem:equiv of orders}, if $\ell(\la)=\ell(\mu)$, then $\la\gtrdot\mu$ is equivalent to $\la$ covering $\mu$ in the dominance order.
    \smallskip

    If $\ell(\la)>\ell(\mu)$ then by Proposition \ref{Prop:coverprpts}, $\la\gtrdot\mu$ implies $\mu=\la^*$ and $\la_1-\la_2\le1$. The other implication follows from Proposition \ref{Prop:coverprpts} (iii).
\end{proof}

\begin{definition} \label{def:fusion}
The \emph{fusion} of partitions $\la,\mu,\dots,\omega$, denoted by $\la\mu\cdots\omega$, is the unique partition whose parts are precisely the parts of $\la,\mu,\dots,\omega$ put together. Denote by $\mu^{\circ k}$ the fusion of $k$ copies of $\mu$.
\end{definition}

\hspace{-1mm}For example, if $\la=(2^2,1)$ and $\mu=(3,2,1)$ then $\la\mu=(3,2^3,1^2)$ and $\mu^{\circ 2}=\mu\mu=(3^2,2^2,1^2)$. Observe that $\ell(\la\mu\cdots\omega)=\ell(\la)+\ell(\mu)+\dots+\ell(\omega)$, and $|\la\mu\cdots\omega|=|\la|+|\mu|+\dots+|\omega|$.

\begin{lemma}\label{lem:fusion}
    Let $k\ge2$ and $\la^i,\mu^i\vdash d_i$ such that $\la^i\succeq\mu^i$ for $i=1,\dots,k$, then $\la^1\cdots\la^k\succeq\mu^1\cdots\mu^k$.
\end{lemma}
\begin{proof}
    It suffices to show the result for $k=2$. Let $\la:=\la^1\la^2$, $\mu:=\mu^1\mu^2$ and $m:=\ell(\mu)$.
    \smallskip
    
    Since $\la^i\succeq\mu^i$, by Proposition \ref{prop:weakorderprpts} we have $\ell(\la^i)\ge\ell(\mu^i)$ and so $\ell(\la)\ge m$. 
    \smallskip
    
    Define $\tilde{\mu}^i$ to be the partition of length $\ell(\la^i)$ such that $\tilde{\mu}^i_{(j)}=\mu_{(j)}^i$ for $j=1,\dots,\ell(\mu^i)$, and $\tilde{\mu}^i_{(j)}=r$ for all $j>\ell(\mu^i)$, where $r$ is larger than any part of $\la^i$ and $\mu^i$ for $i=1,2$. So $\tilde{\mu}^i\prec^\mathbf{w}\la^i$ and, using \cite[5.A.7]{marshall1979inequalities}: if $x\prec^\mathbf{w} y$ on $\R^n$ and $a\prec^\mathbf{w} b$ on $\R^m$ then $(x,a)\prec^\mathbf{w}(y,b)$ on $\R^{m+n}$, we obtain $\tilde{\mu}\prec^\mathbf{w}\la$ where $\tilde{\mu}:=\tilde{\mu}^1\tilde{\mu}^2$. Therefore 
    $$\sum_{j=1}^k\la_{(j)}\le\sum_{j=1}^k\tilde{\mu}_{(j)}$$ for each $k=1,\dots,\ell(\la)$, and in particular for each $k=1,\dots,m$. But then $\la\succeq\mu$ because by construction of $\tilde{\mu}$, the parts $\tilde{\mu}_{(1)},\dots,\tilde{\mu}_{(m)}$ are all smaller than $r$, so they must be precisely the parts of $\mu$.
\end{proof}

The following is a cancellation property of the superdominance order.

\begin{lemma}\label{lem:arquimedean}
Let $\la,\mu\vdash d$ and $\nu$ be any partition. Then $\la\succeq\mu\iff\la\nu\succeq\mu\nu$.
\end{lemma}
\begin{proof}
The forward implication follows from Lemma \ref{lem:fusion}. For the other implication it is enough to consider the case $\ell(\nu)=1$, so let $\nu=(t)$ with $t\in\N_{>0}$. If $t$ occurs first in $\mu\nu$ than in $\la\nu$, i.e., if $t=\mu\nu_{(i)}$ and $t=\la\nu_{(j)}$ with $i\le j$, then we can cancel $t$ without altering the order. The only remaining case is when $t=\mu\nu_{(i)}$ and $t=\la\nu_{(j)}$ with $i>j$. If $t\ge\la_1$ then, since both partitions are of the same size, it is clear we can cancel $t$ from both partitions, and the inequality is preserved. So suppose $\la_{(j-1)}\le t<\la_{(j)}$. For each $j=1,\dots,\ell(\la)$ denote $\check{\la}_j=\sum_{i=1}^j\la_{(i)}$, and define $\check{\la}_j=|\la|$ for all $j>\ell(\la)$. Since $\la\nu\succeq\mu\nu$ we have 
\begin{align*}
    \check{\la}_{j-1}+t+\la_{(j)}&\le\check{\mu}_{j-1}+\mu_{(j)}+\mu\nu_{(j+1)}\\
    \Rightarrow\check{\la}_{j-1}+\la_{(j)}&\le\check{\mu}_{j-1}+\mu_{(j)}+\mu\nu_{(j+1)}-t\\
    &\le\check{\mu}_{j-1}+\mu_{(j)}
\end{align*}
the last inequality because $\mu\nu_{(j+1)}\le t$. Hence $\check{\la_j}\le\check{\mu_j}$, and similarly $\check{\la}_k\le\check{\mu}_k$ as long as $\mu\nu_{(k+1)}\le t$. But after $t$ has appeared in both sides of the inequalities $\check{\la\nu}_r\le\check{\mu\nu}_r$ then the inequalities $\check{\la}_r\le\check{\mu}_r$ are immediate because we can just cancel $t$ on both sides.
\end{proof}

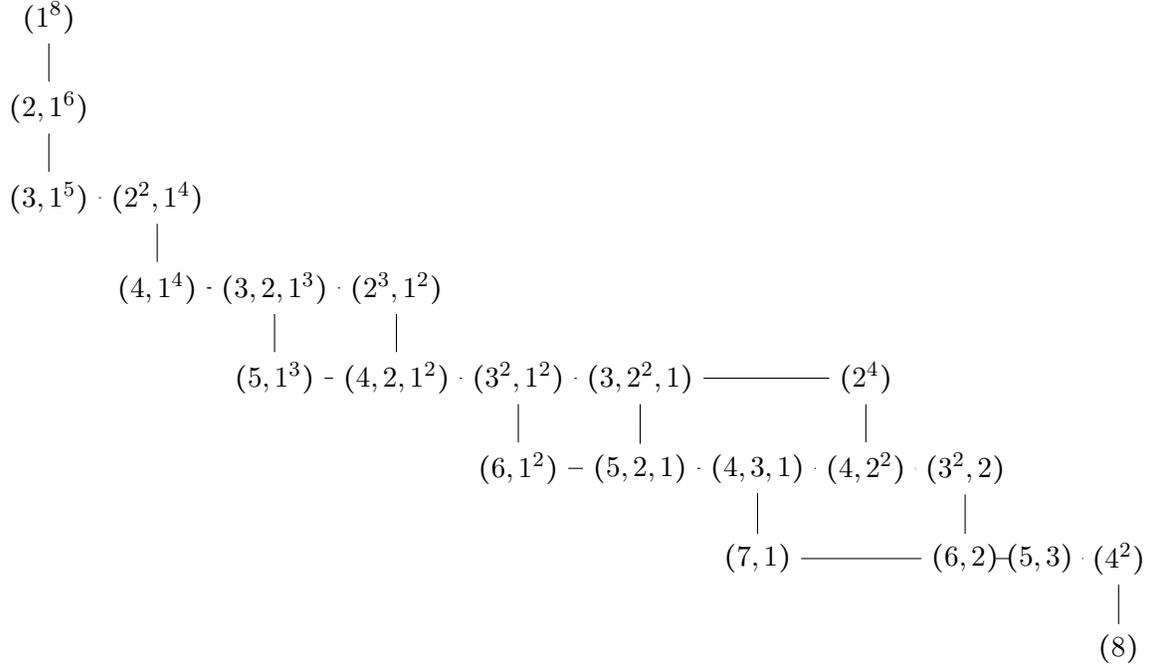
\begin{figure}
\begin{center}
\begin{tikzpicture}[scale=.6]
    \node (a) at (0,16) {$(1^8)$};
    \node (b) at (0,14) {$(2,1^6)$};
    \node (c) at (0,12) {$(3,1^5)$};
    \node (d) at (2.4,12) {$(2^2,1^4)$};
    \node (e) at (2.4,10) {$(4,1^4)$};  
    \node (f) at (5,10) {$(3,2,1^3)$};
    \node (g) at (7.7,10) {$(2^3,1^2)$};
    \node (h) at (5,8) {$(5,1^3)$};
    \node (i) at (7.7,8) {$(4,2,1^2)$};
    \node (j) at (10.4,8) {$(3^2,1^2)$};
    \node (k) at (13.1,8) {$(3,2^2,1)$};
    \node (u) at (18.1,8) {$(2^4)$};
    \node (l) at (10.4,6) {$(6,1^2)$};
    \node (m) at (13.1,6) {$(5,2,1)$};
    \node (n) at (15.7,6) {$(4,3,1)$};
    \node (o) at (18.1,6) {$(4,2^2)$};
    \node (v) at (20.3,6) {$(3^2,2)$};
    \node (p) at (15.7,4) {$(7,1)$};
    \node (q) at (20.3,4) {$(6,2)$};
    \node (r) at (21.95,4) {$(5,3)$};
    \node (s) at (23.7,4) {$(4^2)$};
    \node (t) at (23.7,2) {$(8)$};
  
    \draw (a)--(b)--(c);
    \draw (f)--(h)--(i);
    \draw (g)--(i)--(j)--(k);
    \draw (c)--(d)--(e)--(f);
    \draw (e)--(f)--(g);
    \draw (k)--(u)--(o)--(v)--(q);
    \draw (j)--(l)--(m)--(n)--(o);
    \draw(k)--(m);
    \draw (n)--(p)--(q)--(r)--(s)--(t);
\end{tikzpicture}
\caption{Superdominance poset $(\Lambda_8,\succeq)$ (decreases going down or right).}
\end{center}
\end{figure}

\subsection{Consequences for power sums} \label{sec: Consequences for power sums}

\begin{proposition}\label{prop:equivdominance}
    Let $\la,\mu\vDash 2d$ and $P_\la = \frac{p_\la}{n^{\ell(\la)}}, P_\mu = \frac{p_\mu}{n^{\ell(\mu)}}$, then the following are equivalent:
    \begin{enumerate}
        \item[(i)] $P_\la-P_\mu$ is a sum of squares  for all $n\in\N$,
        \item[(ii)] $P_\la\ge P_\mu$ \hspace{2mm} for all $n\in\N$, 
        \item[(iii)] $\la\unrhd\mu$.
    \end{enumerate}
\end{proposition}
\begin{proof}
    Trivially (i) implies (ii). By \cite[Theorem 4.2]{cuttler2011inequalities} (ii) and (iii) are equivalent. 
    
    If $\la\unrhd\mu$ then \cite[Proposition 2.3]{brylawski1973lattice} and \cite[Lemma 1, Lemma 3]{frenkel2014minkowski} imply (i).
\end{proof}

\begin{corollary}\label{cor:sd->sos}
    Let $\la,\mu\vDash2d$. If $\la\succeq\mu$ then $p_\la-p_{\mu}$ is a sum of squares for all $n\in\N$.
\end{corollary}
\begin{proof}
    If $\ell(\la)=\ell(\mu)$ then this is a consequence of Proposition \ref{prop:equivdominance}. Otherwise, it is a consequence of Proposition \ref{prop:covering}, since $p_\la-p_{\la^*}$ is clearly a sum of squares for all $n\in\N$.
\end{proof}

\begin{proof}[Proof of ``$\Leftarrow$'' direction of Theorem \ref{thm: power sums order}]
Observe that for $\la,\mu\vdash d$ we have $p_\la\ge p_\mu$ on $\R_{\ge0}^n$ if and only if $p_{2\la}\ge p_{2\mu}$ on $\R^n$, where we denote $2\nu:=(2\nu_1,\dots,2\nu_{\ell(\nu)})$ for a partition $\nu$. Hence, Corollary \ref{cor:sd->sos} implies a stronger statement than the ``$\Leftarrow$'' direction of Theorem \ref{thm: power sums order}.    
\end{proof}

The ``$\Rightarrow$'' direction of Theorem \ref{thm: power sums order} (which we prove in Section \ref{sec:tropicalization}), together with Corollary \ref{cor:sd->sos}, implies the following.

\begin{proposition}\label{prop:+impliesSOS}
    Let $\la,\mu\vDash2d$. If\/ $p_\la\ge p_\mu$ $\forall n\in\N$ then $p_\la-p_\mu$ is a sum of squares $\forall n\in\N$.
\end{proposition}

\section{Limit cones} \label{sec:limit cones}

In this section, we analyze the cones of homogeneous nonnegative symmetric functions and sums of squares, i.e. the symmetric functions which are nonnegative/sums of squares in any number of variables. In Subsection \ref{Sec:Sums of squares and nonnegative symmetric functions} we show that these cones are full-dimensional. Subsection \ref{sec:symquartics} is dedicated to studying symmetric quartics, where we identify test sets for determining whether a symmetric quartic is nonnegative or a sum of squares for all number of variables (Theorem \ref{thm:testset}). Moreover, we present a family of nonnegative polynomials that are never sums of squares for any number of variables $n \geq 4$ (Theorem \ref{thm:limit quartic}). This shows that the cone of symmetric sum of squares quartics is strictly contained in the cone of symmetric nonnegative quartics, in contrast to the term-normalized case \cite[Theorem 2.9]{blekherman2021symmetric}.
Finally, in Subsection \ref{sec:non semialgebraic}, we show that for all higher degrees, the cone of symmetric sums of squares is also strictly contained in the cone of symmetric nonnegative forms (Theorem \ref{thm:Limitcones not equal}).

\subsection{Sums of squares and nonnegative symmetric functions} \label{Sec:Sums of squares and nonnegative symmetric functions}

 \begin{definition}\label{def:limit sets}
 The limit sets of sums of squares and nonnegative symmetric forms in $n$ variables of degree $2d$ are defined as 
 \begin{align*}
   \S\Sigma_{2d} := \bigcap_{n \geq 2d} \Sigma_{n,2d}^{\mathcal{S}}\qquad \mbox{and} \qquad \S\P_{2d} := \bigcap_{n \geq 2d} \P_{n,{2d}}^{\mathcal{S}}~.
 \end{align*} 
 \end{definition}

We consider the sets $\S\Sigma_{2d},\S\P_{2d}$ as convex cones in the vector space $\R^{\pi(2d)}$.
\begin{theorem} \label{thm:Limitconesare full dimensional}
 The sets $ \S\Sigma_{2d}$ and $\S\P_{2d}$ are full-dimensional pointed closed convex cones and \begin{align*}
         \S\Sigma_{2d} & = \lim_{n \rightarrow \infty} \Sigma_{n,2d}^{\mathcal{S}}~, \quad 
         \S\P_{2d} = \lim_{n \rightarrow \infty} \P_{n,2d}^{\mathcal{S}}
     \end{align*} 
     are the Kuratowski limits.
\end{theorem}
\begin{proof}
To show that the limit sets defined in Definition \ref{def:limit sets} are indeed the Kuratowski limits of the sequences of convex sets $(\Sigma_{n,2d}^{\mathcal{S}})_n$ and $(\P_{n,2d}^{\mathcal{S}})_n$ we need to prove the inclusions
\begin{align*}
    \limsup_{n} \Sigma_{n,2d}^{\mathcal{S}}  \subset  \S\Sigma_{2d} \subset   \liminf_{n} \Sigma_{n,2d}^{\mathcal{S}}, \quad
     \limsup_{n} \P_{n,2d}^{\mathcal{S}} \subset  \S\P_{2d} \subset   \liminf_{n} \P_{n,2d}^{\mathcal{S}}
\end{align*}
which follow from the nestedness of the sequences.  \\
The limit sets are closed, pointed convex cones as intersections of closed, pointed convex cones. We need to show that they are also full-dimensional.  
The homogeneous component of the ring of symmetric functions of degree $2d$ is spanned by the set \begin{align} \label{eq:partition division}
    \{ \p_\lambda \p_\mu : \lambda,\mu \vdash d \} \cup \{ \p_{a+b} \p_\lambda \p_\mu : 1 \leq a,b \leq d, \lambda \vdash d-a, \mu \vdash d-b\}. 
\end{align} 
This is because the partitions $\omega$ of $2d$ can be divided into two types: those that are the fusions of partitions $\lambda,\mu \vdash d$, and those that are the fusion of three partitions $(a+b),\lambda,\mu$ for some integers $1 \leq a,b \leq d$ and $\lambda \vdash d-a, \mu \vdash d-b$. To see this, let $\omega \vdash 2d$ be a partition that is not the fusion of two partitions of $d$. Let $0 \leq k \leq d-1$ be the maximal integer such that $\omega$ is the fusion of a partition $\mu^1 \vdash k$ and a partition $\nu \vdash 2d-k\geq d+1$. We have $2d-k-\nu_1 < d$ since otherwise $d \geq k + \nu_1$, and since $\lambda $ does not contain a partition of $d$, we have $d > k + \nu_1$ which contradicts the maximality of $k$. We have $\nu = (\nu_1)\mu^2$ for a partition $\mu^2 \vdash 2d-k-\nu_1$. We obtain the fusion representation of $\omega$ by setting $a=d-k, b=k+\nu_1-d$ and have $\omega = (a+b)\mu^1 \mu^2$, where $1 \leq a,b \leq d$ and $\mu^1 \vdash d-a, \mu^2 \vdash d-b$.  \\
We show that $\S\Sigma_{2d}$ is full-dimensional by proving that any element from (\ref{eq:partition division}) is contained in the linear span of $\S\Sigma_{2d}$. Note that $\p_{(\lambda,\lambda)}$, $\p_{(\mu,\mu)}$, $(\p_\lambda-\p_\mu)^2 = \p_{(\lambda,\lambda)}-2\p_{(\lambda,\mu)} +\p_{(\mu,\mu)} \in \S\Sigma_{2d}$ which implies $\{ \p_\lambda \p_\mu : \lambda,\mu \vdash d \}$ is contained the linear span of $\S\Sigma_{2d}$. \\
For some integers $1 \leq a,b \leq d$ and partitions $\lambda \vdash d-a,~ \mu \vdash d-b$ we have $(X_1^a-X_2^a)p_\lambda,(X_1^b-X_2^b)p_\mu \in H_{n,d}$ are equivariants of the Specht module $\mathbb{S}^{(n-1,1)}$. For $a \neq b$ we have 
\begin{align*}   
& f_1 := \frac{n}{2} \sum_{\sigma \in \mathcal{S}_n}\frac{1}{n!}\sigma \cdot \left( (x_1^a-x_2^a)^2p_\lambda^2\right)  &= \quad p_\la p_\la\left(p_{2a}-\frac{p_{(a,a)}-p_{2a}}{n-1}\right), \\
&f_2 :=\frac{n}{2} \sum_{\sigma \in \mathcal{S}_n}\frac{1}{n!}\sigma \cdot \left( (x_1^a-x_2^a)(x_1^b-x_2^b)p_\lambda p_\mu\right)  &= \quad  p_\lambda p_\mu\left( p_{a+b}-\frac{p_{(a,b)}-p_{a+b}}{2(n-1)}\right), \\
&f_3 :=\frac{n}{2} \sum_{\sigma \in \mathcal{S}_n}\frac{1}{n!}\sigma \cdot \left( (x_1^a-x_2^a)(x_1^a-x_2^a)p_\lambda p_\mu\right) &= \quad p_\lambda p_\mu\left(p_{2a}-\frac{p_{(a,a)}-p_{2a}}{n-1}\right). 
\end{align*}
In particular, for any integer $n \geq 2d$ we have by \cite[Theorem~4.15]{blekherman2021symmetric}
\begin{align*}
 f_1-2f_2  + p_{(\mu,\mu)}\left(p_{2b}-\frac{p_{(b,b)}-p_{2b}}{n-1}\right),\quad f_1-f_3+p_{(\mu,\mu)}\left(p_{2a}-\frac{p_{(a,a)}-p_{2a}}{n-1}\right)  \in \Sigma_{n,2d}^{\mathcal{S}}.
\end{align*} 
Thus, $\p_{2a}\p_{(\lambda,\lambda)}-2\p_{a+b}\p_{\lambda}\p_{\mu}+\p_{2b}\p_{(\mu,\mu)}$, $\p_{2a}\p_{(\lambda,\lambda)}-2\p_{2a}\p_{\lambda}\p_{\mu}+\p_{2a}\p_{(\mu,\mu)} \in \S\Sigma_{2d}$. Since also $\p_{(\lambda,\lambda)}$, $\p_{2a}\p_{\lambda}\p_{\lambda} \in \S\Sigma_{2d}$, the cone $\S\Sigma_{2d}$ spans the same linear space as the set (\ref{eq:partition division}) which shows that the cones $\S\Sigma_{2d} \subset \S\P_{2d}$ are full-dimensional.  
\end{proof}
 
It can be seen more directly that the limit cones $\B\Sigma_{2d}$, $\B\P_{2d}$ are full-dimensional in the graded component of even symmetric functions of degree $2d$. Any even power sum is already a sum of squares. Therefore, in the basis of even power sums, the convex cone $\B\Sigma_{2d}$ contains the nonnegative orthant as a  full-dimensional subcone. 

\subsection{Dual cones and the Vandermonde map}
Let $K$ be a cone in a finite-dimensional real vector space $V$. The \emph{dual cone} to $K$ is
\begin{align*}
    K^*:=\{\ell\in V^*:\ell(v)\ge0\text{ for all }v\in K\}
\end{align*}
where $V^*$ is the dual vector space to $V$.

As explained in \cite[Section~4]{acevedo2024power}, \((\P_{n,2d}^{\S})^*\) and \((\P_{n,2d}^{\B})^*\) are the convex hulls of all point evaluations. When working with the power sum basis, the point evaluations can be identified with the sets \(\nu_{2d}(\mathcal{M}_{n,2d})\) and \(\nu_d(\mathcal{N}_{n,d})\), respectively. By homogeneity, we can restrict to the fiber obtained by setting \(p_2 = 1\). Since these are compact sets that do not contain the origin, their convex conical hulls are closed cones. Consequently, we have \((\P_{n,2d}^{\S})^* = \cone (\nu_{2d}(\mathcal{M}_{n,2d}))\) and \((\P_{n,2d}^{\B})^* = \cone (\nu_{d}(\mathcal{N}_{n,d}))\).

\begin{remark}\label{rem:weightprism}
Observe that $\mathcal{M}_d$ and $\mathcal{N}_d$ are \emph{weighted homogeneous sets}, i.e. if $(a_1,\dots,a_d)\in\mathcal{M}_d$ then $(ta_1,t^2a_2,\dots,t^da_d)\in\mathcal{M}_d$ for all $t\in\R$, and similarly for $\mathcal{N}_d$ with $t\in\R_{\ge0}$. Further observe that the first coordinate of $\mathcal{M}_d$ is free (but it is not the case for $\mathcal{N}_d$), which is to say $\mathcal{M}_d$ is a \emph{prism} with respect to the first coordinate, i.e., $\mathcal{M}_d=\mathcal{M}_d+\spn(1,0,\dots,0)$. To see this let $a\in\R$ and consider the sequence $x_n=(\frac{a}{n},\frac{a}{n},\ldots,\frac{a}{n})\in\R^n$, so $$\mathcal{M}_d\ni\lim_{n \rightarrow\infty}(p_1,\ldots,p_d)(x_n)=\lim_{n\rightarrow\infty}(a,\frac{a^2}{n},\ldots,\frac{a^d}{n^{d-1}})=(a,0,\ldots,0)$$ 
which follows since the set $\mathcal{M}_d$ is closed under addition (see Lemma \ref{lem:newton}).
\end{remark}

\begin{lemma}
For any degree $d$ we have $\S\P_{2d}^* = \cone (\nu_{2d}(\M_{2d})) $ and $\B\P_{2d}^* = \cone (\nu_{d}(\mathcal{N}_{d}))$.
\end{lemma}

\begin{proof}
We restrict to the cone $\S\P_{2d}^*$ since the proof for $\B\P_{2d}^*$ works analogously.
The convex cone $\S\P_{2d}$ is closed by Theorem \ref{thm:Limitconesare full dimensional}, and a symmetric function lies in $\S\P_{2d}$ if and only if it is nonnegative on $\nu_{2d} (\mathcal{M}_{n,2d})$ for all $n \in \N$ which is equivalent to the nonnegativity on the closure of the union of those sets. Therefore, we have $\cone(\nu_{2d}(\mathcal{M}_{2d}))^* = \S\P_{2d}$. Moreover, $\cone(\nu_{2d}(\mathcal{M}_{2d}))$ is closed, since the fiber of $\nu_{2d}(\mathcal{M}_{2d})$ by $\p_2=1$ is a compact set (\cite[Prop.~4]{kostov2007stably}) whose convex hull does not contain the origin, and its convex conical hull is $\cone(\nu_{2d}(\mathcal{M}_{2d}))$. This shows the claim by conical duality. 
\end{proof}

\subsection{Symmetric quartics at infinity}\label{sec:symquartics}
For quartics, we present test sets for verifying nonnegativity and being a sum of squares in Theorem \ref{thm:testset}. This allows us to construct a symmetric nonnegative quartic that is not a sum of squares for any number of variables $n\ge4$ (Theorem \ref{thm:limit quartic}).
\smallskip

\begin{theorem}\label{thm:testset}
Let $f(\p_1,\p_2,\p_3,\p_4)$ be a homogeneous even symmetric function of degree $4$. Then, $f$ is nonnegative in all number of variables if and only if $f$ is nonnegative on the discrete set of parallel lines $$\{(x,1,n^{-1/2},n^{-1})\,\,|\,\,x\in\R,\, n\in\N_{>0}\}.$$ Moreover, $f$ is a sum of squares in all number of variables if and only $f$ is nonnegative on $$\{(x,1,u,u^2)\,\,|\,\,x\in\R,\, 0\le u\le1\}.$$
\end{theorem}

\begin{theorem}\label{thm:limit quartic}
The symmetric quartic $ 4p_1^4-5p_2p_1^2-\frac{139}{20}p_3p_1+4p_2^2+4p_4$ is nonnegative in all number of variables. Moreover, these quartic forms are not sums of squares for any number of variables $n\ge4$.
\end{theorem}

From Theorem \ref{thm:limit quartic}, we deduce that $\S\Sigma_4 \subsetneq \S\P_4$.
\smallskip

We denote by $\PSD_k$ the set of $k\times k$ real symmetric matrices that are positive semidefinite.
\begin{definition}\label{def:psdprojection}
    Let $R$ be an $\R$-vector space, $M_i\in R^{r_i\times r_i},i=1,\dots,k$, i.e. $M_i$ is an $r_i\times r_i$ matrix filled with elements of $R$. The \emph{psd projection with respect to} $M_1,\dots,M_k$ is 
    \begin{align*}
        \Pi(M_1,\dots,M_k)=\left\{\sum\Tr(A_iM_i) : A_i\in \PSD_{r_i}\right\}.
    \end{align*}
Note that $\Tr(AM)$ is simply $\sum a_{ij}m_{ij}$ and is thus an element of $R$. Therefore $\Pi(M_1,\dots,M_k)\subseteq R$. For us the vector space $R$ will always be a vector space of functions, so that the product $\Tr(A_iM_i)$ is a function in $R$.
\end{definition}

From Example \ref{ex:sos4} we have that 
$$\S\Sigma_4=\Pi\left( \begin{bmatrix}
    \p_{(1^4)} & \p_{(2,1^2)} \\
    \p_{(2,1^2)} & \p_{(2^2)}
\end{bmatrix}, \begin{bmatrix}
    \p_{(2,1^2)} & \p_{(3,1)} \\
    \p_{(3,1)} & \p_{(4)}
\end{bmatrix},  \p_{(2^2)}-\p_{(4)} \right),$$ 
so the dual cone $\S\Sigma_4^*$ is the spectrahedron containing precisely the vectors $(a,b,c,d,e)\in\R^5$ such that $X=A\oplus B\oplus C\succeq0$ where $A=\begin{bmatrix}a & b\\ b & d\end{bmatrix}$, $B=\begin{bmatrix}b & c\\ c & e\end{bmatrix}$, and $C=d-e$. 
\smallskip

It is well known that the dual of the cone of sums of squares can be identified with a set of positive semidefinite quadratic forms. The second and fourth author show in \cite{blekherman2021symmetric} that an analysis of the extremal rays in the symmetric case can be done similarly to the general case \cite{blekherman2012nonnegative,blekherman2017extreme}. The extremal rays of $\S\Sigma_4^*$ correspond to positive semidefinite quadratic forms with maximal kernel, i.e., to those vectors $(a,b,c,d,e) \in \R^5$ for which $X$ is positive semidefinite and the quadratic form has maximal kernel (\cite[Proposition 4.20]{blekherman2021symmetric}).

\begin{proposition}\label{sos4xrays}
Any extreme ray of $\S\Sigma_4^*$ is spanned by a point in
\begin{align*}
S=\{(1,t^2,st^2,t^4,s^2t^2)\,|\, t\ge0,\, s\in[-t,t]\}\cup\{(0,0,0,1,0)\}\cup\{(0,0,0,1,1)\},
\end{align*}
and every point in $S$ spans an extreme ray of $\S\Sigma_4^*$.
\end{proposition}
\begin{proof}

We use the above spectrahedral representation of $\S\Sigma_4^*$ and \cite[Corollary 4]{ramana1995some} (see also \cite[Prop.~4.20]{blekherman2012nonnegative}) to find the extreme rays of $\S\Sigma_4^*$.
\smallskip

Let $\ker(a,b,c,d,e)$ denote the kernel of the matrix obtained by evaluating $X$ at the point $(a,b,c,d,e)$. Suppose $(a,b,c,d,e)$ spans an extreme ray of $\S\Sigma_4^*$. We begin with a case distinction on the rank of $X$. 
\medskip

If $\rank A=2$ then $\ker X\subseteq(0,0,*,*,*)$, but $\ker(1,0,0,0,0)=(0,*,*,*,*)$, so $\ker X$ is not maximal.

If $\rank A=0$ then $a=b=d=0$, so $c=0$ because $B\succeq0$, and $e=0$ because $d-e\ge0$ and $e\ge0$ ($B\succeq0$). Hence $\rank A=1$.
\medskip

Case $\rank A=1$ and $\rank B=2$. This is impossible since $\ker X\subsetneq\ker(0,0,1,0,0)=(*,*,0,0,*)$.
\medskip

Case $\rank A=1$ and $\rank B=0$. If $\rank B=0$ then $b=c=e=0$. Hence, since $b=0$ and $\rank A=1$ then either $a=0$ or $d=0$. If $d=0$ we have the extreme ray spanned by $(1,0,0,0,0)$, because $\ker(1,0,0,0,0)=(0,*,*,*,*)$ is maximal. If $a=0$ we have the extreme ray spanned by $(0,0,0,1,0)$. Even though $\ker(0,0,1,0,0)=(*,0,*,*,0)$ is 3-dimensional, one can check that the space of matrices $X$ containing such a kernel is 1-dimensional, implying the maximality of the kernel. For instance, consider a vector $[u,0,v,w,0]^\top$ in the kernel of $X$ with non-zero $u,v,w$, then
\begin{align*}
    \begin{bmatrix}
        a & b & 0 & 0 & 0\\
        b & d & 0 & 0 & 0\\
        0 & 0 & b & c & 0\\
        0 & 0 & c & e & 0\\
        0 & 0 & 0 & 0 & d-e
    \end{bmatrix}
    \begin{bmatrix}
        u\\
        0\\
        v\\
        w\\
        0
    \end{bmatrix}=\begin{bmatrix}
        0\\
        0\\
        0\\
        0\\
        0
    \end{bmatrix}
\end{align*}
so the first two rows of the matrix product give $au=0$ and $bu=0$, so $u\ne0$ gives $a=b=0$. The next two rows give $cw=0$ and $cv+ew=0$, so $vw\ne0$ give $c=e=0$. The last row gives no information. Hence the space of matrices $X$ with such a kernel is spanned by $X(0,0,0,1,0)$, so it is 1-dimensional. Therefore the space of matrices $X$ with a larger kernel is 0-dimensional, which shows $\ker(0,0,0,1,0)$ is maximal.
\medskip

Case $\rank A=1$ and $\rank B=1$. If $b=0$ then from $B\succeq0$ we get $c=0$, and from $A$ either $a=0$ or $d=0$. In the first case then $d\ne0$ and $d=e$ will give a maximal kernel. The maximality of $\ker(0,0,0,1,1)=(*,0,*,0,*)$ follows as in the previous case, so $(0,0,0,1,1)$ spans a extreme ray. In the second case, from $C\succeq0$ we get $e=0$, so $\rank B=0$, impossible.
If $a=0$ then from $A\succeq0$ we get $b=0$ and the same follows. So we can assume from now on that $a$ and $b$ are nonzero.
\medskip

From $\det A=\det B=0$ we obtain $d=b^2/a$ and $e=c^2/b$. So $(a,b,c,d,e)=(a,b,c,\frac{b^2}a,\frac{c^2}b)$, and dividing by $a$ we obtain $(1,t,\frac{c}bt,t^2,\frac{c^2}{b^2}t)$ where $t=\frac{b}a$. Setting $s=\frac{c}b$ we obtain $(1,t,st,t^2,s^2t)$. If $C=0$ then $d=e$ and so $b^3=ac^2$ or $\frac{b}a=\frac{c^2}{b^2}$, i.e., $t=s^2$, and so we obtain the family of extreme rays $(1,s^2,s^3,s^4,s^4)$ for $s\in\R\setminus\{0\}$. To show the maximality of the kernel just observe that $\ker(1,s^2,s^3,s^4,s^4)=\{(s^2u,-u,sv,-v,w):u,v,w\in\R\}$, then select nonzero $u,v,w$ and proceed as before.

If $C>0$ then $d>e$ or $t>s^2$, and we get the family of extreme rays $(1,t,st,t^2,s^2t)$ with $t>0$ and $s\in(-t^{1/2},t^{1/2})$. Again, to show the maximality of the kernel observe that $\ker(1,t,st,t^2,s^2t)=\{(tu,-u,sv,-v,0):u,v\in\R\}$ and select nonzero $u,v$. Matrices $X$ with such a kernel must then satisfy $atu-bu=0$ and $btu-du=0$, so $b=at$ and $d=bt=at^2$. Also $bsv-cv=0$ and $csv-ev=0$, so $c=bs=ast$ and $e=cs=as^2t$. Hence $(a,b,c,d,e)=a(1,t,st,t^2,s^2t)$, and so the space of matrices $X$ with such a kernel is 1-dimensional. We thus have found all the extreme rays of $\S\Sigma_4^*$.
\end{proof}

 Kostov \cite{kostov2007stably} describes the set $\mathcal{M}_4$ but in the coordinates of elementary symmetric polynomials. By $\Pi_d(\infty)$ Kostov denotes the closure of the set $ \bigcup_{n \geq 1}(e_1,e_2,\ldots,e_d)(\R^n)$ called the set of \textit{stably hyperbolic} polynomials of degree $d$. 

 His paper focuses on the parametrization of the boundary of $\Pi_4(\infty) \cap \{x \in \R^4 : x_1 = 0,x_2 =-1\}$. We translate Kostov's $\Pi_4(\infty)$ to our setting with power sums via
 Newton's identities (listed in the proof of Theorem \ref{thm:testset} below), which provide a diffeomorphism $$ \mathcal{M}_4 \cap \{x \in \R^4 : x_1 = 0,x_2 =2\} \cong \Pi_4(\infty) \cap \{x \in \R^4 : x_1 = 0,x_2 =-1\}.$$ 

With Kostov's results on $\Pi_4(\infty)$ and the description of the dual cone $\S\Sigma_4^*$ in Proposition \ref{sos4xrays} we prove Theorem \ref{thm:testset}.

\begin{proof}[Proof of Theorem \ref{thm:testset}]
Let $\mathcal{K}$ be \emph{Kostov's leaf}, i.e. the projection onto the last two coordinates of $\Pi_4(\infty) \cap \{x \in \R^4 : x_1 = 0,x_2 =-1\}$ (\cite[Fig. 2]{kostov2007stably}).
\medskip

Let $\mathcal{K}'$ be the projection of $\mathcal{M}_4\cap\{x\in\R^4:x_2=2\}$ onto the last two coordinates. By Remark \ref{rem:weightprism} $\mathcal{K}'$ is the same as the projection of $\mathcal{M}_4\cap\{x\in\R^4:x_1=0,x_2=2\}$ onto the last two coordinates.
\medskip

We show that $\mathcal{K}'$ is the image of $\mathcal{K}$ under an invertible affine linear map. We have $p_1=e_1$ by definition and by Newton's identities
\begin{align*}
    p_2&=e_1^2-2e_2\\
    p_3&=e_1^3-3e_2e_1+3e_3\\
    p_4&=e_1^4-4e_2e_1^2+4e_3e_1+2e_2^2-4e_4
\end{align*}
and so restricting to $\mathcal{K}$, i.e., $e_1=0$ and $e_2=-1$, we have $p_1=0$ and
\begin{align*}
    p_2&=2\\
    p_3&=3e_3\\
    p_4&=2-4e_4
\end{align*}
therefore the map $(x,y)\mapsto(3x,2-4y)$ sends $\mathcal{K}$ to $\mathcal{K'}$. Since this is an invertible affine linear map then the extreme points of $\conv(\KK)$ are in bijection with the extreme points of $\conv(\KK')$.
\medskip

The extreme points of $\conv(\mathcal{K})$ are precisely the cusps of the arcs of $\mathcal{K}$ (see \cite[Remark 11]{kostov2007stably}). Observe that $\mathcal{K}$ has countably infinitely many cusps, so $\conv(\mathcal{K})$ has countably infinitely many vertices.
\medskip

The extreme points of $\conv(\mathcal{K}')$ are the images of the extreme points of $\conv(\mathcal{K})$ under the above map. By (\cite[p.~102]{kostov2007stably}) the extreme points of $\conv(\mathcal{K})$ are 
\begin{align*}
\left\{\left(\pm \frac{2}{3}\sqrt{\frac{2}{s}},\frac{1}{2}-\frac{1}{s}\right) : s \in \N_{>0}\right\},
\end{align*}
and so the extreme points of $\conv(\mathcal{K}')$ are
\begin{align*}
\mathcal{T'} = \{(\pm 2\sqrt{2}n^{-1/2},4n^{-1}) : n \in \N_{>0}\}.
\end{align*}
Now, let $\LL$ be the projection onto the last two coordinates of $\mathcal{M}_4\cap\{x\in\R^4:x_2=1\}$. Since $\mathcal{M}_4$ is weighted homogeneous, then the extreme points of $\conv(\LL)$ are 
\begin{align*}
\mathcal{T} = \{(\pm n^{-1/2},n^{-1}) : n \in \N_{>0}\}.
\end{align*}
This observation is enough to provide a test set for nonnegativity for limit symmetric quartics.
\medskip

Restricting to $\p_2=1$ by homogeneity, and since the first coordinate of $\mathcal{M}_d$ is free by Remark \ref{rem:weightprism}, a limit symmetric quartic $$f:=c_1 \p_1^4+c_2\p_2\p_1^2+c_3\p_3\p_1+c_4\p_2^2+c_5\p_4$$ is nonnegative if and only if each univariate polynomial in the 
family
\begin{align*}
    \mathscr{F}':=\{c_1x^4+c_2x^2+c_3p_3x+c_4+c_5p_4\,\,|\,\, (p_3,p_4)\in\LL\}
\end{align*}
is nonnegative (the case $\p_2=0$ can be disregarded because $\p_2^2\ge \p_4\ge0$, so $\p_4=0$, and $\p_3=0$ because $\p_2\p_4\ge \p_3^2$ and in such case the nonnegativity of $c_1\p_1^4$ is implied by the nonnegativity of the polynomials in the above family). Now, since any convex combination of nonnegative polynomials is nonnegative then $f$ is nonnegative if and only if each univariate polynomial in the family $\conv(\mathscr{F}')$ is nonnegative. Moreover, since the coefficients of the polynomials in this family depend linearly on $p_3$ and $p_4$ then $f$ is nonnegative if and only if each univariate polynomial in the family
$$\mathscr{F}=\{c_1x^4+c_2x^2+c_3p_3x+c_4+c_5p_4\,\,|\,\, (p_3,p_4)\in \mathcal{T}, p_3 > 0\}$$ is nonnegative.  It is sufficient to restrict to the test set $\{ (n^{-1/2},n^{-1}) : n \in \N_{>0}\}$ because polynomials in $\mathscr{F}$ have only one odd degree term, and a polynomial $g(x)$ is nonnegative if and only if $g(-x)$ is nonnegative.

Finally, a symmetric quartic function is in $\S\Sigma_4$ if and only if its evaluation on all extreme rays of $\S\Sigma_4^*$ is nonnegative. By Proposition \ref{sos4xrays}, the extreme rays of $\S\Sigma_4^*$ are spanned by the points in $\{(1,t^2,st^2,t^4,s^2t^2)\,|\, t\ge0,\, s\in[-t,t]\}\cup\{(0,0,0,1,0)\}\cup\{(0,0,0,1,1)\}$. By continuity of the power sums we can reduce to the partially scaled subset $\{(\frac{1}{t^4},\frac{1}{t^2},\frac{s}{t^2},1,\frac{s^2}{t^2})\,|\, t>0,\, s\in[-t,t]\}\cup\{(0,0,0,1,0)\}\cup\{(0,0,0,1,1)\}$. Setting $x=\frac{1}{t}$ and $u=sx$ we obtain the set $\{(x^4,x^2,ux,1,u^2)\,|\, x>0,\, u\in[-1,1]\}\cup\{(0,0,0,1,0)\}\cup\{(0,0,0,1,1)\}$. Again, by continuity of power sums an equivalent set is $\{(x^4,x^2,ux,1,u^2)\,|\, x \in \R,\, u\in[0,1]\}$ which contains the points $(0,0,0,1,0),(0,0,0,1,1)$. We can deduce the claim on the test set $\p_1=x,\p_2=1,\p_3=u,\p_4=u^2$ by observing that any point $(x^4,x^2,xu,1,u^2)$ is the image of $(x,1,u,u^2)$ under the monomial map $\nu_4$.
\end{proof}

We now prove Theorem \ref{thm:limit quartic}, i.e. $f := 4\p_{(1^4)}-5\p_{(2,1^2)}-\frac{139}{20}\p_{(3,1)}+4\p_{(2^2)}+4\p_{(4)}\in\S\P_4\setminus\S\Sigma_4$. The corresponding forms are nonnegative in any number of variables. Moreover, these forms are never sums of squares for any number of variables $n\ge4$.

\begin{proof}
Suppose $F(p_1,p_2,p_3,p_4)=ap_1^4+bp_2p_1^2+cp_3p_1+dp_2^2+p_4$ is nonnegative but not a sum of squares for any number of variables $n\ge4$. By Theorem \ref{thm:testset}, the quadratic polynomial $g_x(u)$ that results from evaluating $F$ at $(x,1,u,u^2)$, $$g_x(u)=u^2+cxu+q(x)$$ where $q(x)=ax^4+bx^2+d$, must be nonnegative at $u=n^{-1/2}$ for all $n\in\N_{>0}$ for all $x\in\R$, but negative for some pair $(u,x)$ with $u\in[0,1]$. 

We find such a polynomial by restricting our search. Suppose that for each $x\in\R$, $g_x(u)$ has roots either in $[\frac1{\sqrt{2}},1]$ or $[-1,-\frac1{\sqrt{2}}]$, or has no real roots. 

Let $\Delta(x):=c^2x^2-4q(x)^2$ and $S_{\Delta}:=\{x\ge0\,|\,\Delta(x)\ge0\}$. We want $\sqrt{2}\le -cx\pm\sqrt{\Delta(x)}\le 2$ for all $x\in S_{\Delta}$. 

We find that this happens for $(a,b,c,d)=(1,-5/4,-139/80,1)$, or, rescaling, for $(a,b,c,d,e)=(4,-5,-139/20=-6.95,4,4)$ which we verify below. 

Write $f_n$ for $f$ considered as quartic form in $n$ variables. We have
\begin{align*}
g_x(u)=4u^2-6.95xu+4x^4-5x^2+4.
\end{align*}
Observe $g_1(0.9)=-0.015<0$ and so, by Theorem \ref{thm:testset}, $f\not\in \S\Sigma_4$. Thus $f_n\not\in\Sigma_{n,4}^{\mathcal{S}}$ for all $n \geq N$ for some $N\ge4$, but already $f\not\in\Sigma_{4,4}^{\mathcal{S}}$ as verified by the SumsOfSquares package \cite{cifuentes2020sums} in Macaulay2 \cite{M2}. 
\smallskip

We claim that for each $x\in\R$, $g_x(u)$ has roots either in $[\frac1{\sqrt{2}},1]$ or $[-1,-\frac1{\sqrt{2}}]$, or has no real roots. Hence, since the leading coefficient of $g_x(u)$ is positive, each $g_x(u)$ would be nonnegative on $\{n^{-1/2}:n\in\N_{>0}\}$, and Theorem $\ref{thm:testset}$ would imply $f\in\S\P_4$.
\smallskip

It suffices to find, or bound, the real range of the functions
\begin{align*}
    r_1(x)&:=6.95 x-\sqrt{\Delta(x)}\\
    r_2(x)&:=6.95 x+\sqrt{\Delta(x)}
\end{align*}
where $\Delta(x):=-64x^4+128.3025x^2-64$ is the discriminant of $g_x(u)$.
\smallskip

Let $a<b$ be the positive roots of $\Delta(x)$. Since the leading coefficient of $\Delta(x)$ is negative then the set of $x\ge0$ such that $\Delta(x)\ge0$ is the interval $[a,b]$.
\smallskip

We restrict to $x\ge0$, because the roots of $g_x(u)$ are the negatives of the roots of $g_{-x}(u)$, and then prove that $4\sqrt{2}\le r_i(x)\le8$ for $x\in[a,b]$, which is equivalent to our claim.
\smallskip

Suppose $x\in[a,b]$. Observe $[a,b]\subset[0.96,1.04]$ and there we have $4\sqrt{2}-6.95x<0$ and $8-6.95x>0$. Then $4\sqrt{2}-6.95x\le\pm\sqrt{\Delta(x)}\le8-6.95x$ follows from both 
\begin{align*}    
\Delta(x)&\le (4\sqrt{2}-6.95x)^2\\
\Delta(x)&\le (8-6.95x)^2
\end{align*}
which after expanding we obtain the following, which can be verified to be globally nonnegative
\begin{align*}
    64x^4-80x^2-\frac{278\sqrt{2}}5x+96&\ge0\\
    64x^4-80x^2-\frac{556}5x+128&\ge0.
\end{align*}
\end{proof}
So we have the following, in contrast to the case of symmetric means where the cones of symmetric nonnegative quartics and sum of squares quartics are equal \cite[Theorem~2.9]{blekherman2021symmetric}.

\begin{corollary}
    \label{thm:S4 and P4}
$\S\Sigma_4\subsetneq \S\P_4$.
\end{corollary}

\begin{example}
The limit form $4\p_{(1^4)}-5\p_{(2,1^2)}-4\sqrt{3}\p_{(3,1)}+4\p_{(2^2)}+4\p_{(4)}$ is contained in the boundary of $\S\Sigma_4$ but not in the boundary of $\S\P_4$.
To see this we consider \begin{align*}
    g_x(u) &=4u^2-4\sqrt{3}xu+4x^4-5x^2+4  \\
     & = (2u-\sqrt{3}x)^2+4(x^2-1)^2
\end{align*}
and observe $g_x(u) = 0$ if and only if $x = \pm 1$ and $u = \pm \frac{\sqrt{3}}{2}$. However, this shows that $g_x(u)$ is strictly positive for any $u\in \mathcal{T}$, but attains $0$ at $u = \frac{\sqrt{3}}{2} \in \left(\frac{1}{\sqrt{2}},1\right)$.
\end{example}
We present another, real algebraic, reason for the disparity between the limit cones of quartics.
\begin{proposition}
The cone $\S\P_4$ is not semialgebraic. Moreover, for each positive integer $n$ there is an isolated family of extremal rays of $\S\P_4^*$.
\end{proposition}
\begin{proof}
We show that the dual cone $\S\P_4^*$ is not semialgebraic. Consequently, by Lemma \ref{lem:dual cone is semialgebraic}, the cone $\S\P_4$ cannot be semialgebraic.

Let $\mathcal{X}:=\{ (x^4,x^2,\frac{x}{\sqrt{n}},1,\frac{1}{n}) \, \, \mid \, \, x \in \R, n \in \N_{>0}\} \subset \R^5$. By Theorem \ref{thm:testset}, 
$$\S\P_4^*=\operatorname{cone} \operatorname{cl} \{(p_1^4,p_2p_1^2,p_3p_1,p_2^2,p_4)(x) \, \, \mid \, \, x \in \R^n, n \in \N_{>0}\} = \cone \mathcal{X}.$$
Thus the infinitely many points in $\mathcal{X}$ must contain a generator of any extremal ray of the cone $\S\P_4^*$. 
We claim that any subset consisting of finitely many points in $\mathcal{X}$ is in convex position. This implies that any element in $\mathcal{X}$ must span an extremal ray.
To prove the claim we assume that there exists a convex combination
\[\sum_{i=1}^N\lambda_i \left(x_i^4,x_i^2,\frac{x_i}{\sqrt{n}_i},1,\frac{1}{n_i}\right) = \left(x^4,x^2,\frac{x}{\sqrt{n}},1,\frac{1}{n}\right)\]
of elements in $\mathcal{X}$. Recall that any subset of distinct points on the real moment curve $t \mapsto (t,t^2)$ is in convex position (\cite[Prop.~0.7]{ziegler2012lectures}). Looking at the first two coefficients of the elements of any point we obtain that all $x_i$ are equal up to sign. 
Then we consider the third and last coordinates in the convex combination and conclude $\sum_{i=1}^N \lambda_i (\frac{1}{\epsilon_i\sqrt{n}_i},\frac{1}{n_i})=(\frac{1}{\sqrt{n}},\frac{1}{n})$, where $\epsilon_i :=1$ if $x_i = x$, and $\epsilon_i := -1$ otherwise. Again, these are points on the moment curve $(t,t^2)$. 
This implies $n_i = n$ for all $1 \leq i \leq n$ and either all $\epsilon_i$ are $1$ or $-1$. 

Therefore, the set $\S\P^*$ has infinitely many isolated extremal rays. 
Thus the set $\S\P_4^*$ cannot be semialgebraic (see \cite[Corollary 2.20]{acevedo2023wonderful}).
\end{proof}

\subsection{Limit cones in higher degrees} \label{sec:non semialgebraic}
 We show that for all even degree $2d \geq 6$ the limit sets of nonnegative (even) symmetric forms are not semialgebraic and conclude strict inclusion between the limit sets of (even) symmetric sums of squares and nonnegative forms of degree $2d$. \smallskip

\begin{theorem} \label{thm:Limitcones not equal}
For all $2d \geq 4$ we have $\S\Sigma_{2d}\subsetneq \S\P_{2d}$ and for all $2d \geq 6$ we have $\B\Sigma_{2d} \subsetneq \B\P_{2d}$. Moreover, the sets $\S\P_{2d}$ and $\B\P_{2d}$ are non-semialgebraic in all these cases, and the sets $\S\Sigma_{2d}$ and $\B\Sigma_{2d}$ are semialgebraic.
\end{theorem}
For any other degree we have equality by Hilbert's theorem from 1888 and since  $\Sigma_{n,4}^{\mathcal{B}} = \P_{n,4}^{\mathcal{B}}$ for all $n$ \cite{harris1999real}.
In the proof of Theorem \ref{thm:Limitcones not equal} we use the non-semialgebraicness of the Vandermonde cell $\mathcal{N}_{d}$ for all $d \geq 3$ (\cite{acevedo2023wonderful}, Corollary 2.20) and we present the proof at the end of this subsection.

\begin{lemma} \label{lem:dual cone is semialgebraic}
The dual cone of a semialgebraic set $S \subset \R^n$ is semialgebraic, i.e., the set $S^* = \{ a \in \R^n : \sum_{i=1}^n a_i w_i  \geq 0, \forall w \in S\}$ is semialgebraic.
\end{lemma}
\begin{proof}
We observe \begin{align*}
    S^* & = \{ a \in \R^n : \exists w \in S, \sum_{i=1}^n a_i w_i < 0\}^c  = \pi \left(\{ (a,w) \in \R^n \times \R^n : w \in S, \sum_{i=1}^n a_i w_i <0 \} \right)^c,
\end{align*}
where $\pi : \R^n \times \R^n \rightarrow \R^n$ denotes the projection onto the first $n$ coordinates. The semialgebraicness of $S^*$ follows from the Tarski-Seidenberg theorem.
\end{proof}

\begin{proposition}\label{prop:even symmetric non-semialgebraic}
For any even degree $2d \geq 6$ the cones $\S\P_{2d}$ and $\B\P_{2d}$ are not semialgebraic.
\end{proposition}
\begin{proof}
By \cite[Theorem 3.7]{choi1987even} an $n$-variate even symmetric sextic $f(p_2,p_4,p_6)=ap_2^3+bp_2p_4+cp_6$ is nonnegative if and only if $f(t,t,t)\ge0$ for all $t\in[n]$. Equivalently, $f\in\P_{n,6}^{\mathcal{B}}$ if and only if \linebreak $(a,b,c)\cdot(t^3,t^2,t)\ge0$ for all $t\in[n]$, and so $\P_{n,6}^{\mathcal{B},*}=\cone\{(t^3,t^2,t):t\in[n]\}$. Hence $\B\P_6^*=\cone\{(t^3,t^2,t)\,\,\mid\,\,t\in\Z_{>0}\}$, but this cone has countably many isolated extreme rays $\{(t^3,t^2,t):t\in\Z_{>0}\}$, and so it cannot be semialgebraic. Note $\B\P_6^*=\cone(p_2^3,p_2p_4,p_6)=\cone(1,p_4,p_6)$.
\medskip

Now let $d\ge4$. If $\B\P_{2d}$ is semialgebraic then also its intersection with the linear subspace $\langle p_2^d, p_2^{d-2}p_4, p_2^{d-3}p_6 \rangle$ would be semialgebraic.

Letting $p_2=1$, $\cone \clo \{(1,p_4,p_6)(x) \, \, \mid \, \, x \in \bigcup_{n \geq 1}\R^n, p_2(x) = 1\}$ would be semialgebraic, but this is precisely $\B\P_6^*$, which is not semialgebraic, a contradiction. 
\medskip

Finally, since $\B\P_{2d}$ is not semialgebraic then neither is $\S\P_{2d}$, since the intersection of a semialgebraic set with a linear subspace is semialgebraic.
\end{proof}

We obtain strict inclusion between (even) symmetric sums of squares and nonnegative limit forms for degree $2d \geq 6$.

\begin{proof}[Proof of Theorem \ref{thm:Limitcones not equal}]
For any degree $2d$ the sets $\S\Sigma_{2d}^*$ and $\B\Sigma_{2d}^*$ are spectrahedral cones (see Proposition \ref{prop:spectrahedral cones}). Thus their duals, $\S\Sigma_{2d}$ and $\B\Sigma_{2d}$, are semialgebraic by Lemma \ref{lem:dual cone is semialgebraic}. However, $\S\P_{2d}$ and $\B\P_{2d}$ are not semialgebraic for any $d\ge3$ by Proposition \ref{prop:even symmetric non-semialgebraic} which proves the claim.
\end{proof}

\section{Tropicalization of $\mathcal{N}_d$ and proof of Theorem \ref{thm: power sums order}} \label{sec:tropicalization}

Our main result in this section is a full description of the tropicalization of the Vandermonde cell $\mathcal{N}_d$ (Theorem \ref{thm:trop(Nd)}). 
Utilizing the polyhedral representation of $\trop(\mathcal{N}_d)$, we then prove Theorem \ref{thm: power sums order}, i.e. that valid pure binomial inequalities in power sums on the nonnegative orthant are captured by the superdominance order. \medskip

The \emph{tropicalization}, or \emph{logarithmic limit}, of a set $S\subset\R^s$ has often been called a \emph{combinatorial shadow} of it, given some of the structural properties it retains from $S$ while being an apparently simpler geometrical object. \smallskip

For a scalar $a >0$ we consider the \emph{logarithm map} \[\log_a : \R_{>0}^s \to \R^s, {(x_1,\ldots,x_s)} \mapsto {\left(\log_a(x_1),\ldots,\log_a(x_s) \right)}.\] We write $\log := \log_e$ and $\trop : \R_{>0}^s \rightarrow \R^s$ denotes the \emph{tropicalization} which sends $S\subseteq\R_{>0}^s$ to the limit set $\lim_{t \rightarrow \infty} \log_t(S)$. We extend this definition for sets $S\subset \R_{\geq 0}^s$ where $\log \left( S \right) := \log \left( S \cap \R_{>0}^s\right)$ and $\trop (S) := \trop (S \cap \R_{>0}^s)$. The set $\trop (S)$ is a closed cone for any set $S \subset \R_{\geq 0}^s$ and $y \in \trop (S)$ if and only if there exist sequences $(x^{(n)}) \in S \cap \R_{> 0}^s$ and $(\tau_n) \in (0,1)$ with $\tau_n \to 0$ for $n \to \infty$ such that $y = \lim_{n\to \infty}\log_{\frac{1}{\tau_n}}(x^{(n)})$ (\cite[Proposition 2.2]{alessandrini2013logarithmic}). \smallskip 

\begin{definition}
A set $S\subseteq\R^n$ has \emph{the Hadamard property} if whenever $u$ and $v$ belong to $S$ then their Hadamard product $u\circ v:=(u_1v_1,\dots,u_nv_n)$ also lies in $S$.
\end{definition}

In other words, a set has the Hadamard property if it is closed under coordinatewise multiplication. 
\smallskip

Given its logarithmic nature, tropicalization captures \emph{pure binomial inequalities} valid on subsets of the nonnegative orthant, i.e. if $$x^\alpha=\prod_{i=1}^s x_i^{\alpha_i}\ge x^\beta=\prod_{i=1}^s x_i^{\beta_i}$$ is valid on $S\subset\R_{\ge0}^s$, then the linear inequality $$\alpha\cdot y=\sum_{i=1}^s\alpha_i y_i\ge \beta\cdot y=\sum_{i=1}^s\beta_i y_i$$ is valid on $\trop(S)$. Furthermore, when $S$ has the Hadamard property, $\trop(S)$ is a closed convex cone that is determined by the pure binomial inequalities valid on $S$ \cite[Lemma 2.2]{blekherman2022tropicalization}. More precisely, the extreme rays of the dual cone $\trop(S)^*$ generate all of the pure binomial inequalities valid on $S$ \cite{blekherman2022path}. This idea was recently fruitfully applied in extremal graph theory \cite{blekherman2022tropicalization, blekherman2022path} where it was used to understand pure binomial inequalities in homomorphism numbers and densities from the tropicalizations of graph homomorphism profiles. It was also used in real algebraic geometry to study the cones of nonnegative and sums of squares polynomials \cite{blekherman2022moments, acevedo2024power}.
\medskip

The set $\mathcal{N}_d$ is closed under addition and has the Hadamard property (Lemma \ref{lem:newton}). A compact level set of it has the structure of a $d$-dimensional \emph{infinite curvy cyclic polytope} (see \cite{acevedo2023wonderful} for context). However, its tropicalization turns out to have a much simpler description. 

\begin{remark}\label{rem:classical ieqs}
We show in Theorem \ref{thm:trop(Nd)} that all binomial inequalities valid on $\mathcal{N}_d$ come from two families of classical inequalities on power sums, namely
\begin{enumerate}
    \item $p_{2k}^{k+1}\ge p_{2k+2}^k$ and
    \item $p_{2k}\cdot p_{2k+4}\ge p_{2k+2}^2$ for all positive integer $k$.
\end{enumerate}
The former comes from the monotonicity of $\ell^p$-norms and the latter from Lyapunov's inequality (\cite[Sections 2.9 and 2.10]{hardy1952inequalities}).
\end{remark}

Denote by $\oplus$ the \textit{tropical addition}, i.e. for $x,y \in \R^n $, \[x \oplus y := \left(\max{(x_1,y_1)}, \ldots,\max{(x_n,y_n)}\right),\]
and by $\odot$ the \emph{tropical scalar multiplication}, i.e. for $a \in \R$ and $x \in \R^n$ \[a \odot x :=(a+x_1,\ldots,a+x_n).\]

\begin{definition}
A set $M \subset \R^n$ is called \emph{max-closed} if $x \oplus y \in M$ for all $x,y \in M$. 
The \emph{max-closure} $\overline{M}$ of $M$ is the smallest max-closed set containing $M$. 
\end{definition}

Observe that the max-closure of a set $M \subset \R^n$ is the intersection of all max-closed subsets of $\R^n$ containing $M$. Since $\R^n$ is a max-closed set containing $M$ the max-closure $\overline{M}$ is well-defined.

\begin{definition}
Let $S \subset \R^n$. The \emph{double hull} of $S$ is defined as the smallest closed and max-closed convex cone containing $S$. We write $\dhu (S)$ for the double hull of $S$.
\end{definition}

\begin{theorem} \label{thm:trop(Nd)}
The set $\trop(\mathcal{N}_d)$ is the double hull of the vectors $(1,\dots,1)$, $(1,2,\dots,d)$ and $-(1,2,\dots,d)$. It is a rational polyhedral cone with the facet-defining inequalities: 
\begin{align*}
    \trop \left( \mathcal{N}_d \right) = \left\{ \left. y\in\R^d : \begin{cases}
    y_k + y_{k+2}\geq 2y_{k+1}\; ,\quad k=1,\ldots,d-2 \\
    dy_{d-1}\geq (d-1)y_d\; .
    \end{cases} \right\}  \right\}.
\end{align*}
Moreover, $\trop (\mathcal{N}_d)$ is the Minkowski sum of the linear span of $(1,2,\ldots,d)$ and the convex cone with the $d-1$ extremal rays $(1,0,\ldots,0),(2,1,0,\ldots,0),\ldots,(d-2,d-3,\ldots,1,0,0)$ and $(1,\ldots,1)$.
\end{theorem}
Before proving the Theorem we study some general properties of tropicalizations.

\subsection{Properties of max-closed sets}\label{sec:prop of trops of max-closed sets}

We prove some general properties of max-closed sets, and in particular, examine the structure of extremal rays of convex cones that are max-closed.  
\smallskip

In the following lemma, we provide a description of the max-closure of sets $M \subset \R^n$ which contain a vector $v \in \R_{> 0}^n$ in their \textit{lineality space}, i.e., $x + \lambda v \in M$ for all $\lambda \in \R$ and all $x \in M$. We denote by $\mathcal{Q}_i \subset \R^n$ the orthant where the $i$-th coordinate is non-positive and any other nonnegative, i.e., $$\mathcal{Q}_i := \{x \in \R^n \, : \, x_i \leq 0, x_j \geq 0\, ~\forall j \in [n]\setminus \{i\}\}.$$ 

The following is a variant of what has been called the \emph{tropical Farkas lemma} \cite[Lemma 2.6]{blekherman2022path}. 

\begin{lemma}\label{le:max-closure}
Let $M \subset \R^n$ and let $v \in \R_{> 0}^n$ be contained in the lineality space of $M$. Then, the max-closure of $M$ equals $ \bigcap_{i=1}^n (M+\mathcal{Q}_i)$.
\end{lemma}
We write $\widehat{M} := \bigcap_{i=1}^n (M+\mathcal{Q}_i)$. 
\begin{proof}
First, we prove that any $x \in \widehat{M}$ is the tropical sum of elements in $M$, i.e. $x = y_1 \oplus y_2 \oplus \ldots \oplus y_m$ for some $y_1,\ldots,y_m \in M$. Since $x \in \widehat{M}$, for $1 \leq i \leq n$ there exist $y_i=(y_{i1},\ldots,y_{in}) \in M$ and $g_i = (g_{i1},\ldots,g_{in})\in \Q_i$ such that $x = y_i + g_i$. Since $g_{ii} \leq 0, g_{ij} \geq 0$ for $j \neq i$ we have $y_1 \oplus \ldots \oplus y_n = (y_{11},\ldots,y_{nn})$ and $x_i \leq y_{ii}$. Suppose that for some $i \in [n]$ we actually have $y_{ii} > x_i$, and let $\lambda>0$ be such that $x_i = y_{ii} - \lambda v_i$. By assumption $\widetilde{y}_i := y_i - \lambda v \in M$ and $x = \widetilde{y}_i + g_i + \lambda v \in M + \Q_i$ since $g_{ii} + \lambda v_i = 0$ and any other coordinate is positive. Recursively, we obtain the existence of $y_1,\ldots,y_n \in M$ with $x= y_1\oplus\ldots \oplus y_n$. \\
Second, we prove that any tropical sum of elements in $M$ is contained in $\widehat{M}$. Let $y_1,\ldots,y_k \in M$ and $x:= y_1\oplus\ldots \oplus y_k$. For any $j \in [n]$ let $i_j \in [k]$ be such that $x_j = y_{i_jj}$. Since $x_k \geq y_{i_j k}$ for any $k \in [n]$ we have $x \in y_{i_j} + \Q_j$. In particular, $x \in \bigcap_{j=1}^n (M + \Q_j)$. \\
Finally, we show that $\widehat{M}$ is max-closed. For vectors $a,b \in \widehat{M}$ we have already seen that these vectors can be written as tropical sums of elements of $M$. Then $a \oplus b$ is the tropical sum of all those elements and lies in $\widehat{M}$, since $\widehat{M}$ is closed under tropical summation with elements in $M$. \\
Thus, $\widehat{M}$ is max-closed and the elements in $\widehat{M}$ are precisely the tropical sums of elements in $M$ which proves the claim.
\end{proof}

\begin{corollary}
Let $S \subset \R^n$ be a convex cone containing some vector $v \in \R_{>0}^n$ in its lineality space. Then, the max-closure of $S$ is a convex cone.
\end{corollary}
\begin{proof} 
The max-closure of $S$ is the intersection of convex cones by Lemma \ref{le:max-closure}. Thus the max-closure is again a convex cone.   
\end{proof}

We denote the all ones vector $(1,\ldots,1) \in \R^n$ by $\mathbf{1}$.
\begin{definition}\label{def:tconv}
Let $M \subset \R^n$ be a set. We define the \emph{tropical conical hull} of $M$ as the set of all tropical linear combinations of points in $M$. We write $\tcone(M)$ for the tropical conical hull of $M$, i.e., $$\tcone (M) = \{(a_1 \odot x_1) \oplus \ldots \oplus (a_l \odot x_l) \, \, \mid \, \, l \in \N, a_1,\ldots,a_l \in \R, x_1,\ldots,x_m \in M\}.$$ 
$M$ is called \emph{tropically convex} if $M$ equals its tropical conical hull.
\end{definition}

We refer to (\cite[Proposition 4]{develin2004tropical}) for a proof of the set theoretical description of $\tcone(M)$ in Definition \ref{def:tconv}.
The set $\tcone(M)$ is the max-closure of $M+\R \cdot \mathbf{1}$. This follows from the fact that  $Z:=M+ \R \cdot \mathbf{1}$ contains the vector $\mathbf{1}$ in its lineality space and $Z$ is the closure of $M$ with respect to tropical scalar multiplication. In particular, we have 
\begin{align}
    \tcone (M) = \bigcap_{j=1}^n (M+ \R \cdot \mathbf{1}  + \Q_j) \label{eq:k}
\end{align}

which follows from Lemma \ref{le:max-closure}.

\begin{corollary} \label{cor:dual of max-closure}
Let $M\subset \R^n$ be a closed convex cone containing a point from $\R_{>0}^n$ in its lineality space. Then, the linear inequalities characterizing the max-closure of $M$ are the linear inequalities valid on $M$ with exactly one negative coordinate.
\end{corollary}

\begin{proof}
By Lemma \ref{le:max-closure} we have \begin{align*}
    (\overline{M})^\ast  & = \left( \bigcap_{i=1}^n (M+\Q_i)\right)^\ast = \clo \left( \bigoplus_{i=1}^n M^\ast \cap \Q_i^\ast \right),
\end{align*}
where we use $(A+B)^*=A^* \cap B^*$ and $(A\cap B)^* = \clo (A^* +B^*)$ for closed convex cones $A,B \subset \R^n$ \cite[Corollary~p.~11]{fenchel1953convex} and $\Q_i^\ast = \Q_i.$ Finally, we use that $M$ contains a vector from $\R_{>0}^n$ in its lineality space to see that the defining linear equalities of $\overline{M}$ are those with exactly one negative coordinate.
\end{proof}
We will apply Corollary \ref{cor:dual of max-closure} to show that all extremal rays of $\trop (\mathcal{N}_d)^*$ have exactly one nonnegative coefficient in Corollary \ref{cor:trop N_d*}.
\begin{remark}
Analogous results for the min-closure and the tropical convex hull can be obtained by using min convention instead and defining tropical addition as coordinate-wise minimum.
\end{remark}

\subsection{The tropical Vandermonde cell} \label{section: Tropicalization of the even Vandermonde map at the limit}
We prove in Theorem \ref{thm:trop(Nd)} a uniform description of $\trop(\mathcal{N}_d)$ as a rational polyhedral cone and present a motivation for the defining linear inequalities. The proof uses methods from (\cite[Sec.~2.1]{blekherman2022path}). 
\medskip

We present some properties of $\mathcal{N}_d$ which allow us to apply the results of Subsection \ref{sec:prop of trops of max-closed sets}.

\begin{lemma}\label{lem:newton}
The sets $\mathcal{M}_d$ and $\mathcal{N}_d$ are closed under addition, and have the Hadamard property.
\end{lemma}
\begin{proof}
We show that the set $\mathcal{M}_d$ satisfies both properties and analogous arguments work for $\mathcal{N}_d$.

Let $u,v\in\bigcup_{n = 1}^\infty \nu_{n,d}(\R^n)$, so $u=\nu_{m,d}(x)$ and $v=\nu_{n,d}(y)$ for some $x\in\R^m$ and $y\in\R^n$.

Hence $u+v=\nu_{m+n,d}(x,y)\in\mathcal{M}_d$ where $(x,y)\in\R^{m+n}$ is the Cartesian product of $x$ and $y$. Hence the set $\bigcup_{n = 1}^\infty \nu_{n,d}(\R^n)$ is closed under addition, and, by continuity of the sum, also its closure $\mathcal{M}_d$.

Also $u\circ v=\nu_{mn,d}(x\otimes y)\in\mathcal{M}_d$, where $x\otimes y:=(x_1y_1,x_1y_2,\dots,x_1y_n,x_2y_1,\dots,x_my_n)\in\R^{mn}$ is the \emph{Kronecker product} of the row vectors $x$ and $y$. Hence the set $\bigcup_{k = 1}^\infty \nu_{k,d}(\R^k)$ has the Hadamard property, and, by continuity of the coordinatewise product, also its closure $\mathcal{M}_d$.
\end{proof}

\begin{lemma}\label{le:trop(Nd) is max-closed}
The set $\trop(\mathcal{N}_d)$ is a max-closed closed convex cone containing the line spanned by $(1,2,3,\ldots,d)$ and the ray spanned by the all ones vector $\mathbf{1}$. 
\end{lemma}

\begin{proof}
The set $\mathcal{N}_d$ has the Hadamard property and is closed under addition by Lemma \ref{lem:newton}. Since $\mathcal{N}_d \subset \R_{\geq 0}^d$ has Hadamard property, $\trop( \mathcal{N}_d)$ is a closed convex cone by (\cite[Lemma 2.2 (2)]{blekherman2022tropicalization}). Moreover, since $\mathcal{N}_d$ is closed under addition the set $\trop(\mathcal{N}_d)$ is max-closed by (\cite[Lemma 2.2]{blekherman2022path}). For $\lambda \in \R$ we have $$2\lambda (1,2,\ldots,d)= \lim_{m \rightarrow \infty} \log_\frac{1}{\tau_m} \left(\left(\frac{1}{\tau_m^{\lambda}}\right)^2,\left(\frac{1}{\tau_m^{\lambda}}\right)^4,\ldots,\left(\frac{1}{\tau_m^{\lambda}}\right)^{2d}\right)$$ for any sequence $(\tau_m)_m \subset \R_{>0}$ converging to $0$, and $$\mathbf{1}=\log_m (m,\ldots,m)=\log_m ((1^2+\ldots+1^2),(1^4+\ldots+1^4),\ldots,(1^{2d}+\ldots+1^{2d}))$$ for all $m \in \N$. Thus, we have inclusions $\R \cdot (1,2,3,\ldots,d)$ and $\R_{\geq 0} \cdot \mathbf{1} \subset \trop (\mathcal{N}_d)$.
\end{proof}
The following corollary will be useful in our proof of the characterization of $\trop (\mathcal{N}_d)$.
\begin{corollary} \label{cor:trop N_d*}
Every extreme ray of $\trop (\mathcal{N}_d)^\ast$ is spanned by a vector with exactly one negative coordinate. 
\end{corollary}

\begin{proof}
This follows directly from Corollary \ref{cor:dual of max-closure}. The set $\trop (\mathcal{N}_d)$ is a max-closed, closed convex cone which contains the line $\R \cdot (1,2,\ldots,d)$ by Lemma \ref{le:trop(Nd) is max-closed}. Since $\trop (\mathcal{N}_d)$ is max-closed it contains $(1,2,\ldots,d)$ in its lineality space. 
\end{proof}

We now present our proof of Theorem \ref{thm:trop(Nd)}. The case distinction in the proof goes analogously to (\cite[Theorem 2.15]{blekherman2022path}) which proves the tropicalization of graph profiles of sets of even cycles.
\begin{proof}[Proof of Theorem \ref{thm:trop(Nd)}]
Let $\mathcal{Q}$ denote the polyhedron on the right-hand side in the theorem. We first prove $\trop (\mathcal{N}_d) = \mathcal{Q}$ and divide its proof into three steps. \medskip 

\textbf{Claim 1:} $\trop \left( \mathcal{N}_d \right) \subset \mathcal{Q}$. \\
By (\cite[Lemma 2.2 (2)]{blekherman2022tropicalization}) we have $\trop (\mathcal{N}_d) = \clo \left( \mbox{conv} (\log (\mathcal{N}_d))\right)$ since $\mathcal{N}_d$ has Hadamard property. Thus, by taking logarithms the families of binomial inequalities in power sums 
\[ p_{2k}\cdot p_{2k+4}\geq p_{2k+2}^2 \mbox{ and } p_{2k}^{k+1} \geq p_{2k+2}^{k}\]
from Remark \ref{rem:classical ieqs} transfer to valid linear inequalities on $\log (\mathcal{N}_d)$ which are preserved under taking the convex hull. We have \[(k+1)\cdot (y_{k}-2y_{k+1}+y_{k+2})+((k+2)y_{k+1}-(k+1)y_{k+2})=(k+1)y_{k}-ky_{k+1}\] for all $k=1,\ldots,d-2$ which shows that $\Q$ is the polyhedron defined by the linearization of the binomial inequalities above. Thus, we have $\trop (\mathcal{N}_d) \subset \Q$. \medskip

\textbf{Claim 2:} The rays $\mathbf{1}$, $\balpha:=(1,2,\ldots,d)$, $-\balpha$ are contained in $\trop \left( \mathcal{N}_d \right)$. \\ This was shown in Lemma \ref{le:trop(Nd) is max-closed}. \medskip

\textbf{Claim 3:} We have $\mathcal{Q} \subset \mbox{dh} \left( \cone (\mathbf{1},\balpha,-\balpha)\right)$, and so $\mathcal{Q} \subset \trop \left( \mathcal{N}_d \right)$. \\
Let $\mathcal{S}:=\cone(\mathbf{1},\balpha,-\balpha)$ and let $\mathcal{D}:= \mbox{dh} (\mathcal{S})$ denote its double hull. From Corollary \ref{cor:dual of max-closure} we deduce $\mathcal{D}^\ast =\bigoplus_{i=1}^d \mathcal{S}^\ast \cap \Q_i$, where $\Q_i$ is the orthant of $\R^d$ in which the $i$-th coordinate is non-positive and any other coordinate is nonnegative. 
Therefore, the set of extreme rays of $\mathcal{D}^\ast$ is contained in the union of extreme rays of $\mathcal{S}^\ast \cap \Q_i$ for $i=1,\ldots,d$. 
For any $i$ the polyhedral cone $\mathcal{S}^\ast \cap \Q_i$ is defined by $(d+3)$-many inequalities, since $$\mathcal{S}^\ast = \{ (a_1,\ldots,a_d) \in \R^d : a_1+\ldots+a_d \geq 0, a_1+2a_2+\ldots+da_d = 0\}.$$ Then $\mathcal{S}^\ast$ is not a full-dimensional polyhedral cone. At least $d-1$ many linear independent inequalities of $\mathcal{S}^\ast \cap \Q_i$ must be tight to form an extreme ray. One of them must be $a_1+2a_2+\ldots+da_d \geq 0$. In particular, $d-2$ of the inequalities in $\mathcal{A}_i := \{a_1+\ldots+a_d \geq0, a_i \leq 0, a_j \geq 0 \, \, \mid \, \, j \in [d] \setminus \{i\}\}$ are tight. In the following, we examine the various combinations for which $d-2$ of these inequalities are tight.  \\
We consider a ray $\mathbf{r} := (r_1,\ldots,r_d)$ of $\mathcal{S}^\ast \cap \Q_i$ where $d-2$ many inequalities from $\mathcal{A}_i$ are tight. 
\begin{itemize}
    \item[(1)] Let $k, l$ be distinct integers in $[d]\setminus\{i\}$. Let $r_\rho = 0$ be the tight inequalities for $\rho \in [d] \setminus \{k,l\}$ together with $kr_k+lr_l = 0$. Then we have $r_k + r_l>0$, $r_k> 0$ and $r_l > 0$. This gives a contradiction.
    \item[(2)] Let $k \in [d] \setminus \{i\}$ and let $r_\rho =0$ for $\rho \in [d]\setminus \{k,i\}$ and $kr_k+ir_i = 0$ be the tight inequalities. We have $r_k + r_i > 0, r_k > 0, r_i < 0$. We can assume that $r_i = -k$ which implies $r_k = i$. Then $0 < r_k+r_i = i-k$ implies $i > k$. Thus $\mathbf{r}=(0,\ldots,0,i,0,\ldots,0,-k,0,\ldots,0)$ where $i$ is the $k$-th and $-k$ the $i$-th coordinate.
    \item[(3)] Let $k, l,m$ be distinct integers in $[d]\setminus\{i\}$. Let $r_{1}+\ldots+r_d = 0, r_{1}+2r_2+\ldots+d \cdot r_d = 0, r_\rho = 0$ for $ \rho \in [d] \setminus \{k,l,m\}$ be the tight inequalities. Then $r_k> 0, r_l> 0$ and $r_m > 0$ which gives a contradiction.
    \item[(4)] Let $k, l$ be distinct integers in $[d]\setminus\{i\}$. Let $r_1 + \ldots + r_d = 0$, $r_1 + 2r_2+\ldots + d \cdot r_d = 0, r_\rho = 0$, for $\rho \in [d] \setminus \{k,l,i\}$ be the tight inequalities. We have $r_k>0, r_l >0$ and $r_i < 0$. Without loss of generality be $l > k$ and $r_i = k-l <0$. Then $r_k = l-i$ and $r_l = i-k$ and since $r_k$ and $r_l$ are positive we have $k < i < l$. Thus, $\mathbf{r}=(0,\ldots,l-i,0,\ldots,k-l,\ldots,i-k,0,\ldots,0)$, where $l-i$ is the $k$-th, $k-l$ the $i$-th and $i-k$ the $l$-th coordinate. 
\end{itemize}
Therefore, \[\mathcal{D} = \{y\in \R^d \, \, \mid \, \, iy_k-ky_i \geq 0, \mbox{ for } i > k, \mbox{ and } (l-i)y_k-(l-k)y_i+(i-k)y_k \geq 0, \mbox{ for } k < i < l\}.\]
Observe that the inequalities $iy_k-ky_i \geq 0$ for $i > k$ are conical combinations of the inequalities $(k+1)y_k - k y_{k+1} \geq 0$. This can be seen by induction, since \[iy_k-ky_i = \frac{i}{i-1}((i-1)y_k-ky_{i-1})+\frac{k}{i-1}(iy_{i-1}-(i-1)y_i)\] for all $1 \leq k < i \leq d$. Moreover, for $1 \leq k < i < l \leq d$ we have 
\[ (l-i)y_k-(l-k)y_i+(i-k)y_l = \sum_{u=k}^{i-1}(l-i)(u-k+1)(y_u-2y_{u+1}+y_{u+2})+\sum_{u=i}^{l-2}(i-k)(l-1-u)(y_u-2y_{u+1}+y_{u+2}).\]
This identity can be found in the proof of Claim 3 in (\cite[Theorem 2.15]{blekherman2022path}). In particular, many of the defining inequalities of $\mathcal{D}$ are redundant as conical combinations of the inequalities defining $\mathcal{Q}$. \\
Thus $\mathcal{Q} =  \mathcal{D}$ and we have $\mathcal{D}  \subset \trop \left( \mathcal{N}_d \right)$ because $\trop(\mathcal{N}_d)$ is a convex cone (Lemma \ref{le:trop(Nd) is max-closed}) containing the line $\R \cdot \balpha$ and the ray $\R_{\geq0} \cdot \mathbf{1}$ by Claim 1, and since $\trop(\mathcal{N}_d)$ is a max-closed set. Thus, we obtain $\mathcal{Q} \subset \trop (\mathcal{N}_d)$. \smallskip

Note that the $d-1$ vectors $(1,-2,1,0,\ldots,0),\ldots,(0,\ldots,0,1,-2,1),(0,\ldots,0,d,-(d-1))$ are linearly independent. Therefore, each linear inequality is indeed facet-defining. \smallskip

The vector $(1,2,\ldots,d)$ is contained in the lineality space of $\trop (\mathcal{N}_d)$. We examine the extremal rays described by $y_1 = c$, for some $c \in \R$, and $d-2$ of the facet defining inequalities of $\trop (\mathcal{N}_d)$. For any $1 \leq k \leq d-2$, the vector $(k,k-1,\ldots,1,0,\ldots,0)$ evaluates to $1$ if evaluated on $y_k-2y_{k+1}+y_{k+2}$, and to $0$ when evaluated on $y_i-2y_{i+1}+y_{i+2}$ for $i \neq k$, or $dy_{d-1}-(d-1)y_d$. On the other hand, we have $y_i-2y_{i+2}+y_{i+2} $ evaluates to $0$ at $(1,\ldots,1)$, but $dy_{d-1}-(d-1)y_d$ evaluates to $1$ at the same point. The $d-1$-vectors are linearly independent and therefore generate the extremal rays of $\trop (\mathcal{N}_d)$.
\end{proof}

\begin{example} \label{ex:tropfor345}
\begin{align*}
    \trop(\mathcal{N}_3) &= \{ y\in\R^3 :\,\, y_1 + y_3 \geq 2y_2,\,\,\, 3y_2 \geq 2y_3\},\\
     \trop(\mathcal{N}_4) &= \{ y\in\R^4 :\,\, y_1 + y_3 \geq 2y_2,\,\,\, y_2+y_4 \geq 2 y_3,\,\,\, 4y_3 \geq 3y_4\},\\
     \trop(\mathcal{N}_5) &= \{ y\in\R^5 :\,\, y_1 + y_3 \geq 2y_2,\,\,\, y_2+y_4 \geq 2 y_3,\,\,\, y_3+y_5 \geq 2 y_4,\,\,\, 5y_4 \geq 4y_5\}.
\end{align*}
\end{example}

For a partition $\la$ let $\la(k)$ denote the number of times $k$ appears as a part of $\la$. If $\la\vdash d$ we alternatively write $$\la=(1^{\la(1)},\dots,d^{\la(d)})$$ usually omitting writing $j^{\la(j)}$ whenever $\la(j)=0$.

\begin{proof}[Proof of Theorem \ref{thm: power sums order}]

That $\lambda \succeq \mu$ implies $p_\la \geq p_\mu$ follows by Corollary \ref{cor:sd->sos}. We will now prove the other implication. 

Let $\Delta_i:=\la(i)-\mu(i)$. The inequality $p_\la\ge p_\mu$ can be rewritten as $\prod_i p_i^{\la(i)}\ge\prod_i p_i^{\mu(i)}$ or $\prod_i p_i^{\Delta_i}\ge1$, therefore $\sum_i\Delta_i\log(p_i)\ge0$, and so $\sum_i\Delta_iy_i\ge0$ is valid on $\trop(\mathcal{N}_d)$.
\smallskip

By Farkas' lemma, the valid linear inequalities on a polyhedral cone are precisely the conical combinations of its facets.

The case $d=4$ is illustrative. Take a conical combination of the facets of $\trop(\mathcal{N}_4)$ (see Theorem \ref{thm:trop(Nd)}) that gives an inequality $\sum_{i=1}^4c_iy_i\ge0$ with $c_i\in\Z$,
\begin{align*}
a_1(y_1-2y_2+y_3)+a_2(y_2-2y_3+y_4)+b(4y_3-3y_4)&\ge0\\
\Rightarrow a_1y_1+(a_2-2a_1)y_2+(a_1-2a_2+4b)y_3+(a_2-3b)y_4&\ge0.
\end{align*}
Since the obtained coefficients are integers, then one can show $a_1,a_2,b$ are natural numbers. Taking the negatives to right-hand side,
\begin{align*}
a_1y_1+a_2y_2+(a_1+4b)y_3+a_2y_4\ge2a_1y_2+2a_2y_3+3by_4
\end{align*}
and we claim $(1^{a_1},2^{a_2},3^{a_1+4b},4^{a_2})\succeq(2^{2a_1},3^{2a_2},4^{3b})$. Then $\lambda$ and $\mu$ are obtained from these partitions by fusing with the same partition $\omega$ on both sides and we have $\la \succeq \mu$ by Lemma \ref{lem:fusion}.

In the general case, one can prove inductively that if a conical combination ends up with integer coefficients then it came from nonnegative integer $a_i$'s and $b$. Then take the negative part to the right-hand side and prove that
\begin{align*}\label{ieq}
    L\succeq M
\end{align*}
where
\begin{align*}
L&=(1^{a_1},2^{a_2},3^{a_1+a_3},4^{a_2+a_4},\dots,(d-2)^{a_{d-4}+a_{d-2}},(d-1)^{a_{d-3}+db},d^{a_{d-2}}),\\
M&=(2^{2a_1},3^{2a_2},4^{2a_3},\dots,(d-1)^{2a_{d-2}},d^{(d-1)b})
\end{align*}
This follows from fusing
\begin{align*}
    (1^{a_1},3^{a_1})&\succeq(2^{2a_1})\\
    (2^{a_2},4^{a_2})&\succeq(3^{2a_2})\\
    &\vdots\\
    ((d-3)^{a_{d-3}},(d-1)^{a_{d-3}})&\succeq((d-2)^{2a_{d-3}})\\
    ((d-2)^{a_{d-2}},d^{a_{d-2}})&\succeq((d-1)^{2a_{d-2}})\\
    ((d-1)^{db})&\succeq(d^{(d-1)b})
\end{align*}
by using Lemma \ref{lem:fusion}. Now observe,
\begin{equation*}
    \begin{aligned}[c]
    a_1&=\Delta_1\\
    -2a_1+a_2&=\Delta_2\\
    a_1-2a_2+a_3&=\Delta_3\\
    &\vdots\\
    a_{d-4}-2a_{d-3}+a_{d-2}&=\Delta_{d-2}\\    a_{d-3}-2a_{d-2}+db&=\Delta_{d-1}\\
    a_{d-2}-(d-1)b&=\Delta_d
    \end{aligned}
    \qquad\Rightarrow\qquad
    \begin{aligned}[c]
a_1&=\Delta_1\\
a_2&=\Delta_2+2a_1\\
a_1+a_3&=\Delta_3+2a_2\\
&\vdots\\
a_{d-4}+a_{d-2}&=\Delta_{d-2}+2a_{d-3}\\
a_{d-3}+db&=\Delta_{d-1}+2a_{d-2}\\
a_{d-2}&=\Delta_d+(d-1)b
    \end{aligned}
\end{equation*}

\begin{align*}
   \Rightarrow L=(1^{\Delta_1},2^{\Delta_2+2a_1},\dots,d^{\Delta_d+(d-1)b})
\end{align*}
So we can fuse, by Lemma \ref{lem:fusion}, both sides of $L\succeq M$ with $\mu$ and obtain $L\mu\succeq M\mu$. Now observe that $L\mu=\la M$, so $\la M\succeq M\mu$, therefore, by Lemma \ref{lem:arquimedean}, $\la\succeq\mu$.
\end{proof}

\section{The superdominance order, power sums and tropicalization of the limit cones}\label{sec:trop/superdominance/power sums/ limit cones}
In this section we tropicalize the cones $\B\P_{2d}^*$ and $\B\Sigma_{2d}^*$ which turn out to be rational polyhedra. 
Additionally, we characterize $\trop(\B\P_{2d}^*)$ in terms of the superdominance order and conjecture that $\trop(\B\Sigma_{2d}^*)$ can also be described by the superdominance order. We provide supporting evidence for this conjecture for small degrees.\smallskip

In Subsection \ref{sec:trop BP_2d*} we characterize $\trop(\B\P_{2d}^*)$ using tropical convexity and our understanding of $\trop (\mathcal{N}_d)$. Then, in Subsection \ref{sec:trop psd superdominance} we define a decreasing chain of polyhedra based on the superdominance order and show that the sequence eventually stabilizes and equals $\trop(\B\P_{2d}^*)$. 
Subsection \ref{sec:trop BSigma_2d} addresses the tropicalization of $\B\Sigma_{2d}^*$, where we compute $\trop(\B\Sigma_{2d}^*)$ using a technical lemma on the tropicalization of spectrahedra, partial symmetry reduction and insights from the superdominance order.  We observe that the tropicalizations of the cones differ for any even degree of at least $10$. Finally, in Subsection \ref{sec:examples} we provide examples of the tropicalizations for degrees $4,6,8,$ and $10$, and construct a nonnegative limit form of degree $10$ that is not a limit sum of squares.

\subsection{The tropicalization of $\mathcal{B}\mathcal{P}_{2d}^*$} \label{sec:trop BP_2d*}

We determine $\trop (\B\P_{2d}^*)$ (Proposition \ref{prop:trop dual to psd} and Corollary \ref{cor:linspace tropmom}) based on our understanding of the rational polyhedron $\trop(\mathcal{N}_d)$ in Theorem \ref{thm:trop(Nd)}. We follow the strategy applied in \cite[Proposition~5.4]{acevedo2024power} where a \emph{commutativity property} of tropicalizations and conical hulls of semialgebraic sets (see \cite[Proposition 2.3]{blekherman2022moments} and \cite[Lemma 8]{allamigeon2019tropical}) was used to compute the tropicalization of the dual cone to nonnegative even symmetric means via tropical convexity. However, from Proposition \ref{prop:even symmetric non-semialgebraic} it follows that $\nu(\mathcal{N}_d)$ is not semialgebraic for any $d\ge3$, but instead has the Hadamard property (Lemma \ref{lem:newton}). Fortunately, the commutativity property also applies to sets possessing the Hadamard property. (Proposition \ref{prop:tropcommutes}). \medskip

The map \[\tilde{\nu}_d : \R^d \to \R^{\pi (d)}, (x_1,\ldots,x_d) \mapsto (dx_1,(d-2)x_1+x_2,\ldots,x_d)\] denotes the tropicalization of the monomial map $\nu_d$. The coordinate-functions of $\tilde{\nu}_d$ have the form $\sum_{i=1}^d \alpha_i x_i$ where $\alpha_i \in \N_0$ and $\sum_{i=1}^d i \alpha_i = d.$ Our first goal is to show that the diagram (\ref{diag:commutative diagram}) commutes for sets with Hadamard property.
\begin{equation}
  \label{diag:commutative diagram}
     \begin{tikzcd}
    \R^d_{\geq 0} \arrow{r}{\nu_d} \arrow{d}{\trop}  & \R_{\geq 0}^{\pi (d)}  \arrow{r}{\conv} & \R_{\geq 0}^{\pi (d)} \arrow{d}{\trop} \\
    \R^{d} \arrow{r} {\tilde{\nu}_d}   & \R^{\pi (d)} \arrow{r} {\tcone} & \R^{\pi (d)}
  \end{tikzcd} 
\end{equation}

\begin{proposition} \label{prop:tropcommutes}
Suppose $S \subset \R_{\geq 0}^n$ has the Hadamard property. Then, $$\trop(\cone( S))= \tcone (\trop (S)).$$
\end{proposition}

Before presenting the proof we show a preliminary lemma.

\begin{lemma}\label{le:inequality defines convex set}
Let $I \subsetneq [n]$ and $m \in [n] \setminus I$. Then, the set of all points $x \in \R_{\geq 0}^n$ satisfying the binomial inequality $$\prod_{i \in I} x_i^{c_i} \geq x_m^d,$$ where $c_i\in \R_{>0}$ and $d =\sum_{i \in I} c_i$, is a convex cone.
\end{lemma}
\begin{proof}
Let $C:=\{x\in\R_{\ge0}^n:\prod_{i \in I} x_i^{c_i}\ge x_m^d\}$. By homogeneity it suffices to show that $C$ is closed under addition, so let $y,z\in C\cap\R_{>0}^n$.
\medskip

By the weighted AM-GM inequality, taking sums and products over $i\in I$,
\begin{align*}
1=\frac{\sum c_i\frac{y_i}{y_i+z_i}}{\sum c_i}+\frac{\sum c_i\frac{z_i}{y_i+z_i}}{\sum c_i}&\ge\left(\prod\left(\frac{y_i}{y_i+z_i}\right)^{c_i}\right)^{1/d}+\left(\prod\left(\frac{z_i}{y_i+z_i}\right)^{c_i}\right)^{1/d}\\
\therefore\qquad \left(\prod (y_i+z_i)^{c_i}\right)^{1/d}&\ge \left(\prod y_i^{c_i}\right)^{1/d}+\left(\prod z_i^{c_i}\right)^{1/d}\ge y_m+z_m
\end{align*}
and therefore $y+z\in C$.
\end{proof}

\begin{proof}[Proof of Proposition \ref{prop:tropcommutes}]
We suppose $z \in \tcone (\trop (S))$, i.e., we have $z = a_1 \odot v_1 \oplus \ldots \oplus a_k \odot v_k$ for some $a_i \in \R$ and $v_i \in \trop (S)$. 

First, we show that $a_i \odot v_i \in \trop (\cone (S))$ for all $1 \leq i \leq k$. Since $v_i \in \trop(S)$ we have $$v_i = \lim_{m \rightarrow \infty} \log_{\frac{1}{\tau_m}}(w^{(m)}_i)$$ for a sequence $(w^{(m)}_i)_m 
\subset S \cap \R_{>0}^s$ and a sequence $(\tau_m)_m \subset (0,\epsilon)$ converging to $0$ (\cite[Proposition 2.1]{alessandrini2013logarithmic}). For all $m$ we have $\frac{1}{\tau_m^{a_i}}w_i^{(m)}\in \cone (S)$ and 
 \begin{align*}
     a_i \odot v_i & = a_i \odot \lim_{m \rightarrow \infty} \log_{\frac{1}{\tau_m}}(w_i^{(m)}) \\ &= \lim_{m \rightarrow \infty} \log_{\frac{1}{\tau_m}}\left( \frac{1}{\tau_m^{a_i}} \right) \odot \log_{\frac{1}{\tau_m}} (w_i^{(m)}) \\ &= \lim_{m \rightarrow \infty} \log_{\frac{1}{\tau_m}}\left( \frac{1}{\tau_m^{a_i}}w_i^{(m)}\right) 
\end{align*} 
which shows $a_i \odot v_i \in \trop(\cone(S))$.
 Since $S$ has Hadamard property the set $\cone(S)$ has also the Hadamard property. We have that the set $\trop( \cone(S))$ is max-closed by \cite[Lemma~2.2.]{blekherman2022path}. Thus, $z \in \trop ( \cone (S)))$ because $z$ is a tropical sum of the $a_i \odot v_i$'s. \\
To prove the remaining inclusion we first note that $\trop (\cone (S))$ and $\tcone (\trop (S))$ are closed convex cones in $\R^n$ by (\cite[Lemma 2.2 (2)]{blekherman2022tropicalization}). More precisely, this follows since $\cone (S)$ has the Hadamard property and because $\trop (S)$ is a closed convex cone the set $\tcone (\trop (S))$ is the intersection of closed convex cones by (\ref{eq:k}). We show instead the equivalent formulation 
$$ \tcone (\trop (S)) ^\ast \subseteq  \trop (\cone (S)) ^\ast.$$ Corollary \ref{cor:dual of max-closure} implies \begin{align*}
    \tcone (\trop (S))^\ast & = \left( \bigcap_{k=1}^n (\trop (S) + \R \cdot \mathbf{1} + \Q_k) \right)^\ast \\ 
    & = \bigoplus_{k=1}^n( \trop (S)^\ast \cap \R \cdot \mathbf{1}^\ast \cap \Q_k^\ast) \\ 
    & =  \bigoplus_{k=1}^n (\trop (S)^\ast \cap \{ z \in \R^n : z_i \geq 0, i \neq k, z_k \leq 0, \sum_{i \neq k} z_i = -z_k \}).
\end{align*}
Therefore, any extremal ray in $\tcone (\trop (S))^\ast $ has precisely one negative coefficient and the sum over all positive coefficients equals the absolute value of the negative coefficient. By (\cite[Proposition 2.4]{blekherman2022path}) any $\mathbf{\alpha} = \mathbf{\alpha}_+-\mathbf{\alpha}_-  \in \trop (S)^*$ with $\mathbf{\alpha}_+,\mathbf{\alpha}_- \in \R_{\geq 0}^n$ and $\mathbf{\alpha}_- \neq 0$ transfers to a valid binomial inequality $x^{\mathbf{\alpha}_+} \geq x^{\mathbf{\alpha}_-}$ on $S$.
Thus, any extremal ray in $\tcone (\trop(S))^\ast$ gives rise to a valid homogeneous binomial inequality $x^\alpha \geq x^\beta $, and $\beta$ has precisely one non-zero entry and $\sum_{i=1}^n \alpha_i = \sum_{i=1}^n \beta_i$. We saw in Lemma \ref{le:inequality defines convex set} that the set of all solutions in $\R_{\geq0}^s$ of the real homogeneous binomial inequality forms a convex cone containing $S$. Thus, it is a valid binomial inequality on the set $\cone (S)$. We deduce that any extremal ray in $\tcone( \trop (S))^*$ is contained in $\trop ( \cone (S))^*$.
\end{proof}

\begin{lemma}
\label{lem:nuN Hadamard}
The set $\nu_d (\mathcal{N}_d)$ is a cone and has Hadamard property. Moreover, the set $\B\P_{2d}^*$ is a convex cone which has Hadamard property.
\end{lemma}
\begin{proof}
Since any $z \in \mathcal{N}_d$ is the limit of a sequence in $\bigcup_{n=1}^\infty \nu_{d,n}^e (\R^n)$ we can use that the map $\nu_d \circ (p_2,\ldots,p_{2d})$ is homogeneous and obtain that $\nu_d (\mathcal{N}_d)$ is a cone. For $x=\nu_d (a),y=\nu_d (b) \in \nu_d (\mathcal{N}_d)$ we have $$(x_1y_1,\ldots,x_sy_s) = (a_1^db_1^d,a_1^{d-2}a_2b_1^{d_2}b_2,\ldots,a_db_d) \in \nu_d(\mathcal{N}_d)$$ since $\mathcal{N}_d$ has Hadamard property by Lemma \ref{lem:newton}. \\
The set $\B\P_{2d}^* = \cone (\nu_d (\mathcal{N}_d))$ has Hadamard property because Hadamard multiplication of convex combinations of elements in $\nu_d( \mathcal{N}_d)$ gives again a convex combination of elements in $\nu_d( \mathcal{N}_d)$.
\end{proof}

\begin{proposition} \label{prop:trop dual to psd}
$\trop (\B\P_{2d}^*) = \tcone \left( \tilde{\nu}_d \left( \trop \left( \mathcal{N}_d \right) \right) \right)$.
\end{proposition}

Claim 1 in the proof below was already observed in \cite[Remark~5.8]{acevedo2024power}.

\begin{proof}
We recall $\B\P_{2d}^* =  \cone (\nu_d \left( \mathcal{N}_d) \right)$ and divide the proof into two parts.\smallskip

\textbf{Claim 1:} $\tilde{\nu}_d \left( \trop ( \mathcal{N}_d) \right) = \trop \left( \nu_d (\mathcal{N}_d ) \right)$. \\
The sets $\mathcal{N}_d$ and $\nu_d(\mathcal{N}_d)$ contain the all ones vector $\mathbf{1}$ and have the Hadamard property by Lemma \ref{lem:newton} and Lemma \ref{lem:nuN Hadamard}. Thus, by (\cite[Lemma 2.2 (2)]{blekherman2022tropicalization}) Claim 1 is equivalent to \begin{align*}
    \tilde{\nu}_d \left( \clo (\conv(\log(\mathcal{N}_d)))\right) = \clo \left( \conv( \log (\nu_d (\mathcal{N}_d)))\right).
\end{align*}
However, \begin{align*}
   \clo(\conv (\log( \nu_d (\mathcal{N}_d)))) = \clo (\conv (\tilde{\nu}_d(\log (\mathcal{N}_d)))) = \clo (\tilde{\nu}_d (\conv(\log(\mathcal{N}_d)))),
\end{align*}
where the first equality follows from the definition of $\log$ and the second equality follows because taking convex hull and applying a linear map commute.
Moreover, since $\tilde{\nu}_d$ is injective and linear we have $\tilde{\nu}_d (\clo (A)) = \clo (\tilde{\nu}_d(A))$ for all sets $A \subset \R^{\pi (d)}$. Therefore, Claim 1 follows. \smallskip

\textbf{Claim 2:} $\trop(\cone (S)) = \tcone (\trop (S))$ for $S= \nu_d (\mathcal{N}_d).$ \\
This follows from Proposition \ref{prop:tropcommutes} since the set $S$ has Hadamard property by Lemma \ref{lem:nuN Hadamard}. \smallskip

Finally, by combining the assertions from Claim 1 and Claim 2 we have \[\trop(\B\P_{2d}^*) = \trop (\cone (\nu_d (\mathcal{N}_d))) = \tcone (\trop(\nu_d(\mathcal{N}_d))) = \tcone (\tilde{\nu}_d (\trop (\mathcal{N}_d))).\] 
\end{proof}

\begin{corollary}\label{cor:linspace tropmom}
The set $\trop(\B\P_{2d}^*)$ is a rational polyhedral cone and tropically convex with lineality space spanned by $\mathbf{1}\in\R^{\pi(d)}$. Moreover, any extreme ray of $\trop(\B\P_{2d}^*)^*$ has exactly one negative coordinate.
\end{corollary}
\begin{proof}
    It follows from Theorem \ref{thm:trop(Nd)} and Proposition \ref{prop:trop dual to psd} that the cone $\trop(\B\P_{2d}^*)$ is rational polyhedral and tropically convex. 
   
   Also observe that the $d-1$ facet-defining inequalities of $\trop(\mathcal{N}_d)\subset\R^d$, given by Theorem \ref{thm:trop(Nd)}, are linearly independent, so the lineality space of $\trop(\mathcal{N}_d)$ is one-dimensional, and one can see it is the line $\spn(1,2,\dots,d)$ (because there all inequalities are tight). This implies that the lineality space of $\tilde{\nu}(\trop(\mathcal{N}_d))$ is the line $\spn(1,\dots,1)=\tilde{\nu}(\spn(1,2,\dots,d))$ in $\R^{\pi(d)}$, because $\tilde{\nu}$ is injective. Now, $\trop(\B\P_{2d}^*)$ is the tropical conical hull of $\tilde{\nu}(\trop(\mathcal{N}_d))$ (Proposition \ref{prop:trop dual to psd}), which preserves the lineality space.
   
   Since the lineality space of $\trop(\B\P_{2d}^*)$ is the line spanned by $\mathbf{1}$ then $\trop(\B\P_{2d}^*)^*$ is contained in the orthogonal complement of $\spn\mathbf{1}$, which is the subspace where the coordinates add up to zero. But, by Lemma \ref{le:max-closure}, the extreme rays of $\trop(\B\P_{2d}^*)^*$ have at most one negative coordinate, and then it must be exactly one negative coordinate.
\end{proof}

\subsection{Characterization of $\operatorname{trop} (\mathcal{B}\mathcal{P}_{2d}^*)$ by the superdominance order}\label{sec:trop psd superdominance}
Using the superdominance order, we define a decreasing chain of polyhedra $T_{2d}^{(k)}$ in $\R^{\pi(d)}$. We prove that this sequence stabilizes and that it does so at the tropicalization of $\B\P_{2d}^*$ (see Theorem \ref{thm:tropmom}). 

\begin{definition} \label{def:sequence T2dk}
    For a positive integer $k$ let $T_{2d}^{(k)}\subset\R^{\pi(d)}$ denote the cone given by the intersection of the half-spaces $$y_{\la^1}+\dots+y_{\la^k}\ge ky_\mu$$ where $\la^1,\dots,\la^k,\mu\in\Lambda_{2d}^\varepsilon$ and $\la^1\cdots\la^k\succeq\mu^{\circ k}$. 
\end{definition}

\begin{example}
    We have $T_6^{(1)}=\{(y_{(2^3)},y_{(4,2)},y_{(6)})\in\R^3\,\,\mid\,\,y_{(2^3)}\ge y_{(4,2)}\ge y_{(6)}\}$ since $(2^3)\succ(4,2)\succ(6)$. Furthermore, we find that $T_6^{(2)}=\left\{  (y_{(2^3)},y_{(4,2)},y_{(6)}) \in \R^3 \,\,\mid\,\,  y_{(2^3)}+y_{(6)} \geq 2y_{(4,2)},\, y_{(4,2)} \geq y_{(6)} \right\}$ is equal to $\trop(\B\Sigma_6^*)$ (see Example \ref{ex:trop(S6)}).
\end{example}

By the fusion lemma (Lemma \ref{lem:arquimedean}), the sequence $(T_{2d}^{(k)})_{k\ge1}$ is nested non-increasing, i.e., $T_{2d}^{(k)}\supseteq T_{2d}^{(k+1)}$ for each positive integer $k$. Now we show that this sequence stabilizes, and the stabilization is $\trop(\B\P_{2d}^*)$. 

\begin{lemma}\label{lemma:momentsbringpower}
    Let $\la^1,\dots,\la^k,\mu^1,\dots,\mu^k\in\Lambda_{2d}^\varepsilon$. If the inequality $y_{\la^1}+\dots+y_{\la^k}\ge y_{\mu^1}+\dots+y_{\mu^k}$ is valid on $\trop(\B\P_{2d}^*)$, then $\la^1\cdots\la^k\succeq\mu^1\cdots\mu^k$.
\end{lemma}
\begin{proof}
Observe that $\B\P_{2d}^*\supset\nu(\mathcal{N}_d)$, so $\trop(\B\P_{2d}^*)\supset\trop(\nu(\mathcal{N}_d))$. Hence \[(*)\,\, y_{\la^1}+\dots+y_{\la^k}\ge y_{\mu^1}+\dots+y_{\mu^k}\] valid on $\trop(\B\P_{2d}^*)$ implies that ($*$) valid on $\trop(\nu(\mathcal{N}_d))$. Now $\nu(\mathcal{N}_d)\subset\R_{\ge0}^{\pi(d)}$ has the Hadamard property, so Lemma 2.2 of \cite{blekherman2022tropicalization} implies $\trop(\nu(\mathcal{N}_d))=\clo(\cone(\log(\nu(\mathcal{N}_d))))$. Therefore ($*$) is valid on $\log(\nu(\mathcal{N}_d))$. This means that $\log(p_{\la^1})+\dots+\log(p_{\la^k})\ge \log(p_{\mu^1})+\dots+\log(p_{\mu^k})$, hence $p_{\la^1}\cdots p_{\la^k}\ge p_{\mu^1}\cdots p_{\mu^k}$, i.e. $p_{\la^1\cdots\la^k}\ge p_{\mu^1\cdots\mu^k}$ by multiplicativeness of power sums, and so by Theorem \ref{thm: power sums order} we obtain $\la^1\cdots\la^k\succeq\mu^1\cdots\mu^k$.
\end{proof}

\begin{theorem}\label{thm:tropmom}
    For each positive integer $d\ge3$ there exists a positive integer $\tau_d\ge2$ such that $T_{2d}^{(k)}=T_{2d}^{(\tau_d)}$ for all $k\ge\tau_d$, $T_{2d}^{(\tau_d-1)}\supsetneq T_{2d}^{(\tau_d)}$, and $\trop(\B\P_{2d}^*)=T_{2d}^{(\tau_d)}$.
\end{theorem}
\begin{proof}
We divide the proof into three steps. 

In (i) we show that every inequality of the form $y_{\la^1}+\dots+y_{\la^k}\ge ky_\mu$, such that $\la^1\cdots\la^k\succeq\mu^{\circ k}$, and $k\ge1$, is valid on $\trop(\B\P_{2d}^*)$.
\smallskip

In (ii) we show that each facet of $\trop(\B\P_{2d}^*)$ is of the form $y_{\la^1}+\dots+y_{\la^k}\ge ky_\mu$ with $\la^1\cdots\la^k\succeq\mu^{\circ k}$ for some $k\ge1$.
\smallskip

Finally in (iii) we show that $\tau_d\ge2$.
\smallskip

This suffices because we have $T_{2d}^{(k)} \subset \trop(\B\P_{2d}^*)$ for all $k$, by (i) and since the rational polyhedron $\trop(\B\P_{2d}^*)$ has a finite number of facets, there exists an upper bound on the integer $\tau_d$ by (ii) depending on $d$. Combining these observations with (iii) we can deduce the claimed statement. 
\medskip

    (i) Let $\la^1,\dots,\la^k,\mu\in\Lambda_{2d}^\varepsilon$ such that $\la^1\cdots\la^k\succeq\mu^{\circ k}$. By Theorem \ref{thm: power sums order} then $p_{\la^1\cdots\la^k}\ge p_{\mu^{\circ k}}$, or equivalently, by the multiplicativeness of power sums, $p_{\la^1}\cdots p_{\la^k}\ge p_{\mu}^k$. This is a binomial inequality which is valid on $\nu(\mathcal{N}_d)$, and, by Lemma \ref{le:inequality defines convex set}, it is precisely the kind of binomial inequality that remains valid on $\cone\nu(\mathcal{N}_d)=\B\P_{2d}^*$, therefore $y_{\la^1}+\dots+y_{\la^k}\ge ky_\mu$
    is valid on $\trop(\B\P_{2d}^*)$.
    \smallskip

    (ii) By Corollary \ref{cor:linspace tropmom} the extreme rays of $\trop(\B\P_{2d}^*)^*$ have exactly one negative coordinate and the coordinates add up to zero. Further, one can take an integer point on each extreme ray since $\trop(\B\P_{2d}^*)$ is a rational polyhedral cone, therefore the facets of $\trop(\B\P_{2d}^*)$ are of the form $y_{\la^1}+\dots+y_{\la^k}\ge ky_\mu$, where $\la^1,\dots,\la^k,\mu\in\Lambda_{2d}^\varepsilon$ (observe that the $\la^i$ are not necessarily distinct). 
    By Lemma \ref{lemma:momentsbringpower} if $y_{\la^1}+\dots+y_{\la^k}\ge ky_\mu$ is valid on $\trop(\B\P_{2d}^*)$ then $\la^1\cdots\la^k\succeq\mu^{\circ k}$.
    \smallskip    
    
    (iii) For $d\ge3$ it holds that $(2^d)\succ(4,2^{d-2})\succ(6,2^{d-3})$ and $(2^d)(6,2^{d-3})\succ(4,2^{d-2})(4,2^{d-2})$, so $y_{(2^d)}+y_{(6,2^{d-3})}\ge2y_{(4,2^{d-2})}$ is not implied by the inequalities in $T_{2d}^{(1)}$ since these are just $y_\la\ge y_\mu$ if and only if $\la\succeq\mu$. Therefore $T_{2d}^{(1)}\supsetneq T_{2d}^{(2)}$, and so $\tau_d\ge2$.
\end{proof}

\begin{corollary}\label{cor:Tgood}
    Let $\la^1,\dots,\la^k,\mu^1,\dots,\mu^k\in\Lambda_{2d}^\varepsilon$. If $\la^1\cdots\la^k\not\succeq\mu^1\cdots\mu^k$ then the inequality \linebreak $y_{\la^1}+\dots+y_{\la^k}\ge y_{\mu^1}+\dots+y_{\mu^k}$ is not valid on $T_{2d}^{(r)}$ for any $r\ge1$. 
\end{corollary}
\begin{proof}
    By Theorem \ref{thm:tropmom} $T_{2d}^{(r)}\supseteq\trop(\B\P_{2d}^*)$, and so the statement follows from the contrapositive of Lemma \ref{lemma:momentsbringpower}.
\end{proof}

The following proposition is not only essential for proving Theorem \ref{thm:tropmom}, but also for showing the inequality between the polyhedra $\trop(\B\P_{2d}^*)$ and $\trop(\B\Sigma_{2d}^*)$ for all $d \geq 5$ in Theorem \ref{thm:tropsdiffer}.
 \medskip

 Recall that we choose the even partitions of $2d$, in the reverse lexicographic order, to label the coordinates of $\R^{\pi(d)}$.

\begin{proposition}\label{prop:T3subsetT2}
    If $d\ge5$ then $T_{2d}^{(2)}\supsetneq T_{2d}^{(3)}$, i.e. $\tau_d\ge3$.
\end{proposition}
\begin{proof}
Denote by $y_i:=y_{\mu^i}$ where $\mu^i$ is the $i$-th largest even partition of $2d$ in the reverse lexicographic order. For instance 
\begin{align*}
y_1&:=y_{(2^d)},\quad\quad\,\,\,\,\, y_2:=y_{(4,2^{d-2})},\quad y_3:=y_{(6,2^{d-3})},\\
\quad y_4&:=y_{(4^2,2^{d-4})},\quad y_5:=y_{(8,2^{d-4})},\quad y_6:=y_{(6,4,2^{d-5})}.
\end{align*}
In the superdominance order, the corresponding partitions are ordered as follows
\begin{align*}
(2^{d})\succ(4,2^{d-2})\succ(6,2^{d-3})\succ(4^2,2^{d-4})\succ(8,2^{d-4})\succ(6,4,2^{d-5})
\end{align*}
and $(6,4,2^{d-5})\succ\omega$ for any even partition $\omega$ of $2d$ not in the above list. 

Observe that the inequality $y_1-3y_3+y_5+y_6\ge0$ is valid on $T_{2d}^{(3)}$ for any $d\ge5$ because 
\begin{align*}
   (2^{d})(8,2^{d-4})(6,4,2^{d-5})\succ(6,2^{d-3})(6,2^{d-3})(6,2^{d-3}). 
\end{align*}
One can check this (and all similar inequalities below) using the cancellation Lemma \ref{lem:arquimedean} to obtain $(8,4)\succ(6,6)$, which is clear.
\medskip

Now let $\gamma_d=(0,-3,-5,-6,-7,-9,-9,\dots,-9)\in\R^{\pi(d)}$. Observe that $\gamma_d\notin T_{2d}^{(3)}$ for any $d\ge5$ because $y_1-3y_3+y_5+y_6\ge0$ on $T_{2d}^{(3)}$ but $0-3(-5)+(-7)+(-9)=-1$.
\smallskip

We claim that $\gamma_d$ belongs to $T_{2d}^{(2)}$ for all $d\ge5$. We do this by considering all possible inequalities $y_i-2y_j+y_k\ge0$ ($1\le i<j<k\le 6$) that are valid on $T_{2d}^{(2)}$, and checking that $\gamma_d$ satisfies them (note it is enough to consider $i<j<k$ instead of $i\le j\le k$ because $\gamma_d\in T_{2d}^{(1)}$ as the coordinates of $\gamma_d$ are non-increasing). This would be enough because $(\gamma_d)_r=-9$ for $6\le r\le\pi(d)$, and $\gamma_i+\gamma_k \geq 2(-9)=-18$ for all $1 \leq i,k  \leq \pi(d)$.

Hence we consider at most $\binom52$ inequalities $y_i-2y_j+y_k\ge0$, one for each pair $(i,j)$ with $1\le i<j\le 5$, because, as the coordinates of $\gamma_d$ are non-increasing, we can always choose the largest possible $k\le 6$ (with $j<k$) giving a valid inequality for the pair $(i,j)$. For example for the pair $(i,j)=(1,2)$ we have $y_1-2y_2+y_3\ge0$ but also $y_1-2y_2+y_4\ge0$ (using Corollary \ref{cor:Tgood} one can see $y_1-2y_2+y_k\ge0$ is not valid on $T_{2d}^{(2)}$ for $k>4$ because $(2^d)\nu\succeq(4,2^{d-2})(4,2^{d-2})$ is tight for $\nu=(4^2,2^{d-4})$, i.e. $\nu$ cannot be replaced by a smaller partition), so we just test $y_1-2y_2+y_4\ge0$ for $\gamma_d$ and it holds. We can further reduce the number of inequalities to test as follows. For instance if $i=1$ then we only need to test $j=2$ as already for $j\ge3$ we have $(\gamma_d)_j\le-5$ so $y_1-2y_j-9\ge 0-2(-5)-9>0$. With this idea, we find that if $i=2$ then it is enough to consider $j\le4$. We list the remaining inequalities to consider:
\begin{align*}
    i=2,\, j=3&:\quad y_2-2y_3+y_5\ge0,\quad (4,2^{d-2})(8,2^{d-4})\succ(6,2^{d-3})(6,2^{d-3}), \\ 
    i=2,\, j=4&:\quad y_2-2y_4+y_6\ge0,\quad (4,2^{d-2})(6,4,2^{d-5})\succ(4^2,2^{d-4})(4^2,2^{d-4}),\\
    i=3,\, j=4&:\quad \text{no valid inequalities,}\\
    i=3,\, j=5&:\quad y_3-2y_5+y_6\ge0,\quad (6,2^{d-3})(6,4,2^{d-5})\succ(8,2^{d-4})(8,2^{d-4}),\\
    i=4,\, j=5&:\quad \text{no valid inequalities,}
\end{align*}
and all these are satisfied by $\gamma_d$. It only remains to analyze the cases $(i,j)=(3,4)$ and $(i,j)=(4,5)$. In the first case we have no valid inequalities because $y_3-2y_4+y_k\ge0$ does not even hold for $k=5$ since $(6,2^{d-3})(8,2^{d-4})\not\succeq(4^2,2^{d-4})(4^2,2^{d-4})$. For the second case observe that $y_4-2y_5+y_k\ge0$ does not even hold for $k=6$ as $(4^2,2^{d-4})(6,4,2^{d-5})\not\succeq(8,2^{d-4})(8,2^{d-4})$. 

We conclude that $\gamma_d\in T_{2d}^{(2)}$, therefore $T_{2d}^{(2)}\supsetneq T_{2d}^{(3)}$, and so $\tau_d\ge3$.
\end{proof}

\begin{question}
Does the sequence $T_{2d}^{(k)}$ stabilize once two consecutive terms are equal?
\end{question}

\subsection{The tropicalization of $\mathcal{B}\Sigma_{2d}^*$ and the superdominance order} \label{sec:trop BSigma_2d}

In this subsection, we nearly describe the set $\trop(\B\Sigma_{2d}^*)$ in terms of the superdominance order. First, we prove the technical Lemma \ref{le:trop of sos}, which bounds the tropicalization of a \emph{spectrahedral cone}. Second, we discuss how partial symmetry reduction (as in \cite[Section~3.2]{acevedo2024power}) can be used to give a spectrahedral representation of $\B\Sigma_{2d}^*$ that works well for tropicalizing the cone. Third, we combine partial symmetry reduction with the lemma to prove that $\trop(\B\Sigma_{2d}^*)$ is a rational polyhedron (Proposition \ref{prop:sos rational polyhedron}), the inclusions $T_{2d}^{(1)} \supseteq \trop(\B\Sigma_{2d}^*)\supseteq T_{2d}^{(2)}$ (Proposition \ref{prop:tropsos*supsetT2}) and strict containment of $\trop(\B\P_{2d}^*)$ in $\trop(\B\Sigma_{2d}^*)$ for all $2d \geq 10$ (Theorem \ref{thm:tropsdiffer}).

\medskip

Recall that we work with $\{p_\lambda \, \, \mid \, \, \lambda \in \Lambda_{2d}^\varepsilon\}$ as a linear basis of $H_{n,2d}^{\B_n}$. Observe that $\B\Sigma_{2d}$ contains the nonnegative orthant because every even power sum is a sum of squares, implying that its dual cone is contained in the nonnegative orthant, i.e., $\B\Sigma_{2d}^*\subset \R_{\geq 0}^{\pi(d)}$. \medskip

In the following proposition, we observe that the coordinates of points in $\B\Sigma_{2d}^*$ are partially ordered, corresponding to the superdominance order.

\begin{proposition}\label{prop: sd sos*}
    Let $\varphi\in\B\Sigma_{2d}^*$ and $\la,\mu\vDash2d$. If $\la\succeq\mu$ then $\varphi_\la\ge\varphi_\mu$. In particular, $T_{2d}^{(1)}\supseteq\trop(\B\Sigma_{2d}^*)$.
\end{proposition}
\begin{proof}
    Let $\{\rho_\nu\}_{\nu\vDash2d}$ be a dual basis for $\{\p_\nu\}_{\nu\vDash2d}$ and let $\lambda \succeq \mu$. By Corollary \ref{cor:sd->sos} then $\p_\la-\p_\mu\in\B\Sigma_{2d}$ and so $\varphi(\p_\la-\p_\mu)\ge0$. Hence
    \begin{align*}
        \varphi(\p_\la-\p_\mu)=\sum_{\nu\vDash2d}\varphi_\nu\rho_\nu(\p_\la-\p_\mu)=\varphi_\la-\varphi_\mu\ge0.
    \end{align*}
    We have $\B\Sigma_{2d}^*\subset\R_{\ge0}^{\pi(d)}$, and $\log$ is an increasing function, so $\log\varphi_\la\ge\log\varphi_\mu$ and $T_{2d}^{(1)}\supseteq\trop(\B\Sigma_{2d}^*)$.
\end{proof}

The sets $\B\Sigma_{n,2d}^*$ have spectrahedral representations that can be computed using symmetry reduction and a symmetry adapted basis $\mathcal{F}=\{f_{(\la,\mu),1},\ldots,f_{(\la,\mu),m_{(\la,\mu)}} \, \, \mid \, \, (\la,\mu) \vDash n\}$ of $H_{n,d}$ (see e.g. \cite[Sec.~4]{blekherman2021symmetric}). We present a brief explanation in the Appendix \ref{sec:appendix}. \medskip

Let $\sym_{\B_n} : H_{n,d} \to H_{n,d}^{\B}, f \mapsto \frac{1}{\mid \B_n \mid }\sum_{\sigma \in \B_n} \sigma \cdot f$ denote the Reynolds operator of the group $\B_n$, and define the matrix polynomial
\[
\mathbf{N}_n = (\sym_{\B_n}(f_{(\lambda,\mu),i} f_{(\lambda',\mu'),j}))_{(\lambda,\mu),i,(\lambda',\mu'),j},
\]
obtained by the symmetrization of pairwise products of elements in $\mathcal{F}$. Then, $f \in \B\Sigma_{n,2d}$ if and only if there exists a positive semidefinite matrix $A$ such that $f = \operatorname{Tr}(A \cdot \mathbf{N}_n)$ (see \cite[Cor.~4.4]{blekherman2021symmetric}). Thus, the set $\B\Sigma_{n,2d}$ is the \emph{psd projection} with respect to the matrix $\mathbf{N}_n$ (recall Definition \ref{def:psdprojection}). Moreover, the matrix $\mathbf{N}_n$ can be block-diagonalized because the symmetrization of the product of two basis elements from two non-isomorphic irreducible representations is zero. 
 \medskip

For the hyperoctahedral group (analogously to the symmetric group), one can uniformly choose the matrices $\mathbf{N}_n$. This is possible because there exists a symmetry-adapted basis $\mathcal{F}$ of $H_{d,d}$, which can be extended to $H_{n,d}$ by replacing any power sum in $\mathcal{F}$ with the identified power sum in $n \geq d$ variables (see Corollary \ref{cor:symbasis} for details). Then, for $n \to \infty$ any coefficient of the matrix $\mathbf{N}_n$ converges so that we can speak about the limit $\mathbf{N}$ of the matrices. The set $\B\Sigma_{2d}$ is then the psd projection of $\mathbf{N}$. \medskip

The following lemma bounds the tropicalization of a spectrahedral cone contained in the nonnegative orthant. It will be a key ingredient for computing the tropicalization of the cones $\B\Sigma_{2d}^*$. 

Let $L = (L_{ij}(\underline{X}))_{ij}$ an $N\times N$ symmetric matrix whose entries are linear forms $$L_{ij}(\underline{X}) =L_{ji}(\underline{X}) = \sum_{k=1}^s a_{ijk} X_k, \quad a_{ijk}\in\R,$$ and let $\ell_{ij}^+ := \{ k : a_{ijk} > 0\}, \ell_{ij}^- := \{ k : a_{ijk} < 0\}$ and $\ell_{ij} := \{ k : a_{ijk} \neq 0\}$. We use the convention $\max \{ a \in \emptyset\} := - \infty$. 

\begin{lemma} \label{le:trop of sos} 
Let $K := \{ x \in \R_{\geq 0}^s : L(x) \succeq 0\}$ and $T \subset \R^s$ be the set of points $x \in \R^s$ satisfying
\smallskip
\begin{itemize}
    \item[{(1)}] $\max \{x_k : k \in \ell_{ii}^+\} \geq \max \{ x_k : k \in \ell_{ii}^-\}$, for each $1 \leq i \leq N$, and
    \medskip
    \item[{(2)}] $\max \{x_k : k \in \ell_{ii}^+\} + \max \{x_k : k \in \ell_{jj}^+\} \geq 2 \max \{ x_k : k \in \ell_{ij}\}$, for each $ 1 \leq i < j \leq N$.
\end{itemize}
\smallskip
If for all $v \in \inti (T)$ all the inequalities in \emph{(1)} and \emph{(2)} are strict then $\clo \left( \inti (T) \right) \subset \trop \left( K \right) \subset T$.
\end{lemma}
\begin{proof}
The conditions $(1)$ and $(2)$ follow for every element $v \in \trop \left( K \right)$ from the positive semidefiniteness of the $1 \times 1$ and $2 \times 2$ principal minors of the defining matrix $L$. Thus, we have $\trop \left( K \right) \subset T$. \\
Next, we prove $\clo \left( \inti (T) \right) \subset \trop (K) $. It suffices to show that any $v \in \inti(T)$ is contained in $\trop(K)$, since $\trop (K)$ is a closed set (\cite[Proposition~2.2]{alessandrini2013logarithmic}). Let $v \in \inti (T)$. We show that $t^v := (t^{v_1},\ldots,t^{v_s}) \in K$ for all $t>0$ sufficiently large. By assumption, all the inequalities in $(1)$ and $(2)$ are strict at $v$. We consider the diagonal entries of $L$ at $t^v$. For $1 \leq i \leq N$ let $q_i \in [n]$ be an index at which the maximum over all $v_j$ with $j \in \ell_{ii}^+$ is attained. Further be $m_i$ be the multiplicity of $v_{q_i}$ in $v$. Then $v_k - v_{q_i} < 0$ for all $k$ with $a_{iik} < 0$. Thus, for all $k \in \ell_{ii}^-$ we have $t^{v_k-v_{q_i}} \to 0$ for $t \to \infty$. Hence, for sufficiently large $t$ 
\begin{align*}
    \frac{L_{ii}(t^v)}{t^{v_{q_i}}} & = \sum_{k\in \ell_{ii}}  a_{iik}t^{v_k-v_{q_i}} = \sum_{k \in \ell_{ii}^+} a_{iik}t^{v_k-v_{q_i}} \geq m_ia_{iiq_i} > 0.
\end{align*} 
Now let $k \geq 2$ and consider the Leibniz formula of the leading $k \times k$ principal minor of $L(t^v)$, 
\begin{align*}
 \sum_{\sigma \in \mathcal{S}_k} \mbox{sgn} (\sigma) \prod_{i=1}^k L_{i\sigma(i)}(t^v).
\end{align*} 
We treat the product of linear forms in $t^v$ as univariate exponential polynomials in $t$ and claim 
\begin{align}\label{eq:1}
    \deg_t \left(\prod_{i=1}^k L_{ii}(t^v)\right) > \deg_t \left( \prod_{i=1}^k L_{i\sigma(i)}(t^v) \right)
\end{align}
for any $\sigma \in \mathcal{S}_k \setminus \{ \mbox{id}\}$. Equivalently to (\ref{eq:1}), since $L$ is symmetric \[\deg_t \left(\prod_{i=1}^k L_{ii}^2(t^v)\right) > \deg_t \left( \prod_{i=1}^k L_{i\sigma(i)}(t^v)L_{\sigma(i)i}(t^v) \right).\] The leading coefficient of the univariate exponential polynomial $\prod_{i=1}^k L_{ii}^2(t^v)$ equals a product of positive coefficients of each $L_{ii}$ by assumption (1). Combining (1) and (2) we obtain at $t^v$ 
\begin{align*}
    \deg_t \left( L_{ii}L_{\sigma(i)\sigma(i)} \right) &= \deg_t \left( L_{ii} \right) + \deg_t \left( L_{\sigma(i)\sigma(i)} \right) \\
   &  > 2 \deg_t \left( L_{i\sigma(i)} \right)  =\deg_t \left(  L_{i\sigma(i)}L_{\sigma(i)i} \right).
\end{align*}
Therefore, $ \sum_{\sigma \in \mathcal{S}_k} \mbox{sgn} (\sigma) \prod_{i=1}^k L_{i\sigma(i)}(t^v) > 0$ for all sufficiently large $t$.
Analogously, for any $k \times k$ principal minor of $L(t^v)$ the product of the diagonal entries has a degree larger than the product of any other generalized diagonal obtained from a permutation of the indices. Thus, for all $t$ sufficiently large we have $t^v \in K$. 
\end{proof}

    In general, the set $T$ in Lemma \ref{le:trop of sos} is not necessarily convex, since the inequalities over the max do not need to split into a finite sum of linear inequalities. However, $T$ is a polyhedral fan, i.e., a polyhedral complex in which every polyhedron is a cone from the origin {\cite{alessandrini2013logarithmic}}. \medskip

    \begin{example}\label{ex:trop(S6)}
    We present an example where the assumption of all inequalities being strict in Lemma \ref{le:trop of sos} does not hold and the tropicalization of the spectrahedron differs indeed from the set $T$. The set $\B\Sigma_6$ equals the psd projection with respect to the matrices
\begin{align} \label{eq:matrices}
\left(\begin{array}{cc}\p_{(2^3)} &  \p_{(4,2)} \\ \p_{(4,2)}& \p_{(6)}   \end{array}  \right), ~ 
   \left(\p_{(2^3)} -3\p_{(4,2)} +2\p_{(6)}\right), ~ 
      \left( \p_{(4,2)}-\p_{(6)} \right)
\end{align}
(see Example \ref{ex:partialsym Bsos6}).
For the spectrahedron $K=\B\Sigma_6^*$ represented by (\ref{eq:matrices}) we apply Lemma \ref{le:trop of sos}:
\begin{align}
\label{eq:strict} T & = \{ 
 (y_{(2^3)},y_{(4,2)},y_{(6)}) \in \R^3  \, \, \mid \, \,  y_{(2^3)}+y_{(6)} \geq 2y_{(4,2)}, ~
       \max \{y_{(2^3)}, y_{(6)}\} \geq y_{(4,2)},~
       y_{(4,2)} \geq y_{(6)}  \} \\
 \nonumber      & = \{ 
 (y_{(2^3)},y_{(4,2)},y_{(6)}) \in \R^3  \, \, \mid \, \,  y_{(2^3)}+y_{(6)} \geq 2y_{(4,2)}, ~
       y_{(4,2)} \geq y_{(6)}  \}
\end{align}
The set $T$ is a full-dimensional polyhedron, and so $T=\clo(\inti T)$. All three inequalities in (\ref{eq:strict}), i.e. the inequalities (1) and (2) from Lemma \ref{le:trop of sos}, are strict for any $v \in \inti T$ and thus $\trop(\B\Sigma_6^*)=T$.
Using the conjugate transpose of the $2\times 2$ matrix in (\ref{eq:matrices}) by $\left( \begin{array}{cc}
    1     & 0 \\
     1    & -1
    \end{array} \right) $ we have that
$\B\Sigma_6$ is also the psd projection of the matrices 
    $$\left( \begin{array}{cc}
    \p_{(2^3)} &  \p_{(2^3)}-\p_{(4,2)} \\ \p_{(2^3)}-\p_{(4,2)}
      & \p_{(2^3)}+\p_{(6)}-2\p_{(4,2)}   
    \end{array} \right),  ~ 
   \left(\p_{(2^3)} -3\p_{(4,2)} +2\p_{(6)}\right), ~ 
      \left( \p_{(4,2)}-\p_{(6)} \right).$$
We compute the set $T$ in Lemma \ref{le:trop of sos} with this alternative description of $K=\B\Sigma_6^*$. Then $T$ is the set of all $(y_{(2^3)},y_{(4,2)},y_{(6)})\in\R^3$ such that
\begin{align*}
    \max\{y_{(2^3)},y_{(6)}\}\ge y_{(4,2)},~
    y_{(2^3)}+\max\{y_{(2^3)},y_{(6)}\}\ge 2\max\{y_{(2^3)},y_{(4,2)}\},~
    y_{(4,2)}\ge y_{(6)}.
\end{align*}
We deduce $y_{(2^3)}\ge y_{(6)}$ and $T = \{(y_{(2^3)},y_{(4,2)},y_{(6)}) \in \R^3 : y_{(2^3)} \ge y_{(4,2)} \ge y_{(6)} \}.$ The three inequalities of the forms (1) and (2) are
\begin{align*}
    y_{(2^3)}\ge y_{(4,2)},~   y_{(2^3)}\ge \max\{x_{(2^3)},y_{(4,2)}\},~
    y_{(4,2)}\ge y_{(6)}.
\end{align*}
 However, the first inequality implies that the second cannot be strict for any point in $T$. So Lemma \ref{le:trop of sos} cannot be applied, and indeed, we have $\trop  \left(\B\Sigma_6^* \right)\subsetneq T =\clo(\inti T)$.
\end{example}

    Instead of considering full symmetry reduction, we apply the idea of partial symmetry reduction \cite[Section~3.2]{acevedo2024power}. Together with the superdominance order this allows us to prove that $\trop(\B\Sigma_{2d}^*)$ is always a rational polyhedron. \medskip
    
The idea of partial symmetry reduction is as follows. For $n \geq d$, up to replacing power sums by those in another number of variables, the same symmetry adapted basis of the $\B_n$-modules $H_{n,d}$ can be used (Corollary~ \ref{cor:symbasis}). Any polynomial in this symmetry-adapted basis can be written as a sum of terms of the form $X^\a p_\la$, where $\la$ is an even partition. The form $p_\la$ is in $n$ variables, $\a \in \N^n$, and $\alpha_{d+1}=\ldots=\alpha_n=0$, but $\a$ is not necessarily in $(2\N)^n$. Therefore, products of any two such terms from symmetry basis elements remain of the same form. We have 
\[\sym_{\B_n}(X^\a p_\la)=\sym_{\B_n}(X^\a)p_\la= \begin{cases}
   M_\a p_\la, & \text{if $\a \in (2\N)^n$}, \\
   0, & \text{else},
\end{cases}\] where $M_\alpha$ is the corresponding term-normalized monomial symmetric polynomial. To \emph{denormalize} $M_\alpha$ we instead rescale the terms as $n^{\ell(\a)/2}X^\alpha p_\la$, where $\ell(\alpha)$ denotes the size of the support of $X^\alpha$. We then express the monomial symmetric polynomial $m_\a$ in the power sum basis. When expressing $p_\la$ in the monomial basis we get precisely $\la$ and partitions that are refined by $\la$ which one observes by expanding $p_\la=(X_1^{\la_1}+\dots)\cdots(X_1^{\la_l}+\dots)$. Therefore $\la$ is the maximum in the superdominance order among the partitions appearing in $p_\la$. This implies the same happens for monomial symmetric functions expressed in power sums. Then when expressing $m_{\a+\beta}p_{\la\mu}$ in the power sum basis, the maximum partition in the superdominance order that occurs is the partition $(\a+\beta)\la\mu$ which we verify in Corollary \ref{cor:refinement}. \medskip

We verify that the entries of the partial symmetrization matrix are well-defined at the limit when $n\to\infty$. For terms $X^\alpha p_\la$ and $X^{\a'}p_{\la'}$ we define $\varepsilon_{(\a,\la),(\a',\la')}$ as $1$, if $\a+\a' \in (2\N)^n$, and otherwise as $0$. For $n$ variables this matrix is 
\begin{align*}
\mathbf{M}_n&:=(\varepsilon_{(\a,\la),(\a',\la')} \cdot n^{(\ell(\a)+\ell(\a'))/2}M_{\a+\a'}p_{\la\la'})_{(\a,\la),(\a',\la')}\\
&\sim(\varepsilon_{(\a,\la),(\a',\la')} \cdot n^{(\ell(\a)+\ell(\a'))/2}\frac{c_\a c_{\a'}}{n^{\ell(\a+\a')}}m_{\a+\a'}p_{\la\la'})_{(\a,\la),(\a',\la')}
\end{align*}
since $M_\a\sim\frac{c_\a}{n^{\ell(\a)}}m_\a$ for a constant $c_\a>0$ depending only on $\a$. But $\frac12(\ell(\a)+\ell(\a'))\le\ell(\a+\a')$ with equality if and only if $\a$ and $\a'$ have the same support, so we have \[\lim_{n\to\infty}\sym_{\B_n}(n^{(\ell(\alpha)+\ell(\alpha'))/2}X^{\a+\a'}p_\la p_{\la'})=0\] otherwise. Moreover, as observed above, the symmetrization $\sym_{\B_n}(n^{(\ell(\alpha)+\ell(\alpha'))/2}X^{\a+\a'}p_\la p_{\la'})$ is zero also when $\a+\a'\notin2\N$.

\begin{example}\label{ex:partialsym Bsos6}
    We compute $\B\Sigma_6$ with partial symmetry reduction. 
    
    Fix a number of variables $n$. We construct a row vector $\v$ for each set of terms of the form $X_1^{\a_1}X_2^{\a_2}X_3^{\a_3}p_\la$ where the $X^\a$ have the same support, each $\alpha_i$ has fixed parity, $\la$ is an even partition (possibly empty), and $|\a|+|\la|=3$. We rescale each row vector $\v$ by $n^{\ell(\a)/2}$, compute $\sym_{\B_n}(\v^\top\v)$ and take the limit when $n$ goes to infinity.
    \begin{alignat*}{3}
        &\v_1&&=n^{1/2}(X_1p_2,X_1^3) &&\Rightarrow\sym_{\B_n}(\v_1^\top\v_1)=\begin{bmatrix}
            p_{(2^3)} & p_{(4,2)}\\
            p_{(4,2)} & p_{(6)}
        \end{bmatrix}\to\begin{bmatrix}
            \p_{(2^3)} & \p_{(4,2)}\\
            \p_{(4,2)} & \p_{(6)}
        \end{bmatrix}\\
        &\v_2&&=nX_1X_2^2 &&\Rightarrow\sym_{\B_n}(\v_2^\top\v_2)=\frac{n^2}{2\binom{n}2}m_{(4,2)}\to\m_{(4,2)}=\p_{(4,2)}-\p_{(6)}\\
        &\v_3&&=n^{3/2}X_1X_2X_3 &&\Rightarrow\sym_{\B_n}(\v_3^\top\v_3)=\frac{n^3}{\binom{n}3}m_{(2^3)}\to3!\m_{(2^3)}=\p_{(2^3)}-3\p_{(4,2)}+2\p_{(6)}
    \end{alignat*}

and so $\B\Sigma_6=\Pi\left(\begin{bmatrix}
            \p_{(2^3)} & \p_{(4,2)}\\
            \p_{(4,2)} & \p_{(6)}
        \end{bmatrix},\,\,\,\p_{(4,2)}-\p_{(6)},\,\,\,\p_{(2^3)}-3\p_{(4,2)}+2\p_{(6)}\right)$.     
\end{example}

We take the following result from \cite[Equation~(3.5)]{egecioglu1991brick} (see also \cite[Proposition~4.13]{zabrocki2015introduction}), and use it to prove the polyhedrality of $\trop(\B\Sigma_{2d}^*)$. 
\begin{lemma}\label{lem:refinement}
    Let $\la\vdash k$. Then $m_\la=\sum_{\mu\vdash k}E_{\la\mu}p_\mu$ where $E_{\la\mu}$ is a constant depending only on $\la$ and $\mu$. Moreover, $E_{\la\mu}\ne0$ if and only if $\mu=\la$ or $\mu$ is obtained by successively merging parts of $\la$, and the sign of $E_{\la\mu}$ is $(-1)^{\ell(\la)-\ell(\mu)}$.
\end{lemma}

For $\la\vDash2d$ let $A_\la^+:=\{\mu\vDash2d:E_{\la\mu}>0\}$, $A_\la^-:=\{\mu\vDash2d:E_{\la\mu}<0\}$, and $A_\la:=\{\mu\vDash2d:E_{\la\mu}\ne0\}$. 

\begin{corollary}\label{cor:refinement}
    With respect to the superdominance order, the largest partition in $A_\la^+$ and in $A_\la$ is $\la$, and the largest partition in $A_\la^-$ is $\la^*$.
\end{corollary}
\begin{proof}
By Lemma \ref{lem:refinement}, we have $E_{\lambda\lambda} > 0$. Since all other partitions $\mu \in A_\lambda$ are obtained by successively merging parts of $\lambda$, we have $\lambda \succ \mu$, making $\lambda$ the largest partition in $A_\lambda$ and hence in $A_\lambda^+$.\smallskip

Among the partitions $\mu\in A_\la^-$ of length $\ell(\la)-1$, the partition $\la^*$ is the largest, because these come from merging exactly two parts of $\la$. Partitions of smaller length in $A_\la^-$ come from merging partitions of length at most $\ell(\la)-1$, and so they are also smaller than $\la^*$. 
\end{proof}

\begin{proposition} \label{prop:sos rational polyhedron}
The set $\trop(\B\Sigma_{2d}^*)$ is a rational polyhedron.
\end{proposition}

\begin{proof}
    To apply Lemma \ref{le:trop of sos} we first show that the set $\trop(\B\Sigma_{2d}^*)$ is full-dimensional in $\R^{\pi(d)}$. Note that $\B\P_{2d}^*$ is a full-dimensional convex cone contained in $\R_{\geq 0}^{\pi(d)}$ which has Hadamard property and contains the all-ones vector. Thus, by \cite[Lemma~2.2]{blekherman2022tropicalization} we have $\trop(\B\P_{2d}^*) = \clo ( \conv (\log (\B\P_{2d}^*)))$. In particular, the logarithmic image of a full-dimensional ball in $\R_{>0}^{\pi(d)}$ lies in $\trop(\B\P_{2d}^*) \subset \trop(\B\Sigma_{2d}^*)$ which shows the full-dimensionality. 
    
    We use partial symmetry reduction and the representation of $\B\Sigma_{2d}$ as psd projection with respect to the matrix $\mathbf{M}=\lim_{n \to \infty} \mathbf{M}_n$. By the discussion above all entries of the matrix $\mathbf{M}$ are $0$ or of the form $c_\alpha c_{\a'} \mathfrak{m}_{\a+\a'}\p_\la \p_{\la'} $, for some positive constants $c_\a,c_{\a'}$, and the monomial symmetric function $\mathfrak{m}_{\a+\a'}$. We express each coefficient of $\mathbf{M}$ in the power sum basis and apply Lemma \ref{lem:refinement}. We obtain that the max-inequalities in Lemma \ref{le:trop of sos} become linear inequalities of the form $x_i\ge x_j$ or $x_i+x_j\ge2x_k$ by Corollary \ref{cor:refinement}. This guarantees that $T$ is a full-dimensional polyhedron and so all inequalities are strict for the interior of $T$. Moreover, since $T=\clo (\inti T)$ we have $\trop(\B\Sigma_{2d}^*)=T$. 
\end{proof}

Based on the superdominance order and the polyhedra $T_{2d}^{(k)}$, we can show a set-theoretic containment relationship for $\trop(\B\Sigma_{2d}^*)$.

\begin{proposition}\label{prop:tropsos*supsetT2}
    We have $T_{2d}^{(1)} \supseteq \trop(\B\Sigma_{2d}^*)\supseteq T_{2d}^{(2)}$.
\end{proposition}
\begin{proof}
The inclusion $T_{2d}^{(1)} \supseteq \trop(\B\Sigma_{2d}^*)$ was already proven in Proposition \ref{prop: sd sos*}. 
We show that type (1) inequalities $y_i\ge y_j$ and type (2) inequalities $y_i+y_j\ge2y_k$ from Lemma \ref{le:trop of sos} make sense in the superdominance order. The type (1) inequality from the diagonal entry $m_{\a}p_{\la}$ is $y_\mu\ge y_{\mu^*}$ where $\mu=\alpha\la$, and the type (2) inequality from a $2\times2$ block is $y_{(2\a)\la\la}+y_{(2\a')\la'\la'}\ge2y_{(\a+\a')\la\la'}$ which is an inequality of $T_{2d}^{(2)}$ because by Lemma \ref{lem:arquimedean} $$(2\a)(2\a')\la\la\la'\la'\succeq(\a+\a')(\a+\a')\la\la\la'\la'\iff(2\a)(2\a')\succeq(\a+\a')(\a+\a')$$ which follows by fusing all the $(2\a_i,2\a'_i)\succeq(\a_i+\a'_i,\a_i+\a'_i)$ by Lemma \ref{lem:fusion}.  
\end{proof}

Combining the previous Proposition with Proposition \ref{prop:T3subsetT2} and Theorem \ref{thm:tropmom} we show that tropicalizations of the limit cones are not equal for $2d \geq 10$.

\begin{theorem}\label{thm:tropsdiffer}
    $\trop(\B\Sigma_{2d}^*)\supsetneq\trop(\B\P_{2d}^*)$ for all $d\ge5$.
\end{theorem}
\begin{proof}
By Proposition \ref{prop:tropsos*supsetT2} we have $\trop(\B\Sigma_{2d}^*)\supseteq T_{2d}^{(2)}$ for all $d\ge3$, and by Proposition \ref{prop:T3subsetT2} we have $T_{2d}^{(2)}\supsetneq T_{2d}^{(3)}$ for all $d\ge5$. By Theorem \ref{thm:tropmom} we have $T_{2d}^{(3)}\supseteq\trop(\B\P_{2d}^*)$, so for all $d\ge5$,
\begin{align*}
    \trop(\B\Sigma_{2d}^*)\supseteq T_{2d}^{(2)}\supsetneq T_{2d}^{(3)}\supseteq\trop(\B\P_{2d}^*).
\end{align*}
\end{proof}

\begin{remark}\label{rem:T1 contains sos}
    By Proposition \ref{prop: sd sos*} we have $T_{2d}^{(1)}\supseteq\trop(\B\Sigma_{2d}^*)$. The inclusion is strict for $d\ge3$. It can be seen for $d=3$ in Example \ref{ex:trop(S6)}, and for $d\ge4$ it follows from the fact that 
    $y_{(2^d)}-2y_{(4,2^{d-2})}+y_{(4^2,2^{d-4})}\ge0$ is valid on $\trop(\B\Sigma_{2d}^*)$, but not on $T_{2d}^{(1)}$ because $(2^d)\succ(4,2^{d-2})\succ(4^2,2^{d-4})$. The validity on $\trop(\B\Sigma_{2d}^*)$ can be justified by considering the $2\times 2$ submatrices $\sym_{\B_n}(v^\top v)$ where $v=(p_2^k,p_2^{k-2}p_4)$ if $d=2k$, and $v=(\sqrt{n}x_1p_2^k,\sqrt{n}x_1p_2^{k-2}p_4)$ if $d=2k+1$ of the partial symmetrization matrices. 
\end{remark}

We have seen that $\trop(\B\Sigma_{2d}^*)=T_{2d}^{(2)}$ for $d=3,4,5$. From Remark \ref{rem:T1 contains sos} and Proposition \ref{prop:tropsos*supsetT2} we have $T_{2d}^{(1)}\supsetneq\trop(\B\Sigma_{2d}^*)\supseteq T_{2d}^{(2)}$ for all $d\ge3$. 

\begin{conjecture}\label{conj:tropsos}    
$\trop(\B\Sigma_{2d}^*)=T_{2d}^{(2)}$ for all $d\ge3$.
\end{conjecture}

\begin{remark} \label{rmk:Hadamard}
$\B\Sigma_6^*$ has the Hadamard property. From Example \ref{ex:trop(S6)} it is enough to see that $a\ge b\ge c\ge0$, $x\ge y\ge z\ge0$, $a+2c\ge3b$ and $x+2z\ge3y$ imply $ax+2cz\ge3by$. This can be seen from $$ax\ge(3b-2c)(3y-2z)=9by-6bz-6cy+4cz\ge3by-2cz$$ and this last inequality is equivalent to $(b-c)(y-z)\ge0$.

With the description of $\S\Sigma_4$ in Subsection \ref{sec:symquartics} it is straightforward to show that $\S\Sigma_4^*$ has the Hadamard property. With the descriptions of $\B\Sigma_8$ and $\B\Sigma_{10}$ in Appendix \ref{sec:appendix} it might be possible to show that their duals have the Hadamard property. We do not know if $\B\Sigma_{2d}^*$ has the Hadamard property for any $d\ge4$.
\end{remark}

\subsection{Examples} \label{sec:examples}
We present the tropicalizations of $\B\P_{2d}^*$ and $\B\Sigma_{2d}^*$ for degrees $4,6,8$ and $10$. Strict inclusion between the tropicalizations of the cones occurs for the first time in degree $10$, despite we have $\B\Sigma_6 \subsetneq \B\P_6$ and $\B\Sigma_8 \subsetneq \B\P_8$ by Theorem \ref{thm:Limitcones not equal}. We use a linear inequality valid on $\trop(\B\P_{10}^*)$ but not on $\trop( \B\Sigma_{10}^*) $ to provide an example of a form in $\B\P_{10}\setminus \B\Sigma_{10}$ (Proposition \ref{prop:deg 10 nonnegative but not sos}).  \smallskip

We consider systems of linear inequalities in $2,3,5$ and $7$ variables whose coordinates are indexed by the even partitions of $4,6,8$ and $10$,  
\begin{align*}
    \mathcal{L}_1 &:= \left\{ y_{(2^2)} \geq y_{(4)}\right\}, \\
    \mathcal{L}_2 &:= \left\{ y_{(2^3)}+y_{(6)} \geq 2y_{(4,2)},\quad y_{(4,2)} \geq y_{(6)} \right\},\\
    \mathcal{L}_3 &:= \left. \begin{cases}  y_{(2^4)}+y_{(4^2)} \geq 2 y_{(4,2^2)},  
     \;\quad y_{(4,2^2)}+y_{(8)} \geq 2 y_{(6,2)}, \\ 
     y_{(6,2)} \geq y_{(4^2)},  
   \quad \quad \quad \quad\quad\, y_{(4^2)} \geq y_{(8)}.
     \end{cases} \right\},\\
     \mathcal{L}_4 &:= \left. \begin{cases}    y_{(6,2^2)} \geq y_{(4^2,2)}, \quad \quad\quad\;\;\, \quad y_{(8,2)} \geq y_{(6,4)},  \quad \quad \quad \quad \quad \;\;  y_{(6,4)} \geq y_{(10)}, \\
    y_{(6,2^2)}+y_{(10)} \geq 2y_{(8,2)},  \quad \;\,\, y_{(4^2,2)} + y_{(10)} \geq 2y_{(6,4)},  \quad \;\;\, y_{(4^2,2)} \geq y_{(8,2)}, \\
    y_{(4,2^3)} + y_{(6,4)} \geq 2y_{(4^2,2)},  \quad y_{(4,2^3)} + y_{(8,2)} \geq 2y_{(6,2^2)}, \quad y_{(2^5)} + y_{(4^2,2)} \geq 2 y_{(4,2^3)} \end{cases} \right\},
\end{align*}
and write $\mathcal{L}_i(y)$ if a real vector $y$ satisfies all inequalities in $\mathcal{L}_i$.
     
\begin{proposition}\label{prop:trop sos sextics}
We have
\begin{align*}
    \trop(\B\Sigma_4^*) & = \{ y\in \R^{\pi(2)} : \mathcal{L}_1(y) \},  \\
    \trop(\B\Sigma_6^*) & = \{ y\in \R^{\pi(3)} : \mathcal{L}_2(y) \},  \\
    \trop(\B\Sigma_{8}^*) & = \{ 
    y\in \R^{\pi(4)} : \mathcal{L}_3(y) \}, \\
    \trop(\B\Sigma_{10}^*) & = \{ 
    y\in \R^{\pi(5)} : \mathcal{L}_4(y) \},
\end{align*}
and the inequalities in $\mathcal{L}_i$ are facet-defining.
\end{proposition}
\begin{proof}
We already presented the $H$-representation of $\trop(\B\Sigma_6^*) $ in Example \ref{ex:trop(S6)}. 
We determine $\B\Sigma_8^*$ and $\B\Sigma_{10}^*$ using partial symmetry reduction in Examples \ref{ex:sos8} and \ref{ex:sos10}. From there we calculate the tropicalizations via Lemma
\ref{le:trop of sos} and used SAGE \cite{sagemath} to obtain $H$-representations. The claim about quartics follows from Proposition \ref{prop:trop dual to psd}, since $\B\P_4 = \B\Sigma_4$. 
\end{proof}

To compare the polyhedra $\trop(\B\P_{2d}^*)$ and $\trop(\B\Sigma_{2d}^*)$, for $2d \in \{6,8\}$, we can apply Proposition \ref{prop:trop dual to psd} to find a representation of $\trop (\B\P_{2d}^*)$. We calculate a vertex representation of the polyhedral cone $\trop(\mathcal{N}_d)$ and apply the map $\tilde{\nu}_d$. Then we determine the tropical convex hull of this image and compare this set with the cones $\trop (\B\Sigma_{2d}^*)$ (Proposition \ref{prop:trop sos sextics}). The computations were done using SAGE \cite{sagemath} and we found equality for these degrees. \medskip

Using Sage, we find that the convex polyhedral cones $\trop(\B\P_{10}^*)$ and $\trop(\B\Sigma_{10}^*)$ differ in the following sense. Every linear inequality in the $H$-representation of $\trop(\B\Sigma_{10}^*)$ is also in the $H$-representation of $\trop(\B\P_{10}^*)$, but there exists precisely one additional facet defining linear inequality of $\trop(\B\P_{10}^*)$:
\begin{align*} 
  y_{(2^5)}+y_{(6,4)}+y_{(8,2)} \geq 3y_{(6,2^2)}.  
\end{align*}  
We can use this linear inequality to produce an example of a limit symmetric function of degree $10$ that is nonnegative but not a sum of squares. A similar result was obtained for an analogous inequality for symmetric means and degree $6$ in \cite[Sec.~6.1]{acevedo2024power}.

\begin{lemma}\label{lem:decicsnonnegative}
Let $a_1,a_2,a_3 \in \R_{>0}$ such that $a_1a_2a_3 = 1$. Then the even symmetric form 
$$ a_1 p_{(2^5)} +a_2p_{(6,4)}+a_3 p_{(8,2)}-3p_{(6,2^2)} $$
is nonnegative in any number of variables.
\end{lemma}
\begin{proof}
We fix a number of variables. By Lyapunov's inequality (Remark \ref{rem:classical ieqs}) we have $p_6^3 \leq p_8p_6p_4$, so $$p_{(6,2^2)} = \sqrt[3]{p_6^3p_2^6} \leq \sqrt[3]{p_8 p_6 p_4 p_2^6}$$ and by the arithmetic-geometric mean inequality 
 $$\sqrt[3]{p_8 p_6 p_4 p_2^6} = \sqrt[3]{(a_1p_{(2^5)})(a_2p_{(6,4)})(a_3p_{(8,2)})} \leq  \frac{ a_1 p_{(2^5)} +a_2p_{(6,4)}+a_3 p_{(8,2)} }{3}.$$
 Therefore $a_1 p_{(2^5)} +a_2 p_{(6,4)}+a_3 p_{(8,2)}-3p_{(6,2^2)}\ge0$ is valid for any number of variables.
\end{proof}

\begin{proposition} \label{prop:deg 10 nonnegative but not sos}
The even symmetric form $  \frac{1}{18}p_{(2^5)}+3 p_{(8,2)} + 6 p_{(6,4)}-3p_{(6,2^2)}$
is nonnegative for all number of variables but not a sum of squares for any number of variables $n \geq N$ for some fixed $N \in \mathbb{N}$.
\end{proposition}

\begin{proof}
The nonnegativity of the limit form follows from Lemma \ref{lem:decicsnonnegative}. We consider a dual basis to $(\p_{(2^5)},  \p_{(4,2^3)}, \p_{(6,2^2)},  \p_{(4^2,2)}, \p_{(8,2)},  \p_{(6,4)},  \p_{(10)})$. 
In Example \ref{ex:sos10} we find a representation of $\B\Sigma_{10}$ as the psd projection with respect to six matrices. Then the dual cone $\B\Sigma_{10}^*$ is the spectrahedron containing the points in $\R^{\pi(d)}$ for which the matrices are positive semidefinite. We find that $\mathbf{a} = (450228,75326,24986,12656,8325,4159,2803) \in \B\Sigma_{10}^*$ and $\mathbf{a}^T(1/18,0,-3,0,3,6,0)=-49/3$. Thus the claimed form cannot be a sum of squares in any number of variables.
\end{proof}

\printbibliography[heading = subbibliography]

\newpage
\appendix

\section{Appendix} \label{sec:appendix}

\subsection{Symmetric sums of squares and symmetry reduction}

In general, using representation theory, one can describe the invariant sums of squares forms more efficiently \cite{gatermann2004symmetry}. We refer to \cite[Section~4]{blekherman2021symmetric} and \cite[Section~3]{debus2023reflection} for background on the description of the cone of (even) symmetric sums of squares in an increasing number of variables. \smallskip

In characteristic zero we have a correspondence between the irreducible $\mathcal{S}_n$-modules and the partitions $\lambda \vdash n$. For $\mathcal{B}_n$ the irreducible modules are in correspondence with bipartitions $(\lambda,\mu)$ of $n$, i.e. pairs of partitions of nonnegative integers $k$ and $n-k$. We also write $(\lambda,\mu)\vdash n$ for bipartitions $(\lambda,\mu)$ of $n$. For both groups, the irreducible representations are called \textit{Specht modules} and we denote them by $\mathbb{S}^\lambda$, resp. $\mathbb{S}^{(\lambda,\mu)}$ (see e.g. \cite{sagan2001symmetric,musili1993representations} for background on representation theory of $\S_n$ and $\B_n$). 
Asymptotically, the isotypic decomposition of the modules $H_{n,d}$ stabilizes in a certain sense for $n \geq 2d$. For any partition $\alpha$ let $\alpha+ke_1$ denote the partition obtained from $\alpha$ by adding $k$ to its first part, where $k$ is a positive integer.

\begin{proposition}[\cite{riener2013exploiting,debus2023reflection}] \label{prop:repr stability}
Let $\la\vdash2d$, and $(\mu,\nu) \vdash d$, then
\begin{itemize}
    \item[{(i)}] the multiplicity of $\mathbb{S}^\la$ in $H_{2d,d}$ equals the multiplicity of $\mathbb{S}^{\la+ke_1}$ in $H_{2d+k,d}$, and all $\S_{2d+k}$-irreducible representations in $H_{2d+k,d}$ are of this form,
    \medskip
    
    \item[{(ii)}] the multiplicity of $\mathbb{S}^{(\mu,\nu)}$ in $H_{d,d}$ equals the multiplicity of $\mathbb{S}^{(\mu+ke_1,\nu)}$ in $H_{d+k,d}$ and all $\B_{d+k}$-irreducible representations in $H_{d+k, d}$ are of this form. 
\end{itemize}
\end{proposition}

Hence, for each positive integer $d$ there exist positive integers $d_1,\dots,d_k$ corresponding to the multiplicities of the distinct $\mathbb{S}^\la$ in $H_{2d,d}$. And similarly for each positive integer $d$ there exist positive integers $d_1',\dots,d_l'$ corresponding to the multiplicities of the distinct $\mathbb{S}^{(\mu,\nu)}$ in $H_{d,d}$.

\begin{proposition} [\cite{blekherman2021symmetric}, Cor.~4.4]\label{prop:repr inv sos}
For any $n \geq 2d$ there exist symmetric matrices $A_n^{(i)}\in \left( {H_{n,2d}^{\mathcal{S}}}\right)^{d_i \times d_i}$, $i=1,\dots,k$, such that any $f\in \Sigma^{\mathcal{S}}_{n,2d}$ can be written as \[ f = \sum_{i=1}^k \Tr (A_n^{(i)}B^{(i)}) \] for some positive semidefinite matrices $B^{(i)}\in\R^{d_i\times d_i}$. 
\end{proposition}
The same is true verbatim for $\mathcal{B}_n$ and $n \geq d$. The matrices $A_n^{(i)}$ are obtained in the following way. One finds a symmetry-adapted basis $\mathcal{F}$ of $H_{n,d}$ and the matrices contain the symmetrization of pairwise products of elements belonging to the same isotypic component of $\mathcal{F}$.

\subsection{A uniform symmetry adapted basis for partial symmetry reduction}
In view of partial symmetry reduction as in \cite{acevedo2024power} we are interested in a symmetry-adapted basis of $H_{n,d}$ where any two polynomials belonging to the same isotypic component have the same support and the size of the support is at most $d$ for all $n\geq d$. There we discard all factors of the basis elements that are power sums, since they vary in their number of variables. We prove the existence of such a symmetry-adapted basis with this property in Lemma \ref{lem:modifiedHigherSpecht}. Moreover, we deduce that the same symmetry-adapted basis can be used for the $\B_n$-modules $H_{n,d}$ in Corollary \ref{cor:symbasis} and only the power sums are varying in the number of variables. As a consequence, we see that the matrices $A_n^{(i)}$ in Proposition \ref{prop:repr inv sos} can be chosen uniformly and converge asymptotically. \medskip

The \emph{coinvariant algebra} $\R[\underline{X}]_{\B_n}$ of the hyperoctahedral group $\B_n$ is defined as the quotient ring of the real polynomial ring modulo the ideal generated by all invariant polynomials of positive degrees, i.e., $\R[\underline{X}]_{\B_n}=\R[\underline{X}]/(p_2,\ldots,p_{2n})$. The ideal $(p_2,\ldots,p_{2n})$ is closed under the action of $\B_n$ and therefore $\R[\underline{X}]_{\B_n}$ is itself a $\B_n$-module which is isomorphic to the regular representation of the hyperoctahedral group as $\B_n$-modules \cite[Section~8.1]{bergeron2009algebraic}. Moreover, we have an isomorphism of graded $\B_n$-modules $\R[\underline{X}]\cong \R[p_2,\ldots,p_{2n}]\otimes_{\R} \R[\underline{X}]_{\B_n}$. Therefore a symmetry-adapted basis of the coinvariant algebra lifts to a symmetry-adapted basis of the polynomial ring by multiplication with products of even power sums.

\begin{lemma}\label{lem:modifiedHigherSpecht}
There exists a symmetry-adapted basis of the coinvariant algebra $\R[\underline{X}]_{\mathcal{B}_n}$ for which the support of any two polynomials belonging to the same isotypic component is equal. Moreover, the support equals the support of a Specht polynomial belonging to the same isotypic component.
\end{lemma}
\begin{proof}
It was proven by Ariki, Terasoma and Yamada \cite{ariki1997higher} that there exists a symmetry-adapted basis of $\R[\underline{X}]_{\B_n}$ where the part with respect to the isotypic component of $(\lambda,\mu)$ is the set of all higher Specht polynomials $F_\mathbf{T}^\mathbf{S}$ where $\mathbf{S}$ ranges over the set of standard bitableaux of shape $(\lambda,\mu)$ and $\mathbf{T}$ is a fixed bitableau of same shape. Then, $F_\mathbf{T}^\mathbf{T}$ is actually equal to the Specht polynomial $\spe_\mathbf{T}$ up to multiplication by some scalar in $\R\setminus \{0\}$. Let $\operatorname{supp}(\spe_\mathbf{T})$ denote the support of $\spe_\mathbf{T}$. Without loss of generality we can suppose $\operatorname{supp}(\spe_\mathbf{T})=\{X_1,\ldots,X_m\}$. For any standard bitableau $\mathbf{S}$ of same shape as $\mathbf{T}$ we have
\[ F_\mathbf{T}^\mathbf{S} =  \sum_{\alpha \in \N_0^n} c_\alpha x^\alpha = \sum_{\alpha \in \N_0^n} c_\alpha \prod_{i =1}^mX_i^{\alpha_i}\prod_{m+1}^nX_i^{\alpha_i} = \sum_{\beta \in \N_0^m} f_\beta (\underline{X})\prod_{i=1}^mX_i^{\beta_{i}}, \]
where the $ f_\beta $ are some polynomials with $\operatorname{supp}(f_\beta) \subset \{X_{m+1},\ldots,X_n\}$, for all $\beta \in \N_0^m$. 
Since $\spe_\mathbf{T} \mapsto F_\mathbf{T}^\mathbf{S}$ defines a $\B_n$-equivariant linear isomorphism, we observe that permuting any of the variables, which do not lie in $\operatorname{supp}(\spe_\mathbf{T})$, does not change $F_\mathbf{T}^\mathbf{S}$. Analogously, the change of sign of any of these variables does not change $\spe_\mathbf{T}$ and therefore also $F_\mathbf{T}^\mathbf{S}.$ Therefore, each polynomial $f_\beta$ must be an invariant polynomial with respect to the reflection action of the hyperoctahedral group $\B_{n-m}$ acting on $\R[X_{m+1},\ldots,X_n]$ by permutation and change of signs. \\
Then for any $\beta \in \N_0^m$ we have $f_\beta(\underline{X}) = g_\beta (p_2(X_{m+1},\ldots,X_n),\ldots,p_{2(n-m)}(X_{m+1},\ldots,X_n))$ for some $(n-m)$-variate polynomials $g_\beta$.
Since 
\[p_{2k}(X_{m+1},\ldots,X_n) \equiv -p_{2k}(X_1,\ldots,X_m) \mod (p_2(\underline{X}),\ldots,p_{2n}(\underline{X}))\]
we have 
\[ F_\mathbf{T}^\mathbf{S} \equiv \sum_{\beta \in \N_0^m} g_\beta(-p_2(X_1,\ldots,X_m),\ldots,-p_{2(n-m)}(X_1,\ldots,X_m)) \prod_{i=1}^m X_i^{\beta_{i}}=: \widetilde{F}_\mathbf{T}^\mathbf{S} \mod (p_2,\ldots,p_{2n}).\]
The polynomial $\widetilde{F}_\mathbf{T}^\mathbf{S}$ has the same support as $\spe_\mathbf{T}$ and can be used instead of $F_\mathbf{T}^\mathbf{S}$ in the symmetry adapted basis which proves the Lemma.
\end{proof}

\begin{corollary} \label{cor:symbasis} 
There exists a symmetry adapted basis $\mathcal{F}=\{ f_{{(\lambda,\mu)},i_{(\lambda,\mu)}} \, \, \mid \, \, \text{$(\lambda,\mu) \vdash d$}, i_{(\lambda,\mu)} =1,\ldots,m_{(\lambda,\mu)}\}$ of the $\B_d$-module $H_{d,d}$ such that any $f_{{(\lambda,\mu)},i_{(\lambda,\mu)}}$ is a sum of terms of the form $x_1^{\alpha_1}\cdots x_d^{\alpha_d} \cdot p_\kappa$ for exponents $\alpha_1,\ldots,\alpha_d \in \mathbb{Z}_{\geq 0}$ and even partitions $\kappa$, and the union of the supports of the monomials $x^\alpha$ in any term $x^\alpha p_\kappa$ of $f_{{(\lambda,\mu)},i_{(\lambda,\mu)}}$ equals the support of a fixed Specht polynomial of shape ${(\lambda,\mu)}$. Moreover, if one replaces the $d$-variate even symmetric polynomials by $n$-variate even symmetric polynomials in any term $x^\alpha p_\kappa$ of $f_{{(\lambda,\mu)},i_{(\lambda,\mu)}}$, then $\mathcal{F}$ is a symmetry adapted basis of $H_{n,d}$ in any number of variables $n \geq d$.
\end{corollary}
\begin{proof}
By Lemma \ref{lem:modifiedHigherSpecht} there exists a symmetry adapted basis of the direct sum of the graded components of degree at most $d$ of the coinvariant algebra $\R[\underline{X}]_{\B_d}$ which consists only of polynomials with same support as a fixed corresponding $\B_d$-Specht polynomial $F_{\mathbf{T}_{(\la,\mu)}}^{\mathbf{T}_{(\la,\mu)}}$ for each $(\la,\mu) \vdash d$.  
By multiplying the polynomials with even power sums such that the obtained degree equals $d$, we obtain a symmetry-adapted basis \[\mathcal{F}=\{ f_{{(\lambda,\mu)},i_{(\lambda,\mu)}} \, \, \mid \, \, {(\lambda,\mu)} \vdash n, i_{(\lambda,\mu)} =1,\ldots,m_{(\lambda,\mu)}\}\] of $H_{d,d}$. 
Moreover, we can suppose that $F_{\mathbf{T}_{(\la,\mu)}}^{\mathbf{T}_{(\la,\mu)}} \mapsto f_{{(\lambda,\mu)},i_{(\lambda,\mu)}}$ defines an isomorphism of $\B_d$-modules.
All terms of basis elements in $\mathcal{F}$ are of the form $X_1^{\alpha_1}\cdots X_d^{\alpha_d} \cdot p_\kappa$ for exponents $\alpha_1,\ldots,\alpha_d \in \mathbb{Z}_{\geq 0}$ and even partitions $\kappa$. For $n \geq d$ let $f^{(n)}_{{(\lambda,\mu)},i_{(\lambda,\mu)}}$ be the polynomial where in each term $X_1^{\alpha_1}\cdots X_d^{\alpha_d} \cdot p_\kappa$ of $f_{{(\lambda,\mu)},i_{(\lambda,\mu)}}$ the power sums in $d$ variables are replaced by power sums in $n$ variables. 
By construction of $\mathcal{F}$, for any $n \geq d$ the maps $F_{\mathbf{T}_{(\la,\mu)}}^{\mathbf{T}_{(\la,\mu)}} \mapsto f_{{(\lambda,\mu)},i_{(\lambda,\mu)}}^{(n)}$ define naturally isomorphisms of $\B_n$-modules.
Since the linear spans of the orbits of pairwise distinct basis elements of $\mathcal{F}$ intersect trivially in $H_{d,d}$, the same holds for $\mathcal{F}^{(n)}$ and $H_{n,d}$.
Together with the representation stability from Proposition \ref{prop:repr stability} we obtain that $\mathcal{F}^{(n)}$ is a symmetry-adapted basis of $H_{n,d}$.
\end{proof}

\begin{remark} \label{rem:symbasis}
An analogous result to Corollary \ref{cor:symbasis} holds for the $\S_n$-modules $H_{n,d}$ and $n \geq 2d$ with almost the same proof.
\end{remark}

\begin{proposition}\label{prop:spectrahedral cones}
    For any degree $2d$ the sets $\S\Sigma_{2d}^*$ and $\B\Sigma_{2d}^*$ are spectrahedral cones. 
\end{proposition}
\begin{proof}
We only prove the claim for $\S\Sigma_{2d}^*$ since the proof for $\B\Sigma_{2d}^*$ follows analogously. Using partial symmetry reduction and Remark \ref{rem:symbasis} we observe that for any $n \geq d$ the set $\Sigma_{n,2d}^{\S}$ is the psd projection with respect to the matrix obtained by symmetrizing pairwise products of all homogeneous degree $d$ terms of the form $X_1^{\a_1}\cdots X_d^{\a_d}p_\la$, where $p_\la$ is in $n$-variables. The product of such terms is again of this form but of degree $2d$. Suppose $X_1^{\a_1}\cdots X_d^{\a_d}p_\la$ is of degree $2d$ and for $1 \leq i \leq 2d$ let $\beta_i$ denote the number of integers $1 \leq j \leq 2d$ with $\a_j = i$ and $|\a|:=\sum_{i=1}^d \a_i$. We have $\sym_{S_n}(X_1^{\a_1}\cdots X_d^{\a_d}p_\la) = \frac{\beta_1!\cdots\beta_{2d}! (n-|\a|)!}{n!} m_\a p_\la \longrightarrow \beta_1!\cdots \beta_{2d}!\mathfrak{m}_\alpha \p_\la$ in $\R^{\pi(d)}$ for $n \longrightarrow \infty$ by identification with respect to the transition maps of the inverse system $H_{n,2d}^{\B}$. The set $\S\Sigma_{2d}$ is then the psd projection with respect to a matrix filled with terms of the form $\beta_1!\cdots \beta_{2d}!\mathfrak{m}_\alpha \p_\la$. Using a base change from monomial symmetric functions to power sum functions we have that $\S\Sigma_{2d}$ is the psd projection with respect to a matrix $\mathbf{M}_{2d}$ whose coefficients are linear combinations of power sum functions. In particular, $\S\Sigma_{2d}^*$ is the spectrahedron defined as the positive semidefinite locus of the matrix obtained from $\mathbf{M}_{2d}$ by replacing the $\p_\la$'s with a dual basis of $\R^{\pi(2d)}$. 
\end{proof}

\subsection{Spectrahedral representations}
We calculate, using partial symmetry reduction, spectrahedral representations of the cones $\B\Sigma_8,\B\Sigma_{10}$ and $\S\Sigma_4$, and the H-representation of the rational polyhedra $\trop(\B\Sigma_{2d}^*)$ for $d=3,4,5$.

\begin{example}\label{ex:sos4}
Partial symmetry reduction to compute $\S\Sigma_4$.
    \begin{alignat*}{3}
        &\v_0&&:=(p_1^2,p_2)&&\Rightarrow M_{\v_0}=\begin{bmatrix}
            \p_{(1^4)} & \p_{(2,1^2)}\\
            \p_{(2,1^2)} & p_{(2^2)}
        \end{bmatrix}\\
        &\v_1&&:=n^{1/2}x_1(p_1,x_1)&&\Rightarrow M_{\v_1}=\begin{bmatrix}
            \p_{(2,1^2)} & \p_{(3,1)}\\
            \p_{(3,1)} & \p_{(4)}
        \end{bmatrix}\\
        &\v_2&&:=n(x_1x_2)&&\Rightarrow M_{\v_2}=2\m_{(2^2)}=\p_{(2^2)}-\p_{(4)}
    \end{alignat*}
$$\Rightarrow\S\Sigma_4=\Pi\left(\begin{bmatrix}
            \p_{(1^4)} & \p_{(2,1^2)}\\
            \p_{(2,1^2)} & p_{(2^2)}
        \end{bmatrix},\begin{bmatrix}
            \p_{(2,1^2)} & \p_{(3,1)}\\
            \p_{(3,1)} & \p_{(4)}
        \end{bmatrix},\p_{(2^2)}-\p_{(4)}\right).$$
\end{example}

\begin{example}\label{ex:sos8}
    Partial symmetry reduction to compute $\B\Sigma_8$ and $\trop(\B\Sigma_8^*)$.

    We follow the same strategy of Example \ref{ex:partialsym Bsos6}. The vector $\v_{(i,j)}$ has entries $x^\a p_\la$ of degree 4 where $\la$ is an even partition (possibly empty), $x^\a$ has support $\{X_1,\dots,X_{i+j}\}$, $\a_1,\dots,\a_i$ are odd and $\a_{i+1},\dots,\a_{i+j}$ are even. Observe $\v_{(i,j)}$ is rescaled with $n^{(i+j)/2}$. Let $M_{\mathbf{u}}:=\lim_{n\to\infty}\sym_{\B_n}(\mathbf{u}^\top\mathbf{u})$. We obtain,
    \begin{alignat*}{3}
        &\v_{(0,0)}&&:=(p_2^2,\, p_4)&&\Rightarrow M_{\v_{(0,0)}}=\begin{bmatrix}
        \p_{(2^4)} & \p_{(4,2^2)} \\
        \p_{(4,2^2)} & \p_{(4^2)}
    \end{bmatrix}\\
        &\v_{(0,1)}&&:=n^{1/2}X_1^2(p_2,\, X_1^2)&&\Rightarrow M_{\v_{(0,1)}}=\begin{bmatrix}
        \p_{(4,2^2)} & \p_{(6,2)}\\
        \p_{(6,2)} & \p_{(8)}
    \end{bmatrix}\\
        &\v_{(2,0)}&&:=nX_1X_2(p_2,\, X_1^2,\, X_2^2)&&\Rightarrow M_{\v_{(2,0)}}=\begin{bmatrix}
        2\m_{(2^2)}\p_{(2^2)} & \m_{(4,2)}\p_{(2)} & \m_{(4,2)}\p_{(2)}\\
        \m_{(4,2)}\p_{(2)} & \m_{(6,2)} & 2\m_{(4^2)}\\
        \m_{(4,2)}\p_{(2)} & 2\m_{(4^2)} & \m_{(6,2)}
    \end{bmatrix}\\
        &\v_{(0,2)}&&:=n(X_1^2X_2^2)&&\Rightarrow M_{\v_{(0,2)}}=2!\m_{(4^2)}=\p_{(4^2)}-\p_{(8)}\\
        &\v_{(2,1)}&&:=n^{3/2}(X_1X_2X_3^2)&&\Rightarrow M_{\v_{(2,1)}}=2!\m_{(4,2^2)}=\p_{(4,2^2)} -2\p_{(6,2)} -\p_{(4^2)} +2\p_{(8)}\\
        &\v_{(4,0)}&&:=n^2(X_1X_2X_3X_4)&&\Rightarrow M_{\v_{(4,0)}}=4!\m_{(2^4)}=\p_{(2^4)} -6\p_{(4,2^2)} +8\p_{(6,2)} +3\p_{(4^2)} -6\p_{(8)}.
    \end{alignat*}
Converting to the power sum basis we obtain
\begin{align*}
    AM_{\v_{(2,0)}}A^\top&=\begin{bmatrix}
        \p_{(2^4)}-\p_{(4,2^2)} & \p_{(4,2^2)}-\p_{(6,2)} & 0\\
        \p_{(4,2^2)}-\p_{(6,2)} & \frac12\p_{(6,2)}+\frac12p_{(4^2)}-\p_{(8)} & 0\\
        0 & 0 & 2\p_{(6,2)}-2\p_{(4^2)}
    \end{bmatrix}
\end{align*}

$\Rightarrow\B\Sigma_8=\Pi(M_{\v_{(0,0)}},M_{\v_{(0,1)}},AM_{\v_{(2,0)}}A^\top,M_{\v_{(0,2)}},M_{\v_{(2,1)}},M_{\v_{(4,0)}})$ where $A=\begin{bmatrix}
    1 & 0 & 0\\
    0 & 1/2 & 1/2\\
    0 & 1 & -1
\end{bmatrix}$.

We compute $\trop(\B\Sigma_8^*)$ using Lemma \ref{le:trop of sos} and Corollary \ref{cor:refinement}, and compute facet defining inequalities using SAGE to obtain the inequalities given in Proposition \ref{prop:trop sos sextics}.
\end{example}

\begin{example}\label{ex:sos10}
    Partial symmetry reduction to compute $\B\Sigma_{10}$ and $\trop(\B\Sigma_{10}^*)$.

    \begin{alignat*}{3}
        &\v_{(1,0)}&&:=n^{1/2}X_1(p_2^2,\, X_1^2p_2,\, p_4,\, X_1^4)&&\Rightarrow M_{\v_{(1,0)}}=\begin{bmatrix}
\p_{(2^5)} & \p_{(4,2^3)} & \p_{(4,2^3)} & \p_{(6,2^2)}\\
\p_{(4,2^3)} & \p_{(6,2^2)} & \p_{(4^2,2)} & \p_{(8,2)}\\
\p_{(4,2^3)} & \p_{(4^2,2)} & \p_{(4^2,2)} & \p_{(6,4)}\\
\p_{(6,2^2)} & \p_{(8,2)} & \p_{(6,4)} & \p_{(10)} 
\end{bmatrix}\\
        &\v_{(1,1)}&&:=nX_1X_2^2(p_2,\, X_2^2,\, X_1^2)&&\Rightarrow M_{\v_{(1,1)}}=\begin{bmatrix}
\m_{(4,2)}\p_{(2^2)} & \m_{(6,2)}\p_{(2)} & 2\m_{(4^2)}\p_{(2)}\\
\m_{(6,2)}\p_{(2)} & \m_{(8,2)} & \m_{(6,4)}\\
2\m_{(4^2)}\p_{(2)} & \m_{(6,4)} & \m_{(6,4)}
\end{bmatrix}\\
        &\v_{(3,0)}&&:=n^{3/2}X_1X_2X_3(p_2,\, X_1^2,\, X_2^2,\, X_3^2)&&\Rightarrow M_{\v_{(3,0)}}=\begin{bmatrix}
6\m_{(2^3)}\p_{(2^2)} & 2\m_{(4,2^2)}\p_{(2)} & 2\m_{(4,2^2)}\p_{(2)} & 2\m_{(4,2^2)}\p_{(2)}\\
2\m_{(4,2^2)}\p_{(2)} & 2\m_{(6,2^2)} & 2\m_{(4^2,2)} & 2\m_{(4^2,2)}\\
2\m_{(4,2^2)}\p_{(2)} & 2\m_{(4^2,2)} & 2\m_{(6,2^2)} & 2\m_{(4^2,2)}\\
2\m_{(4,2^2)}\p_{(2)} & 2\m_{(4^2,2)} & 2\m_{(4^2,2)} & 2\m_{(6,2^2)}
\end{bmatrix}\\
        &\v_{(1,2)}&&:=n^{3/2}(X_1X_2^2X_3^2)&&\Rightarrow M_{\v_{(1,2)}}=\m_{(4^2,2)}\\
        &\v_{(3,1)}&&:=n^2(X_1X_2X_3X_4^2)&&\Rightarrow M_{\v_{(3,1)}}=\m_{(4,2^3)}\\
        &\v_{(5,0)}&&:=n^{5/2}(X_1X_2X_3X_4X_5)&&\Rightarrow M_{\v_{(5,0)}}=\m_{(2^5)}.
    \end{alignat*}

Let $a:=\p_{(2^5)}$, $b:=\p_{(4,2^3)}$, $c:=\p_{(6,2^2)}$, $d:=\p_{(4^2,2)}$, $e:=\p_{(8,2)}$, $f:=\p_{(6,4)}$, $g:=\p_{(10)}$. Converting to the power sum basis we obtain
\begin{align*}
M_{\v_{(1,1)}}&=\begin{bmatrix}
b-c & c-e & d-e\\
c-e & e-g & f-g\\
d-e & f-g & f-g
\end{bmatrix}\\
    AM_{\v_{(3,0)}}A^\top&=\begin{bmatrix}
a-3b+2c & b-2c-d+2e & 0 & 0\\
b-2c-d+2e & \frac13(c+2d-4e-5f+6g) & 0 & 0\\
0 & 0 & \frac32(c-d-e+f) & 0\\
0 & 0 & 0 & 2(c-d-e+f)
\end{bmatrix}\\
2!M_{\v_{(1,2)}}&=d - e - 2f + 2g\\
3!M_{\v_{(3,1)}}&=b - 3c - 3d + 6e + 5f - 6g\\
5!M_{\v_{(5,0)}}&=a - 10b + 20c + 15d - 30e - 20f + 24g
\end{align*}

$\Rightarrow\B\Sigma_{10}=\Pi(M_{\v_{(1,0)}},M_{\v_{(1,1)}},AM_{\v_{(3,0)}}A^\top,M_{\v_{(1,2)}},M_{\v_{(3,1)}},M_{\v_{(5,0)}})$, where $A=\begin{bmatrix}
    1 & 0 & 0 & 0\\
    0 & 1/3 & 1/3 & 1/3\\
    0 & -1 & 1/2 & 1/2\\
    0 & 0 & -1 & 1
\end{bmatrix}$. 

We compute $\trop(\B\Sigma_{10}^*)$ using Lemma \ref{le:trop of sos} and Corollary \ref{cor:refinement}, and compute facet defining inequalities using SAGE to obtain the inequalities given in Proposition \ref{prop:trop sos sextics}.
\end{example}

\end{document}